\nonstopmode
\documentclass[reqno,10pt]{amsart}
\usepackage{latexsym}
\usepackage{fancyhdr}
\usepackage{amsmath, amssymb}
\usepackage[all]{xy}
\usepackage{pdflscape}
\usepackage{longtable}
\usepackage{rotating}
\usepackage{verbatim}
\usepackage{hyperref}
\usepackage{subfigure}
\usepackage{mathrsfs}
\usepackage{cleveref}
\usepackage{mdwlist}
\usepackage{color}

\newcounter{mnotecount}[section]

\newcommand{\rmnote}[1]{}

\DeclareFontFamily{U}{mathb}{\hyphenchar\font45}
\DeclareFontShape{U}{mathb}{m}{n}{
      <5> <6> <7> <8> <9> <10> gen * mathb
      <10.95> mathb10 <12> <14.4> <17.28> <20.74> <24.88> mathb12
      }{}
\DeclareSymbolFont{mathb}{U}{mathb}{m}{n}
\DeclareFontSubstitution{U}{mathb}{m}{n}

\let\dot\relax
\DeclareMathAccent{\dot}{0}{mathb}{"39}
\let\ddot\relax
\DeclareMathAccent{\ddot}{0}{mathb}{"3A}
\let\dddot\relax
\DeclareMathAccent{\dddot}{0}{mathb}{"3B}
\let\ddddot\relax
\DeclareMathAccent{\ddddot}{0}{mathb}{"3C}

\theoremstyle{plain}
\newtheorem*{theorem*}{Theorem}
\newtheorem{theorem}{Theorem}[section]
\newtheorem*{lemma*}{Lemma}
\newtheorem{lemma}[theorem]{Lemma}
\newtheorem*{proposition*}{Proposition}
\newtheorem{proposition}[theorem]{Proposition}
\newtheorem*{corollary*}{Corollary}
\newtheorem{corollary}[theorem]{Corollary}
\newtheorem*{claim*}{Claim}

\newtheorem*{conjecture*}{Conjecture}

\newtheorem*{question*}{Question}
\theoremstyle{definition}
\newtheorem*{definition*}{Definition}
\newtheorem{definition}[theorem]{Definition}
\newtheorem*{example*}{Example}
\newtheorem{example}[theorem]{Example}

\newtheorem*{algorithm*}{Algorithm}
\newtheorem*{remark*}{Remark}
\newtheorem*{remarks*}{Remarks}
\newtheorem{remark}[theorem]{Remark}

\newtheorem*{convention*}{Convention}

\numberwithin{equation}{section}

\allowdisplaybreaks[1]

\newcommand{\al}{\alpha}
\newcommand{\be}{\beta}
\newcommand{\ga}{\gamma}
\newcommand{\de}{\delta}
\newcommand{\ep}{\epsilon}
\newcommand{\ve}{\varepsilon}
\newcommand{\ze}{\zeta}
\newcommand{\et}{\eta}
\newcommand{\vt}{\vartheta}

\newcommand{\la}{\lambda}

\newcommand{\rh}{\rho}

\newcommand{\si}{\sigma}

\newcommand{\ta}{\tau}

\newcommand{\vh}{\varphi}
\newcommand{\ch}{\chi}
\newcommand{\ps}{\psi}
\newcommand{\om}{\omega}

\newcommand{\Ga}{\Gamma}

\newcommand{\Ps}{\Psi}
\newcommand{\Om}{\Omega}

\newcommand{\C}{\mathbb{C}}

\newcommand{\N}{\mathbb{N}}

\newcommand{\R}{\mathbb{R}}

\newcommand{\cB}{\mathcal{B}}

\newcommand{\cD}{\mathcal{D}}
\newcommand{\cE}{\mathcal{E}}

\newcommand{\cG}{\mathcal{G}}

\newcommand{\cJ}{\mathcal{J}}

\newcommand{\cL}{\mathcal{L}}

\newcommand{\cR}{\mathcal{R}}

\newcommand{\fF}{\mathfrak{F}}
\newcommand{\fG}{\mathfrak{G}}

\newcommand{\fL}{\mathfrak{L}}
\newcommand{\fM}{\mathfrak{M}}
\newcommand{\fN}{\mathfrak{N}}

\newcommand{\fS}{\mathfrak{S}}

\newcommand{\fW}{\mathfrak{W}}

\newcommand{\p}{\partial}
\newcommand{\db}{\overline{\p}}

\renewcommand{\Im}{\mathrm{Im}}
\renewcommand{\o}{\circ}

\newcommand{\seq}{\mathbf}

\newcommand{\supp}{\on{supp}}

\newcommand{\on}{\operatorname}

\newcommand{\sr}[1]%
{\ifmmode{}^\dagger\else${}^\dagger$\fi\ifvmode
\vbox to 0pt{\vss
 \hbox to 0pt{\hskip\hsize\hskip1em
 \vbox{\hsize3cm\raggedright\pretolerance10000
 \noindent #1\hfill}\hss}\vss}\else
 \vadjust{\vbox to0pt{\vss%
 \hbox to 0pt{\hskip\hsize\hskip1em%
 \vbox{\hsize3cm\raggedright\pretolerance10000%
 \noindent #1\hfill}\hss}\vss}}\fi%
}

\newcommand{\A}{\;\forall}
\newcommand{\E}{\;\exists}

\newcommand{\ol}{\overline}
\newcommand{\ul}{\underline}

\def\D{\mathcal D}
\def\CC{\mathcal C}
\def\O{\mathcal O}

\def\S{\mathcal S}

\def\O{\mathcal O}

\def\ep{\varepsilon}

\newcommand{\COT}{T^\ast\Omega \setminus \{0\}}

\DeclareMathOperator{\WF}{WF}
\DeclareMathOperator{\real}{Re}
\DeclareMathOperator{\imag}{Im}
\DeclareMathOperator{\singsupp}{sing\, supp}
\DeclareMathOperator{\Char}{Char}

\DeclareMathSymbol{\Mu}{\mathalpha}{operators}{"4D}

\title[Almost analytic extensions and microlocal analysis]
{Almost analytic extensions of ultradifferentiable functions with applications to microlocal analysis}

\author[S.~F\"urd\"os]{Stefan F\"urd\"os}

\address{S.~F\"urd\"os: Department of Mathematics and Statistics, Masaryk University,  
Kotlarska 2, 611 37 Brno, Czech Republic}
\email{furdos@math.muni.cz}

\author[D.N.~Nenning]{David Nicolas Nenning}

\address{D.N.~Nenning: 
Fakult\"at f\"ur Mathematik, Universit\"at Wien, 
Oskar-Morgenstern-Platz~1, A-1090 Wien, Austria}
\email{david.nicolas.nenning@univie.ac.at}

\author[A.~Rainer]{Armin Rainer}

\address{A.~Rainer: 
University of Education Lower Austria,
Campus Baden M\"uhlgasse 67, A-2500 Baden \&
Fakult\"at f\"ur Mathematik, Universit\"at Wien, 
Oskar-Morgenstern-Platz~1, A-1090 Wien, Austria}
\email{armin.rainer@univie.ac.at}

\author[G.~Schindl]{Gerhard Schindl}

\address{G.~Schindl: Fakult\"at f\"ur Mathematik, Universit\"at Wien, 
Oskar-Morgenstern-Platz~1, A-1090 Wien, Austria}
\email{gerhard.schindl@univie.ac.at}

\begin{document}

\begin{abstract}
	We review and extend the description of ultradifferentiable functions by their 
   almost analytic extensions, i.e., extensions to the complex domain with specific vanishing rate 
   of the $\ol \p$-derivative near the real domain. 
   We work in a general uniform framework which comprises the main classical ultradifferentiable classes 
   but also allows to treat unions and intersections of such.
   The second part of the paper is devoted to applications in microlocal analysis. 
   The ultradifferentiable wave front set is defined in this general setting and characterized in  
   terms of almost analytic extensions and of the FBI transform. This allows to extend its definition to 
   ultradifferentiable manifolds. 
   We also discuss ultradifferentiable versions of the elliptic regularity 
   theorem and obtain a general quasianalytic Holmgren uniqueness theorem.
\end{abstract}

\thanks{S.\ F\"urd\"os was supported by GACR grant 17-19437S, D.N.\ Nenning and A.\ Rainer by FWF Project P~26735-N25, and 
G.\ Schindl by FWF Project J 3948-N35.}
\keywords{Almost analytic extensions, characterization of ultradifferentiable classes, stability properties, 
ultradifferentiable wave front set, boundary values, FBI transform, elliptic regularity, uniqueness of distributions}
\subjclass[2010]{26E10, 30D60, 35A18, 46E10, 46F05, 46F20, 58C25}
\date{\today}

\maketitle

\tableofcontents

\section{Introduction}

An almost analytic extension of a real function $f$ is an extension $F$ to the complex domain 
such that $\ol \p F(z)$ has a certain growth rate as $z$ approaches the real domain. 
It is well-known that this growth rate encodes regularity properties of $f$.

In this article we review and extend the characterization of ultradifferentiable function classes 
by their almost analytic extensions. 
The almost analytic description of Denjoy--Carleman classes goes back to Dynkin \cite{Dynkin80}.   
For the non-quasianalytic classes introduced by Beurling \cite{Beurling61} and 
Bj\"orck \cite{Bjoerck66} the characterization was proved by Petzsche and Vogt \cite{PetzscheVogt84}. 

We introduce a uniform approach which generalizes all mentioned results. 
Our characterization theorems work under very weak conditions, in particular, we need not assume 
non-quasianalyticity. This is achieved by refining the extension method of Dynkin 
following the ideas of \cite{RainerSchindl16a,RainerSchindl17}
and combining it with the description of ultradifferentiable classes by weight matrices 
which was introduced in \cite{RainerSchindl12}.     

In the special case of Beurling--Bj\"orck classes we even obtain a complete characterization of the 
classes which admit a description by almost analytic extension: these are precisely the classes 
that are stable by composition.  

In the second part of the paper we apply these results to microlocal analysis.
More precisely, we deal with the ultradifferentiable wave front set.
	The wave front set was introduced in the smooth case by H\"ormander 
	and in the analytic category by Sato as a refinement of the singular support.
In \cite{H_rmander_1971} H\"ormander introduced the ultradifferentiable wave front set with respect to
 Denjoy--Carleman classes given by weight sequences.
In  particular he gave an alternative
 definition of the analytic wave front set by the Fourier transform, in contrast to Sato's
 approach using holomorphic extensions.
 Bony \cite{MR0650834} showed that the definitions of Sato, H\"ormander and the one of Bros--Iagolnitzer
 \cite{MR0399494} 
 using the FBI transform describe the same set.
 	The first author \cite{Fuerdoes1} showed that the theorem of Dynkin can be used to prove a version
 	of Bony's Theorem for the ultradifferentiable wave front set in the case of Denjoy--Carleman classes.
 	
 	On the other hand Albanese--Jornet--Oliaro \cite{Albanese:2010vj} defined the ultradifferentiable wave front
 	 set for Beurling--Bj\"ork classes and proved a microlocal elliptic regularity theorem. Our aim is to unify and generalize these results.

We begin by recalling and extending the definition of the
 ultradifferentiable wave front set to classes given by weight matrices.
 We characterize it in  
terms of almost analytic extensions as well as in terms of the FBI transform.
In the last section of the article we discuss ultradifferentiable versions of the elliptic regularity 
theorem and obtain a general quasianalytic Holmgren uniqueness theorem.

\subsection{Almost analytic extensions}

Let $h : (0,\infty) \to (0,1]$ be an increasing continuous function which tends to $0$ as $t \to 0$. 
Let $\rh>0$. Let $U \subseteq \R^n$ be a bounded open set.
  We say that a function $f : U \to \R$ admits an \emph{$(h,\rh)$-almost analytic extension} if 
  there is a function $F \in C^1_c(\C^{n})$ and a constant $C\ge 1$ such that $F|_U = f$ and 
  \[
      |\ol \p F(z)| \le C\, h ( \rh d(z,\ol U)), \quad \text{ for } z \in \C^n.
  \] 
  Here $d(z,\ol U) := \inf_{x \in \ol U} |x-z|$ denotes the distance of $z$ to $\ol U$.
  A vector valued function $f=(f_1,\ldots,f_m) : U \to \R^m$ admits an $(h,\rh)$-almost analytic extension 
  if each component $f_j$ does. 

We wish to emphasize that functions that admit almost analytic extension have good stability properties: 

\begin{proposition} \label{prop:composition}
  Suppose that $f : U \to \R$ has an $(h,\rh)$-almost analytic extension and 
  $g : V \to U$ has a $(k,\si)$-almost analytic extension.  
  Then $f \o g$ admits a $(\max\{h,k\},\max\{C\rh,\si\})$-almost analytic extension, where the 
  constant $C$ equals the Lipschitz constant of the extension of $g$.  
\end{proposition}

\begin{proof}
  Let $F$ and $G$ denote the respective extensions. 
  Then 
  \[
    \p_{\ol z_i} (F\o G) = \sum_{j=1}^n \p_{z_j} F(G) \p_{\ol z_i} G_j  + \sum_{j=1}^n \p_{\ol z_j} F(G) \p_{\ol z_i} \ol G_j.
  \]
  Since $G \in C^1_c$, we have  
  $d(G(z),G(\ol V)) \le \on{Lip}(G)\, d(z,\ol V)$. 
  The assertion follows.
\end{proof}

Notice that stability under inverse/implicit functions and solving ODEs follows in a similar way; we refer to \cite{Dynkin80}.

Let $\seq M=(M_k)$ be a positive sequence. 
For $\rh>0$ we consider
the Banach space $\cB^{\seq M}_\rh(U) := \{f \in C^\infty (U) : \|f\|^{\seq M}_{\rh}< \infty\}$, where 
\[
  \|f\|^{\seq M}_{\rh} := \sup_{x \in U,\, \al \in \N^n} \frac{|\p^\al f(x)|}{\rh^{|\al|} M_{|\al|}},
\]
and the limits 
\begin{equation*} 
  \cB^{\{\seq M\}}(U) := \on{ind}_{\rh \in \N} \cB^{\seq M}_\rh(U) \quad \text{ and } \quad 
  \cB^{(\seq M)}(U) := \on{proj}_{\rh \in \N} \cB^{\seq M}_{1/\rh}(U).   
\end{equation*}
Then $\cB^{\{\seq M\}}(U)$ and $\cB^{(\seq M)}(U)$ are called \emph{Denjoy--Carleman classes of 
Roumieu and Beurling type}, respectively. 
We shall also need the \emph{local} classes 
\[
	\cE^{[\seq M]}(U) := \on{proj}_{V \Subset U} \cB^{[\seq M]}(V),  
\]
where $V$ ranges over the relatively compact open subsets of $U$; we write $[\seq M]$ if we mean either 
$\{\seq M\}$ or $(\seq M)$.

Let $\seq m = (m_k)$ be the sequence defined by $m_k : = M_k/k!$ and let us assume that $m_k^{1/k} \to \infty$ 
as $k \to \infty$. We define 
\begin{equation} \label{h}
	 h_{\seq m}(t) := \inf_{k \in \N} m_k t^k, \quad \text{ for } t > 0, \quad \text{ and } \quad h_{\seq m}(0):=0.
\end{equation}
The following theorem is due to Dynkin \cite{Dynkin80}.
	
\begin{theorem} \label{Dynkin}
	Assume that $\seq m$ is logarithmically convex, $m_k^{1/k} \to \infty$, and $(M_{k+1}/M_k)^{1/k}$ is bounded. 
	Let $U \subseteq \R^n$ be open. 
	Then $f \in \cE^{\{\seq M\}}(U)$ if and only if for each ball $B \Subset U$ the restriction 
	$f|_B$ has an $(h_{\seq m},\rh)$-almost analytic extension for some $\rh>0$. 	
\end{theorem}

Our goal is to extend this result to the Beurling case and to the classes 
of Beurling and Bj\"orck which were equivalently described by Braun, Meise, and Taylor \cite{BMT90}. 
These classes are defined in terms of a weight function $\om$. 
By a \emph{weight function} we mean a continuous increasing function $\om : [0,\infty) \to [0,\infty)$ with $\om(0) =0$ 
that satisfies
\begin{align}
   & \om(2t) = O(\om(t)) \quad\text{ as } t \to \infty, \label{om1}\\
   & \om(t) = O(t) \quad\text{ as } t \to \infty, \label{om2}\\
   & \log t = o(\om(t)) \quad\text{ as } t \to \infty, \label{om3}\\
   & \vh(t) := \om(e^t) \text{ is convex}.  \label{om4}
\end{align}
Note that \eqref{om3} implies  
$\lim_{t \to \infty} \om(t) = \infty$.

For $\rh >0$ we consider the Banach space $\cB^{\om}_\rh(U) := \{f \in C^\infty (U) : \|f\|^\om_{\rh}< \infty\}$, 
where, for $\vh^*(t) := \sup_{s \ge 0} (s t - \vh(s))$, 
\[
  \|f\|^\om_{\rh} := \sup_{x \in U,\,\al \in \N^n} |\p^\al f(x)| \exp(-\tfrac{1}{\rh} \vh^*(\rh |\al|)),
\]
and the limits 
\begin{equation*}
  \cB^{\{\om\}}(U) := \on{ind}_{\rh \in \N} \cB^{\om}_\rh(U) \quad \text{ and  } \quad 
  \cB^{(\om)}(U) := \on{proj}_{\rh \in \N} \cB^{\om}_{1/\rh}(U).
\end{equation*}
The corresponding \emph{local} classes are defined by 
\[
	\cE^{[\om]}(U) := \on{proj}_{V \Subset U} \cB^{[\om]}(V);  
\]
we write $[\om]$ if we mean either $\{\om\}$ or $(\om)$. 
We recall that $\cE^{[\om]}(U)$ contains non-trivial functions with compact support in $U$ if and only if 
\[
	\int_1^\infty \frac{\om(t)}{t^2} \, dt < \infty;
\]
cf.\ \cite{BMT90} or \cite{RainerSchindl12}.
In that case we say that $\om$ is \emph{non-quasianalytic} and it makes sense to set 
\[
\cD^{[\om]}(U) := \cE^{[\om]}(U) \cap \cD(U),
\]
where $\cD(U)$ denotes the space of smooth functions with compact support in $U$.

In \cite{PetzscheVogt84} the authors prove the following result.

\begin{theorem} \label{PetzscheVogt}
	Let $\om$ be a concave non-quasianalytic weight function. 
	Let $U \subseteq \R$ be open and $f \in \cD(U)$. Then:
	\begin{enumerate}
		\item $f \in \cD^{\{\om\}}(U)$ if and only if there exist $\rh>0$ and $F \in \cD(\widetilde U)$ such that $F|_\R = f$ and 
		\begin{equation} \label{eq:PetzscheVogt}
			\sup_{z \in \C \setminus \R} |\ol \p F(z)| \exp( \rh \om^\star (|y|/\rh)) < \infty.
		\end{equation}
		\item $f \in \cD^{(\om)}(U)$ if and only if for each $\rh>0$ there exists $F \in \cD(\widetilde U)$ such that $F|_\R = f$ and 
		\eqref{eq:PetzscheVogt}.
	\end{enumerate}
	Here $\widetilde U$ is an open subset of $\C$ such that $U = \widetilde U \cap \R$ and 
	$\om^\star(t) = \sup_{s\ge 0 } (\om(s) - s t)$. 
\end{theorem} 

In \cite{PetzscheVogt84}
the almost analytic extensions were obtained by an explicit formula suggested by Mather based on 
  the Fourier transform. That proof does not work for quasianalytic classes. 

\begin{remark}
  In \cite{PetzscheVogt84} the assumption \eqref{om4} is not made. 
  This condition is important for the equivalence of the classes $\cE^{[\om]}$ with the classes originally 
  introduced by Beurling and Bj\"orck using the Fourier transform; cf.\ \cite{BMT90}. 
\end{remark}

We will prove results which generalize both \Cref{Dynkin} and \Cref{PetzscheVogt} and which work also in the quasianalytic setting. 
Our most general results are formulated and proved for ultradifferentiable classes defined by weight matrices; see 
\Cref{thm:Rchar} and \Cref{thm:Bchar}. 
We give full details in the proofs, since Dynkin's papers seem not to be widely known.

For classes described by weight functions we obtain a complete characterization:

\begin{theorem}
	Let $\om$ be a weight function satisfying $\om(t) =o(t)$ as $t\to \infty$. 
	The following are equivalent.
	\begin{enumerate}
		\item $\cE^{\{\om\}}$ can be described by almost analytic extensions.
		\item $\cE^{(\om)}$ can be described by almost analytic extensions.
		\item $\cE^{\{\om\}}$ is stable under composition.
		\item $\cE^{(\om)}$ is stable under composition.
		\item $\om$ is equivalent to a concave weight function.
	\end{enumerate}
\end{theorem}

This follows from the much more comprehensive \Cref{thm:omegachar} 
in which 
also
the precise meaning of the phrase ``$\cE^{[\om]}$ can be described by almost analytic extensions''  
is 
explained.
See also \Cref{PetzscheVogtnew} for our new version of \Cref{PetzscheVogt}.

A widely used family of ultradifferentiable classes which falls into this framework is the scale of \emph{Gevrey classes}
\[
  \cG^s = \cE^{\{(k!^s)_k\}} = \cE^{\{t \mapsto t^{1/s}\}} ,\quad s>1; 
\]
note that $\cG^1= C^\om$.

\subsection{Applications to microlocal analysis}

The uniform approach to ultradifferentiable classes by imposing growth conditions in terms of weight matrices 
provides us with a general framework to treat the ultradifferentiable wave front sets for distributions $u \in \cD'$.  
Our setting comprises and generalizes the wave front sets $\WF_{[\seq M]}$
of H\"ormander \cite{Hoermander83I} for weight sequences $\seq M$ and $\WF_{[\om]}$
of Albanese, Jornet, and Oliaro \cite{Albanese:2010vj} for weight functions $\om$. 

In \Cref{sec:WF} we develop the basic properties trying to impose minimal assumptions on the weights.

As an application of the description of ultradifferentiable classes by almost analytic extensions
we obtain in \Cref{sec:WFreg} a characterization of the ultradifferentiable wave front set by 
almost analytic extensions; see \Cref{RegWFLocalChar2}.
This description allows us to  
show that the ultradifferentiable wave front set is compatible with pullbacks by mappings of the 
corresponding ultradifferentiable class and hence the definition of the wave front set can be extended to 
ultradifferentiable manifolds; see \Cref{thm:pullback}. 
Furthermore, we obtain a general ultradifferentiable version of Bony's theorem, that is a characterization 
of the ultradifferentiable wave front set  not only by almost analytic extensions but also in terms of the FBI transform; see \Cref{Bony}. 

In the particular case of a weight function the latter takes the following form.

\begin{theorem} \label{Bony:intro}
Let $\om$ be a concave weight function satisfying $\om(t) =o(t)$ as $t\to \infty$.
Let $u\in\D^\prime(\Omega)$ and $(x_0,\xi_0)\in T^\ast\Omega \setminus \{0\}$. 
Then 
\begin{enumerate} 
\item $(x_0,\xi_0)\notin\WF_{\{\om\}} u$ if and only if there exist a test function $\psi\in\D(\Omega)$ with $\psi\equiv 1$ 
near $x_0$, a conic neighborhood $U\times\Gamma$ of $(x_0,\xi_0)$, and a constant $\gamma>0$ such that 
\begin{equation}\label{M-FBIestimate:intro}
\sup_{(t,\xi)\in U\times\Gamma}e^{\ga \omega(|\xi|)}
\bigl|\mathfrak{F}(\psi u)(t,\xi) \bigr|<\infty.
\end{equation}
\item $(x_0,\xi_0)\notin\WF_{(\om)} u$ if and only if there exist a test function $\psi\in\D(\Omega)$ with $\psi\equiv 1$ near $x_0$ and a conic neighborhood $U\times\Gamma$ of $(x_0,\xi_0)$ such that
\eqref{M-FBIestimate:intro} is satisfied for all $\gamma>0$.
\end{enumerate}
\end{theorem}

We refer to \Cref{sec:Bony} for the definition of the generalized FBI transform $\fF$.

In the last \Cref{sec:elliptic} we investigate ultradifferentiable versions of the elliptic regularity theorem. 
Our most general result is \Cref{elliptic-regThm} which is formulated for classes defined by weight matrices.
It comprises the versions of  
H\"ormander \cite{Hoermander83II} for weight sequences $\seq M$ and
of Albanese, Jornet, and Oliaro \cite{Albanese:2010vj} for weight functions $\om$ as special cases. 
The proof follows closely the approach of H\"ormander. 
As a corollary we obtain a general version of Holmgren's uniqueness theorem; see \Cref{thm:Holmgren}.

Notice that in the Beurling case we must in general assume that the coefficients of the linear operator 
are \emph{strictly more regular} than the wave front set in question, just as in \cite{Albanese:2010vj}; 
H\"ormander only considers operators with analytic coefficients. 
There are however circumstances when the operator can be as regular as the wave front set
(both in the case of a single weight sequence and of a weight function); see 
\Cref{sec:openproblem}. 
In particular, this occurs in the setting considered in \cite{Albanese:2010vj}, whence our result 
\Cref{cor:strong} actually strengthens \cite[Theorem 4.1]{Albanese:2010vj}.

A further interesting corollary of \Cref{elliptic-regThm} is the following.
We are interested in the intersection of all non-quasianalytic Gevrey classes
\[
  \cE^{(\fG)} := \bigcap_{s>1} \cG^s;
\]
this is a non-quasianalytic function class,  
cf.\ \cite{RainerSchindl12}.

\begin{theorem}\label{Gevrey}
Let $P(x,D) = \sum_{|\al| \le m} a_\al(x) D^\al$ be a linear partial differential operator with $\cE^{(\fG)}(\Omega)$-coefficients.
Then
\begin{equation}
\WF_{(\fG)} u \subseteq \WF_{(\fG)} Pu\cup \Char P
\end{equation}
for all $u\in\D^\prime(\Omega)$. If $P$ is elliptic, then $\WF_{(\fG)}u = \WF_{(\fG)} Pu$.
\end{theorem}

That \Cref{Gevrey} follows from \Cref{elliptic-regThm} will be proved in \Cref{sec:openproblem}.

\begin{remark}
   It is clearly possible to define ultradistributions and their wave front sets 
   based on non-quasianalytic weight matrices (as dual spaces of the 
   respective spaces of ultradifferentiable test functions). 
   For weight sequences and weight functions there exists a comprehensive theory of ultradistributions,
   see e.g.\ \cite{Komatsu73,Komatsu77,Komatsu82}.
   One can expect that results similar to those obtained in this paper hold in that situation. 
   For instance, an elliptic regularity theorem for ultradistributions of Braun--Meise--Taylor type 
   is proved in \cite{AlbaneseJornetOliaro12}. 
   However, it seems that different techniques will be required, since the growth of the Fourier--Laplace 
   transform of compactly supported ultradistributions quite differs from the one of classical distributions 
   (cf.\ \cite{Komatsu77}). 
   In \cite{AlbaneseJornetOliaro12}, for instance, tools from the theory of ultradifferentiable pseudodifferential 
   operators of infinite order are used. These tools are not yet developed in the framework of general 
   weight matrices.       
\end{remark}

\subsection*{Acknowledgment}
We wish to thank the anonymous referee for valuable suggestions that improved the presentation of the paper.

\section{Weights and ultradifferentiable classes} \label{sec:spaces}

\subsection{Weight sequences} \label{weights}

Let $\mu = (\mu_k)$ be a positive increasing sequence,  
$1 = \mu_0 \le \mu_1 \le \mu_2 \le \cdots$.
We associate the sequences $\seq M=(M_k)$ and $\seq m = (m_k)$ defined by 
\begin{equation} \label{def:M}
  \mu_0 \mu_1 \mu_2 \cdots \mu_k = M_k = k!\, m_k, 
\end{equation}
for all $k \in \N$. 
We call $\seq M$ a \emph{weight sequence} if $M_k^{1/k} \to \infty$.    
A weight sequence $\seq M$ is called \emph{non-quasianalytic} if 
\begin{equation}
 \sum_{k} \frac{1}{\mu_k} < \infty. 
\end{equation}
We say that $\seq M$ has \emph{moderate growth} if there exists $C>0$ such that $M_{j+k} \le C^{j+k} M_j M_k$ for all
$j,k \in \N$, or equivalently,
\begin{equation}\label{Char:ModGrowth}
  \mu_k \lesssim M_k^{1/k};
\end{equation}
we refer to \cite[Lemma 2.2]{RainerSchindl16a} for a proof and further equivalent conditions.
(For real valued functions $f$ and $g$ we write $f \lesssim g$ if $f \le C g$ for some positive constant $C$.)

Two weight sequences $\seq M$ and $\seq N$ are said to be \emph{equivalent} if there is a constant $C>0$ such that 
$1/C \le M_k^{1/k}/N_k^{1/k} \le C$ for all $k$.
We write $\seq M \preceq \seq N$ (resp.\ $\seq M \lhd \seq N$) if $M_k^{1/k}/N_k^{1/k}$ is bounded (resp.\ tends to $0$).  

\begin{remark}
Note that $\mu$ uniquely determines $\seq M$ and $\seq m$, and vice versa. In analogy we shall use 
$\nu \leftrightarrow \seq N \leftrightarrow \seq n$, $\si \leftrightarrow \seq S \leftrightarrow \seq s$, etc.
That $\mu$ is increasing means precisely that $\seq M$ is logarithmically convex (\emph{log-convex} for short). 
Log-convexity of $\seq m$ is a stronger condition: if $\seq m$ is log-convex we shall say that $\seq M$ is 
\emph{strongly log-convex}. 
\end{remark}

The results contained in the next lemma are easy to check; the proof is left to the reader.

\begin{lemma}[Properties of weight sequences] \label{lem:basicM}
Let $1 = \mu_0 \le \mu_1 \le \mu_2 \le \cdots$. Then:
  \begin{enumerate}
      \item $M_k^{1/k}$ is increasing, equivalently,
      \begin{equation}\label{mucompare}
        \A k\in \N_{>0} : M_k^{1/k} \le \mu_k.
      \end{equation}
      \item $M_jM_k\le M_{j+k}$ for all $k,j$. 
      \item If $M_k^{1/k} \to \infty$, then $\mu_k \to \infty$.   
      \item If $m_k^{1/k} \to \infty$, then $m_{k}/m_{k-1} = \mu_k/k \to \infty$.
      \item The condition $m_k^{1/k} \to \infty$ implies
\begin{equation} \label{strictInclusion2}
  \A \rh>0 \E C >0 \A k \in \N : k^k \le C \rh^k M_k. 
\end{equation}
  \end{enumerate}  
\end{lemma}


\subsection{Functions associated with weight sequences} \label{hGaSi}

There are a few functions which one naturally associates with a weight sequence; cf.\ 
\cite{Mandelbrojt52}, \cite{Komatsu73}, \cite{ChaumatChollet94}. 

Let $\seq m =(m_k)$ be a positive sequence satisfying $m_0 = 1$ and $m_k^{1/k} \to \infty$. 
We have already introduced the function $h_{\seq m}$ in \eqref{h}.
Furthermore, we need 
\begin{align} 
  \label{counting2}
  \ol \Ga_{\seq m}(t) &:= \min\{k : h_{\seq m}(t) =  m_k t^k\}, \quad t > 0,  
\end{align}  
and, provided that $m_{k+1}/m_{k} \to \infty$,
\begin{align}
\ul \Ga_{\seq m} (t) &:=  \min\Big\{k : \frac{m_{k+1}}{m_k}  \ge \frac{1}{t} \Big\}, \quad t > 0.
\end{align}
The next lemma is immediate from the definitions, cf.\ \cite[Lemma 3.2]{RainerSchindl17}.

\begin{lemma} \label{basic}
Let $\seq m =(m_k)$ be a positive sequence satisfying $m_0 = 1$, $m_k^{1/k} \to \infty$, and $m_{k+1}/m_k \to \infty$. Then: 
\begin{enumerate}
  \item $h_{\seq m}$ is increasing, continuous, and positive for $t>0$. For large $t$ we have $h_{\seq m}(t) = 1$. 
  \item $\ul \Ga_{\seq m}$ is decreasing 
  and $\ul \Ga_{\seq m}(t) \to \infty$ as $t\to 0$.
  \item $k \mapsto m_k t^k$ is decreasing for $k \le \ul \Ga_{\seq m}(t)$. \label{eq:ulGa3}
  \item $\ul \Ga_{\seq m}  
  \le \ol \Ga_{\seq m}$. If $\seq m$ is log-convex then $\ul \Ga_{\seq m} = \ol \Ga_{\seq m}$.
\end{enumerate}  
\end{lemma}

It will be crucial to also have an ``upper bound for $\ol \Ga$ in terms of $\ul \Ga$''. 
The next lemma provides a sufficient condition for this. 

\begin{lemma}[{\cite{RainerSchindl17}}] \label{lem:m1}
  Let $\seq M$ and $\seq N$ be weight sequences satisfying $m_k^{1/k} \to \infty$ and $n_k^{1/k} \to \infty$.
  Assume that 
  \begin{equation} \label{qai}
    \E C\ge 1 \A 1 \le j \le k :  \frac{\mu_j}j \le C \frac{\nu_k}k.
  \end{equation}
  Then, for all $t >0$,
  \begin{equation} \label{eq:compare} 
    \ol \Ga_{\seq n} (Ct) \le \ul \Ga_{\seq m}(t). 
  \end{equation}
\end{lemma}

We also consider the function 
\begin{equation} \label{ep:omegam}
   \om_{\seq m} (t) = - \log h_{\seq m}(1/t) = \sup_{k \in \N}  \log \Big(\frac{t^k}{m_k}\Big), \quad t>0,
\end{equation} 
which is increasing, convex in $\log t$, and zero for sufficiently small $t>0$.
The \emph{log-convex minorant} $\ul {\seq m}$ of $\seq m$ is given by  
\[
  \ul m_k := \sup_{t>0} \frac{t^k}{\exp(\om_{\seq m}(t))}, \quad k \in \N.
\]
In particular, $\seq m$ is log-convex if and only if $\seq m=\ul {\seq m}$.

\subsection{Basic properties of Denjoy--Carleman classes} 

For weight sequences $\seq M$ and $\seq N$ we have $\cB^{[\seq M]} \subseteq \cB^{[\seq N]}$ if and only if 
$\seq M \preceq \seq N$ 
and $\cB^{\{\seq M\}} \subseteq \cB^{(\seq N)}$ if and only if 
$\seq M \lhd \seq N$.
In particular, $\seq M$ and $\seq N$ are equivalent if and only if the corresponding classes $\cB^{[\seq M]}$ and $\cB^{[\seq N]}$ 
coincide.
By the Denjoy--Carleman theorem 
(e.g.\ \cite[Theorem 1.3.8]{Hoermander83I}),
$\cB^{[\seq M]}(U)$ contains non-trivial elements 
with compact support if and only if $\seq M$ is non-quasianalytic.

\subsection{Weight matrices and corresponding spaces of functions}

A \emph{weight matrix} is a  
family $\fM$ of weight sequences which is totally ordered with respect to the pointwise order relation on sequences, i.e., 
\begin{enumerate}
	\item $\fM \subseteq \R^\N$,
	\item each $\seq M \in \fM$ is a weight sequence in the sense of \Cref{weights},
	\item for all $\seq M,\seq N \in \fM$ we have $\seq M \le \seq N$ or $\seq M \ge \seq N$. 
\end{enumerate} 

Let $\fM$ and $\fN$ be two weight matrices. We define 
\begin{align*}
\fM \{\preceq\} \fN \quad &:\Leftrightarrow  \quad  \A \seq M\in \fM \E \seq N \in \fN : \seq M \preceq \seq N,
\\
\fM (\preceq) \fN \quad &:\Leftrightarrow  \quad  \A \seq N\in \fN \E \seq M \in \fM : \seq M \preceq \seq N.
\\
\fM \{\lhd) \fN \quad &:\Leftrightarrow  \quad  \A \seq  M\in \fM \A \seq N \in \fN : \seq M \lhd \seq N.
\end{align*} 
We say that $\fM$ and $\fN$ are \emph{R-equivalent} (resp.\ \emph{B-equivalent}) if $\fM \{\preceq\} \fN \{\preceq\} \fM$ 
(resp.\ $\fM (\preceq) \fN (\preceq) \fM$) and simply \emph{equivalent} if they are both $R$- and $B$-equivalent.

For a weight matrix $\fM$  
we consider the corresponding Roumieu class
\begin{align} \label{def:BRM}
\cB^{\{\fM\}}(U) &:= \on{ind}_{\seq M \in \fM} \cB^{\{\seq M\}}(U),
\end{align}
and Beurling class
\begin{align} \label{def:BBM}
	\cB^{(\fM)}(U) &:= \on{proj}_{\seq M \in \fM} \cB^{(\seq M)}(U).
\end{align} 
For weight matrices $\fM$, $\fN$ we have $\cB^{[\fM]} \subseteq \cB^{[\fN]}$ 
if and only if $\fM [\preceq] \fN$
and 
$\cB^{\{\fM\}} \subseteq \cB^{(\fN)}$ 
if and only if $\fM \{\lhd) \fN$; 
cf.\ \cite{RainerSchindl12}.

The limits in the definitions \eqref{def:BRM} and \eqref{def:BBM} can always be assumed countable 
as is shown in the next lemma. 

\begin{lemma} \label{rem:matrixass}
  Let $\fM$ be a weight matrix. 
  There exists a countable weight matrix $\fL \subseteq \fM$ such that $\cB^{[\fL]}(U) = \cB^{[\fM]}(U)$ 
  algebraically and topologically.
\end{lemma}

\begin{proof}
  Let us prove the Roumieu case.
  For every $k \in \N$ let $\fM_k := \{M_k : \seq M \in \fM\}$ which is a subset of $\R_{>0}$. 

Case 1: If $\ol{\seq  M} := (\sup \fM_k)_k \in \fM$ then $\ol {\seq M} \ge \fM$ and hence 
$\cB^{\{\ol {\seq M}\}}(U) = \cB^{\{\fM\}}(U)$. 

Case 2: Assume $\ol {\seq M} \not\in \fM$ but $\sup \fM_k \in \fM_k$ for all $k$.
For each $k$ there exists $\seq M^k \in \fM$ such that $M^k_k = \sup \fM_k$. 
Then $\fL := \{ \seq M^k : k \in \N\}$ is a countable totally ordered subfamily of $\fM$.
Moreover, $\cB^{[\fL]}(U) = \cB^{[\fM]}(U)$ follows from the claim that 
for each $\seq M \in \fM$ there exists $\seq L \in \fL$ such that $\seq M \le \seq L$. 
Since $\seq M \ne \ol {\seq M}$, there is a $k_0$ such that $M_{k_0} < \sup  \fM_{k_0} = M_{k_0}^{k_0}$. 
Since $\fM$ is totally ordered, $\seq M \le \seq M^{k_0} =: \seq L$ and
the claim is proved. 

Case 3: Assume 
$\sup \fM_{k_0} \not\in \fM_{k_0}$ for some $k_0$.
For each $k$ choose a strictly increasing sequence $M_k^n$ in $\fM_k$ such that 
$M_k^n \to \sup \fM_k$ as $n \to \infty$.
For each $k$ and each $n$ choose $\seq L=\seq L(k,n) \in \fM$ such that $L_k = M_k^n$.
This gives a countable subfamily $\fL \subseteq \fM$. 
By construction, for given $k_0$ we clearly find $\seq L \in \fL$ such that $M_{k_0} < L_{k_0}$ 
which implies  $\cB^{[\fL]}(U) = \cB^{[\fM]}(U)$ as in Case 2. 

The Beurling case is analogous (replacing $\sup$ by $\inf$).
\end{proof}

The corresponding \emph{local} classes are defined by 
\[
	\cE^{[\fM]}(U) := \on{proj}_{V \Subset U} \cB^{[\fM]}(V).  
\]

We say that a weight matrix $\fM$ is \emph{quasianalytic} if each $\seq M \in \fM$ is quasianalytic. 
For a quasianalytic $\fM$ the class $\cB^{[\fM]}(U)$ is quasianalytic in the sense that it cannot contain
non-trivial elements with compact support. 
It is easy to see that in the Roumieu case $\cB^{\{\fM\}}(U)$ also the converse is true. 
In the Beurling case the class $\cB^{(\fM)}(U)$ is quasianalytic if and only if there exists a quasianalytic 
$\seq M \in \fM$; 
this follows from \cite[Proposition 4.7]{Schindl15}. 
In that case we may remove all non-quasianalytic sequences from $\fM$ without altering the class (thanks to the
total order, see (3)).

\begin{definition}[Regular weight matrix] \label{def:Rregular}
  A weight matrix $\fM$ satisfying 
  \begin{enumerate}
    \item[(0)] $m_k^{1/k} \to \infty$ as $k \to \infty$ for all $\seq M \in \fM$ \label{strictInclusion}
  \end{enumerate}
  is called \emph{R-regular} (for \textbf{R}oumieu) if 
  \begin{enumerate}
    \item $\forall  \seq M \in \fM \E \seq N \in \fM \E C\ge 1 \A j \in \N : M_{j+1} \le C^{j+1} N_j$, \label{R-Derivclosed1}
    \item 
    $\forall \seq M \in \fM \E \seq N \in \fM \E C\ge 1 
    \A t>0 : \ol \Ga_{\seq n}(Ct) \le \ul \Ga_{\seq m}(t)$, \label{goodR}
  \end{enumerate}
  and \emph{B-regular} (for \textbf{B}eurling) if
  \begin{enumerate} \setcounter{enumi}{2}
    \item $\forall \seq M \in \fM \E \seq N \in \fM \E C \ge 1 \A j \in \N : N_{j+1} \le C^j M_j$, \label{B-Derivclosed1}
    \item $\forall \seq M \in \fM \E \seq N \in \fM \E C \ge 1 
    \A t>0 : \ol \Ga_{\seq m}(Ct) \le \ul \Ga_{\seq n}(t)$. \label{goodB}
  \end{enumerate} 
  Moreover, $\fM$ is called \emph{regular} if it is both R- and B-regular. 
  We say that a weight matrix $\fM$ is \emph{R-semiregular} (resp.\ \emph{B-semiregular}) 
  if it satisfies (0) and (1) (resp.\ (3)), and $\fM$ is called \emph{semiregular} if 
  it is both R- and B-semiregular.   
  Occasionally, we will also use [semiregular] (or [regular]) and mean that the weight matrix in question is 
  assumed to be R- or B-semiregular (R- or B-regular) depending on the case that is considered.
\end{definition}

Let us discuss the relations among the conditions in this definition.

\begin{remark} \label{rem:regular}
   We have the following equivalences; see \cite[Proposition 4.6]{RainerSchindl12}:
   \begin{itemize}
     \item $C^\om \subseteq \cE^{(\fM)}$ if and only if $\fM$ satisfies (0).
     \item $\cB^{\{\fM\}}$ (equiv.\ $\cE^{\{\fM\}}$) is stable under derivation if and only if $\fM$ satisfies (1).
     \item $\cB^{(\fM)}$ (equiv.\ $\cE^{(\fM)}$) is stable under derivation if and only if $\fM$ satisfies (3). 
   \end{itemize}

   Suppose that $\fM$ is an R-semiregular weight matrix. Then the following three conditions are gradually 
   weaker:
   \begin{enumerate}
     \item $\forall \seq M \in \fM \E \seq N \in \fM \E C\ge 1 \A j\le k  : \frac{\mu_j}{j} \le C  \frac{\nu_k}{k}$
     \item $\fM$ satisfies \Cref{def:Rregular}\eqref{goodR}.
     \item $\forall \seq M \in \fM \E \seq N \in \fM \E C>0 \A j \le k : m_j^{1/j} \le C n_k^{1/k}$ \label{aiR}
   \end{enumerate}  
   Indeed, 
   that (1) implies (2) follows from \Cref{lem:m1}; in \Cref{example} we shall see that (1) is 
   strictly stronger than (2).  
   And that (2) implies (3) follows from \Cref{prop:composition} and \Cref{thm:Rchar}, since 
   (3) holds if and only if the class $\cB^{\{\fM\}}$ (equiv.\ $\cE^{\{\fM\}}$) is stable 
  under composition; cf.\ \cite{RainerSchindl14}. 

  Similarly, if $\fM$ is a B-semiregular weight matrix, then the following  
  three conditions are gradually 
   weaker:
   \begin{enumerate}
     \item[(4)] $\forall \seq M \in \fM \E \seq N \in \fM \E C\ge 1 \A j\le k  : \frac{\nu_j}{j} \le C  \frac{\mu_k}{k}$
     \item[(5)] $\fM$ satisfies \Cref{def:Rregular}\eqref{goodB}.
     \item[(6)] $\forall \seq M \in \fM \E \seq N \in \fM \E C>0 \A j \le k : n_j^{1/j} \le C m_k^{1/k}$ \label{aiB}
   \end{enumerate}
   This follows from \Cref{lem:m1}, \Cref{prop:composition}, \Cref{thm:Bchar}, and since (6)   
   holds if and only if the class $\cB^{(\fM)}$ (equiv.\ $\cE^{(\fM)}$) is stable 
  under composition; cf.\ \cite{RainerSchindl14}. 

  The conditions \Cref{def:Rregular}\eqref{goodR} and \Cref{def:Rregular}\eqref{goodB} 
  are a minimal requirement (aside from semiregularity) 
  for our proofs of \Cref{thm:Rchar} and \Cref{thm:Bchar} to work.

  Additionally, we wish to emphasize that 
  (1) holds if and only if $\fM$ is R-equivalent to a weight matrix which consists of nothing 
  but strongly log-convex 
  weight sequences. 
  In the same way (4) holds if and only if $\fM$ is B-equivalent to a weight matrix which consists of nothing 
  but strongly log-convex 
  weight sequences. See \cite[Corollaries 9 and 10]{Rainer:aa}.
\end{remark}

\begin{example} \label{example}
   There exist two positive sequences $\seq M\le \seq N$ such that:
   \begin{enumerate}
     \item They satisfy \eqref{eq:compare}.
     \item If two sequences $\seq M'$ and $\seq N'$ satisfy \eqref{qai} (with a possibly different constant), 
     then either $\seq M$ is not equivalent to $\seq M'$ or $\seq N$ is not equivalent to $\seq N'$. 
     \item $\mu_k/k \to \infty$, $\nu_k/k \to \infty$, $m_k^{1/k} \to \infty$, and $n_k^{1/k} \to \infty$ as $k \to \infty$. 
   \end{enumerate}

\begin{proof}
  Let $a_j$, $j\ge 1$, be integers satisfying
  \begin{equation*}
    a_1 := 1, \quad a_{j+1} \ge \max\{a_j^2, a_j + 3\} \quad \text{ for all } j \ge 1,  
  \end{equation*}
  and $b_j$, $j \ge 1$, positive numbers such that
  \begin{equation*}
     b_1 := 1, \quad b_{j+1} > b_j \ge j^{a_{j+1}} \quad \text{ for all } j \ge 1. 
   \end{equation*} 
   We define $\mu_0 := 1$ and for $k \ge 1$ 
   \[
    \mu_k := \begin{cases}
      a_j b_j & \text{ if }  a_j \le k <a_{j+1} -1 
      \\
      j^{-a_{j+1} +1} (a_{j+1} -1) b_j  & \text{ if } k = a_{j+1} -1.
    \end{cases}
  \]
  Let $c_j$, $j \ge 1$, be positive numbers such that 
  \begin{equation*}
      c_{j+1} > c_j \ge \frac{b_j}{a_j} \max\Big\{a_{j+2}-2, 
      \Big(j^{a_{j+1} -1} \prod_{\ell = a_j}^{a_{j+1}-2} \ell \Big)^\frac{1}{a_{j+1} -a_j-1}\Big\}.
    \end{equation*} 
  Define $\nu_0:= 1$ and for $k\ge 1$ 
  \[
    \nu_k := \begin{cases}
        a_j c_j & \text{ if }  a_j \le k <a_{j+1} -1 
      \\
      \mu_k  & \text{ if } k = a_{j+1} -1.
    \end{cases}
  \] 

  (1) The various definitions imply that
  \begin{equation*}
    \A j \ge 0 \A k \ge a_{j+1} : b_{j+1}^{k-a_{j+1}+1} \le \frac{\nu_{a_{j+1}}}{a_{j+1}} \cdots \frac{\nu_k}{k}.
  \end{equation*}
  In particular, if $t>0$ is such that $b_{j} = \frac{\mu_{a_j}}{a_j} < \frac{1}t \le b_{j+1} =  \frac{\mu_{a_{j+1}}}{a_{j+1}}$, then 
  \[
    t^{a_{j+1}- k-1} \le \frac{\nu_{a_{j+1}}}{a_{j+1}} \cdots \frac{\nu_k}{k} = \frac{n_{k}}{n_{a_{j+1}-1}}
  \]
  i.e.\ $n_k t^k \ge n_{a_{j+1}-1} t^{a_{j+1}-1}$. Since by construction $\frac{\mu_{k+1}}{k+1} < \frac{1}{t}$ for all $k< a_{j+1}-1$, 
  we have $\ul \Ga_{\seq m} (t) = a_{j+1}-1$. Hence $n_k t^k \ge n_{\ul \Ga_{\seq m} (t)} t^{\ul \Ga_{\seq m} (t)}$ 
  for all $k \ge \ul \Ga_{\seq m}(t)$ and, 
  consequently, $\ol \Ga_{\seq n}(t) \le \ul \Ga_{\seq m}(t)$.

  (2) If $\seq M'$ and $\seq N'$ satisfy \eqref{qai} and $D^{-1} \le ( M'_k/M_k)^{1/k} \le D$
  as well as $D^{-1} \le ( N'_k/N_k)^{1/k} \le D$ for a positive constant $D$, then 
  \[
    \E C,H\ge 1 \A 1 \le j \le k : \frac{\mu_j}{j} \le H C^k \frac{\nu_k}{k}. 
  \]
  Clearly, this property is violated by the constructed sequences (to see this replace $j$ by $a_j$ and $k$ by $a_{j+1} -1$).

  (3) It is easy to see that $\mu_k/k \le \nu_k/k$ for all $k$. 
  That $\mu_k/k \to \infty$ as $k\to \infty$ follows from $b_j \ge j^{a_{j+1}}$. This shows all assertions 
  since $\mu_k/k \to \infty$ implies $m_k^{1/k} \to \infty$;
  cf.\ the arguments given in \cite{RainerSchindl12} before Lemma 2.13.
\end{proof}

 The constructed sequences $\seq M$ and $\seq N$ are not log-convex, 
 but since $m_k^{1/k}$ and $n_k^{1/k}$ tend to $\infty$ as $k \to \infty$,
 we have $\cE^{[\seq M]} = \cE^{[\ul {\seq M}]}$ and 
 $\cE^{[\seq N]} = \cE^{[\ul {\seq N}]}$, where $\ul {\seq M}$ denotes the log-convex minorant of $\seq M$; 
 see \cite[Theorem 2.15]{RainerSchindl12}.
\end{example}

For later use we also show the following.

\begin{theorem} \label{AnalyticClosed}
Let $\fM$ be a weight matrix satisfying $m_k^{1/k} \to \infty$ for all $\seq M \in \fM$.   
If $\vh : \Omega_1\rightarrow \Omega_2$ is a real analytic mapping between open sets
$\Omega_j\subseteq\R^{n_j}$, $j=1,2$,  
then the pullback $\vh^\ast:\cE^{[\fM]}(\Omega_2)\rightarrow\cE^{[\fM]}(\Omega_1)$ of $\vh$ is well defined.
\end{theorem}

\begin{proof}
  Let us first assume that $\fM$ consists of a single weight sequence $\seq M$.
  In the Roumieu case the statement follows easily from the proof of \cite[Proposition 8.4.1]{Hoermander83I}; it is enough 
  that $\seq M$ is a positive sequence with $m_k^{1/k} \to \infty$. 

  Suppose that $u\in\cE^{(\seq M)}(\Omega_2)$ and 
  $K \subseteq \Om_1$ is compact. For each $\rh>0$ there exists $C >0$ such that 
  $L_k:= \max\{k!, \max_{|\al| = k} \sup_{x \in \vh(K)} |\p^\al u(x)|\} \le C \rh^{k} M_{k}$ for all
  $k$. Then the sequence $N_k:= \sqrt{L_k M_k}$ satisfies $\seq L \lhd \seq N \lhd \seq M$ and $n_k^{1/k} \to \infty$.
  So $u$ belongs to $\cB^{\{\seq N\}}(\vh(K))$ and, by the Roumieu case, 
  $\vh^* u \in \cB^{\{\seq N\}}(K) \subseteq \cB^{(\seq M)}(K)$.

  The general case follows immediately. 
\end{proof}

\subsection{Whitney ultrajets} \label{def:ultrajets}

Let $E$ be a compact subset of $\R^n$. 
We denote by $\cJ^\infty(E)$ the vector space of all jets $F= (F^\al)_{\al\in \N^n} \in C^0(E,\R)^{\N^n}$ on $E$. 
For $a \in E$ and $p \in \N$ we associate the Taylor polynomial
\begin{gather*}
  T^p_a : \cJ^\infty(E) \to C^\infty (\R^n,\R), ~ F \mapsto T^p_a F(x) := \sum_{|\al|\le p} \frac{(x-a)^\al}{\al!} F^\al(a),    
\end{gather*}
and the remainder $R^p_a F = ((R^p_a F)^\al)_{|\al| \le p}$ with
\begin{gather*}
  (R^p_a F)^\al (x) := F^\al(x) - \sum_{|\be| \le p-|\al|} \frac{(x-a)^\be}{\be!} F^{\al+\be}(a), \quad a,x \in E.      
\end{gather*}
Let us denote by 
$j^\infty_E$ the mapping which assigns to a $C^\infty$-function $f$ on $\R^n$
the jet $j^\infty_E(f) := (\p^\al f|_E)_\al$. 
By Taylor's formula, $F  =j^\infty_E(f)$ satisfies
\begin{equation*}
  (R^p_a F)^\al (x) = o(|x-a|^{p-|\al|}) \quad \text{ for $a,x \in E$, $p\in \N$, $|\al| \le p$ as $|x-a|\to 0$.}
\end{equation*}
Conversely, if a jet $F \in \cJ^\infty(E)$ has this property, then it admits a 
$C^\infty$-extension to $\R^n$,
by Whitney's extension theorem \cite{Whitney34a} (for modern accounts see e.g.\ 
\cite[Ch.~1]{Malgrange67}, \cite[IV.3]{Tougeron72}, or 
\cite[Theorem 2.3.6]{Hoermander83I}).

Let $\seq M=(M_k)$ be a weight sequence.
For fixed $\rh>0$ we denote by $\cB^{\seq M}_\rh(E)$ the set of all jets $F$ such that  
there exists $C>0$ with 
\begin{gather*}
  |F^\al(a)| \le C \rh^{|\al|} \,  M_{|\al|}, \quad \al \in \N^n,~ a \in E,
  \\
  |(R^p_a F)^\al(b)| \le C \rh^{p+1} \, M_{p+1}\,  \frac{|b-a|^{p+1-|\al|}}{(p+1-|\al|)!}, \quad p \in \N,\, |\al| \le p,~ a,b \in E.  
\end{gather*}
The smallest constant $C$ defines a complete norm on $\cB^{\seq M}_\rh(E)$. We define the Roumieu class
$$\cB^{\{\seq M\}}(E) := \on{ind}_{\rh \in \N} \cB^{\seq M}_\rh(E),$$
and the Beurling class
$$ \cB^{(\seq M)}(E) := \on{proj}_{\rh > 0} \cB^{\seq M}_\rh(E).$$
An element of $\cB^{[\seq M]}(E)$ is called a \emph{Whitney ultrajet of class $\cB^{[\seq M]}$ on $E$}.

If $\fM$ is a weight matrix we set 
\[
	\cB^{\{\fM\}}(E) := \on{ind}_{\seq M \in \fM} \cB^{\{\seq M\}}(E) \quad \text{ and  }
	\quad 
	\cB^{(\fM)}(E) := \on{proj}_{\seq M \in \fM} \cB^{(\seq M)}(E).
\]

\begin{remark}
   If $U$ is an open subset of $\R^n$ and $F \in \cJ^\infty(U)$ satisfies 
   \begin{equation*}
  (R^p_a F)^\al (x) = o(|x-a|^{p-|\al|}) \quad \text{ for $a,x \in U$, $p\in \N$, $|\al| \le p$ as $|x-a|\to 0$},
\end{equation*}
then there exists $f \in C^\infty(U)$ with $F = j^\infty_U(f)$. 
It follows that the space of functions and the space of jets that were both denoted by $\cB^{[\seq M]}(U)$ coincide, 
which justifies the consistent use of the notation.
\end{remark}

\subsection{Quasiconvex domains} \label{lem:metext} 

A subset $X$ of $\R^n$ is called \emph{quasiconvex} if any two points $x,y \in X$ can be joined by a rectifiable path 
      in $X$ of length $\le C|x-y|$, for some constant $C$ independent of $x,y$.
By a \emph{quasiconvex domain} in $\R^n$ we mean a non-empty open subset $U \subseteq \R^n$ that is quasiconvex.       
	
It follows easily that the closure of any quasiconvex domain $U$ is quasiconvex as well, in fact, any two points $x,y$ in 
the boundary of $U$ 
can be joined by a rectifiable path of length $\le C|x-y|$ (with possibly a larger constant) which lies in $U$ except the endpoints.

\begin{lemma}
	\label{prop:whitext}
	Let $U \subseteq \R^n$ be a bounded quasiconvex domain and $f \in \cB^{[\seq M]}(U)$. Then each partial derivative $f^{(\al)}$ 
  admits a unique continuous extension $f^\al$ to $\ol U$ such that $(f^\al)_{\al \in \N^n}\in \cB^{[\seq M]}(\ol U)$.
\end{lemma}

\begin{proof}
	That the extension $f^\al$ exists (and is unique) follows from the mean value theorem, 
	since all first order derivatives of $f^{(\al)}$ are uniformly bounded on $U$.
	Since $E=\ol U$ is quasiconvex, $(f^\al)$ is a Whitney jet of class $C^\infty$ and hence extends to a 
  smooth function on $\R^n$; cf.\ \cite[Proposition 1.10]{Rainer18}. 
	That $(f^\al)\in \cB^{[\seq M]}(E)$ follows from 
	 \cite[Lemma 10.1]{Rainer18} (which is only formulated for Roumieu classes, but its proof also shows the Beurling case). 
\end{proof}

\section{Ultradifferentiable classes by almost analytic extensions} \label{sec:analyticextension}

\subsection{Characterization theorems}

Before we formulate the main theorems of this section, we need one additional definition.

\begin{definition}
	Let $\fM$ be a weight matrix.
  \begin{enumerate}
    \item A function $f : U \to \R$ is called \emph{$\{\fM\}$-almost analytically extendable} if it has an  
  $(h_{\seq m},\rh)$-almost analytic extension for some $\seq M \in \fM$ and some $\rh>0$.
    \item A function $f : U \to \R$ is called \emph{$(\fM)$-almost analytic extendable} 
    if, for all $\seq M \in \fM$ and all $\rh>0$, there is an $(h_{\seq m},\rh)$-almost analytic extension of $f$. 
  \end{enumerate}
\end{definition}

\begin{theorem}[Roumieu case]
	\label{thm:Rchar}
	Let $\fM$ be an R-regular weight matrix and $U \subseteq \R^n$ a bounded quasiconvex domain. 
	Then $f \in \cB^{\{\fM\}}(U)$ if and only if $f$ is $\{\fM\}$-almost analytically extendable.
\end{theorem}

Since any open subset of $\R^n$ can be exhausted by relatively compact quasiconvex domains (e.g., connected finite unions of balls) 
we immediately get a characterization of local classes.

\begin{corollary}
	\label{cor:Rcharacterization}
		Let $\fM$ be an R-regular weight matrix.
		Let $U \subseteq \R^n$ be open. 
		Then $f \in \cE^{\{\fM\}}(U)$ if and only if  
		$f|_V$ is $\{\fM\}$-almost analytically extendable for each quasiconvex domain $V$ relatively compact in $U$.
\end{corollary}

\begin{theorem}[Beurling case]
  \label{thm:Bchar}
  Let $\fM$ be a B-regular weight matrix and $U \subseteq \R^n$ a bounded quasiconvex domain. Then $f \in \cB^{(\fM)}(U)$ if and only if $f$ is 
  $(\fM)$-almost analytically extendable.
\end{theorem}

Again the following is immediate.

\begin{corollary}
  \label{cor:Bcharacterization}
    Let $\fM$ be a B-regular weight matrix.
    Let $U \subseteq \R^n$ be open. 
    Then $f \in \cE^{(\fM)}(U)$ if and only if  
    $f|_V$ is $(\fM)$-almost analytically extendable for each quasiconvex domain $V$ relatively compact in $U$.
\end{corollary}

\begin{remark}
    In the case that $\fM$ consists only of a single weight sequence, \Cref{thm:Rchar} reduces to a 
    slight generalization of Dynkin's original result \cite{Dynkin80}. In fact, Dynkin's assumption 
    that $\mu_k/k$ is increasing implies \Cref{def:Rregular}\eqref{goodR} with $\seq n = \seq m$. 

    If the assumption \Cref{def:Rregular}\eqref{goodR} is replaced by \Cref{rem:regular}(1) 
    which is strictly stronger, by \Cref{example}, then one 
    can use \cite[Corollary 9]{Rainer:aa} and the result of Dynkin to get \Cref{thm:Rchar}.   
\end{remark}

\subsection{Proofs of \texorpdfstring{\Cref{thm:Rchar}}{} and \texorpdfstring{\Cref{thm:Bchar}}{}}

The arguments in this section are essentially due to Dynkin \cite{Dynkin80}. 
First we recall the Bochner-Martinelli formula.
In the standard Wirtinger notation 
\[
\frac{1}{(2i)^n} (d\ol z_1\wedge dz_1)\wedge \dots \wedge (d\ol z_n\wedge dz_n) = d \cL^{2n}(z)
\]
is the usual volume element of $\R^{2n} \cong \C^n$ and 
\[
\db F(z) := \sum_{j = 1}^n \frac{\p F (z)}{\p \ol z_j}d\ol z_j.
\] 

\begin{theorem}[Bochner-Martinelli formula]
	\label{thm:Bochner}
	Let $V \subseteq \C^n$ be a bounded domain with $C^1$ boundary and $F \in C^1(\ol{V})$. Then 
	\[
	F(z) = \int_{\p V} F(\ze) \om(\ze,z) - \int_{V} \db F(\ze)\wedge \om(\ze,z), 
	\]
	where $\om$ is the $(n,n-1)$-form ($\widehat{d\ol \ze_j}$ means that $d\ol \ze_j$ is omitted) 
	\[
	\om(\ze,z) = \frac{(n-1)!}{(2\pi i)^n} \frac{1}{|z  - \ze|^{2n}} \sum_{j = 1}^n (\ol \ze_j - \ol z_j)\, 
  d\ol \ze_1\wedge d\ze_1 \wedge \dots \wedge \widehat{ d\ol \ze_j}\wedge \dots \wedge d\ol \ze_n \wedge d\ze_n.
	\]
\end{theorem}

\begin{proposition}
	\label{prop:restriction}
	Let $\seq M$ be a positive sequence with $m_k^{1/k} \to \infty$, $\rh>0$, and $U \subseteq \R^n$ bounded open. 
  Any $f : U \to \R$ with an $(h_{\seq m},\rh)$-almost analytic extension belongs to 
  $\cB^{\{\seq M\}}(U)$. 
  If for every $\rh>0$ there is an $(h_{\seq m},\rh)$-almost analytic extension of $f$, then $f$ belongs to 
  $\cB^{(\seq M)}(U)$. 
\end{proposition}

\begin{proof}
	Let $F$ be an $(h_{\seq m},\rh)$-almost analytic extension of $f$. Since $F$ has compact support,
	\Cref{thm:Bochner} implies 
	\[
	f(x) = F(x) = - \int_{\C^n}  \ol \p F(\ze) \wedge \om(\ze,x), \quad x \in U.
	\]
	By differentiating under the integral sign it is easy to check that 
	$f := F|_U$ is of class $C^\infty$ on $U$ with 
	\[
	\p^\al f(x) = - \int_{\C^n}  \ol \p F(\ze) \wedge \p^\al\om(\ze,x), \quad x \in U.
	\]
	By Fa\`a di Bruno's formula and the Leibniz rule, we get
	\[
	\bigg|\p^\al \bigg(\frac{1}{|x-\ze|^{2n}} \sum_{j=1}^n (\ol \ze_j - x_j)\bigg) \bigg| \le \frac{C(n)^{|\al|} |\al|! }{|x-\ze|^{2n+|\al|-1}}.
	\]
  Choose $R>0$ large enough such that $U\cup\supp(F) \subseteq B(0,R)$. Writing $D=\frac{(n-1)!}{\pi^n}$, 
  we get for $x \in U$, 
	\begin{align*}
	\frac{|\p^\al f(x)|}{(C(n)\rh)^{|\al|} M_{|\al|}} &\le D \int_{B(0,R)} 
  \frac{A h_{\seq m} ( \rh d(\ze,\ol U))}{\rh^{|\al|} m_{|\al|} |x-\ze|^{2n+|\al|-1}} \, d\cL^{2n}(\ze)
	\\
	&\le A D   \int_{B(0,R)} \frac{d(\ze,\ol U)^{|\al|}}{|x-\ze|^{2n+|\al|-1}} \, d\cL^{2n}(\ze)
	\\
	&\le A D  \int_{B(0,R)} \frac{1}{|x-\ze|^{2n-1}} \, d\cL^{2n}(\ze) < \infty. 
	\end{align*}
  The assertions follow.
\end{proof}

\begin{lemma}
	\label{lem:taylorest}
	Let $E \subseteq \R^n$ be compact and $f=(f^\al) \in \cB^{\{\seq M\}}(E)$ (resp.\ $f \in \cB^{(\seq M)}(E)$). 
	Then there exist $C,D > 0$ (resp.\ for each $D$ there exists $C$) such that for all $a_1,a_2 \in E$, $z \in \C^n$, 
  and $\al \in \N^n$ with $|\al|\le j$,
		\begin{equation*} 
		|\p_z^\al T^j_{a_1}f(z)- \p_z^\al T^j_{a_2}f(z)|\le  C D^{j+1}  |\al|!\, m_{j+1}(|a_1-z|+|a_1-a_2|)^{j-|\al|+1}.
		\end{equation*}
\end{lemma}

\begin{proof}
	For fixed $a_1,a_2 \in E$ and $j \in \N$, the function $z \mapsto T^j_{a_1}f(z)-T^j_{a_2}f(z)$ is a 
	polynomial in $z$ of degree $j$ satisfying   
	\[
	T^j_{a_1}f(z)-T^j_{a_2}f(z)= \sum_{|\be|=0}^{j} (R^j_{a_2}f)^\be(a_1) \frac{(z-a_1)^\be}{\be!}.
	\]
	From this the assertion follows easily; cf.\ \cite[Proposition 10]{ChaumatChollet94}.
\end{proof}

A crucial ingredient in the subsequent construction consists of the so-called \emph{regularized distance}. 
Given a closed set $E \subseteq \R^n$, the distance function $z \mapsto d(z,E)$ is far from being smooth. 
But it is possible to construct a smoothened version of the distance, having essentially the same properties. 

\begin{proposition}[{\cite[VI 2.1 Theorem 2]{Stein70}}]
	\label{prop:regdist}
	Let $E \subseteq \R^n$ be closed. There is a $C^\infty$-function $\de : \R^n \setminus E \to \R$ such that 
	\begin{enumerate}
		\item $c_1 d(z,E) \le \de(z) \le c_2d(z,E)$ for all $z \notin E$,
		\item for all $\al \in \N^n$ and $z \in \R^n \setminus E$,
		\[
		\big| \p^\al \de(z) \big| \le B_\al d(z,E)^{1-|\al|},
		\]
	\end{enumerate}
	where the constants $B_\al, c_1,c_2$ are independent of $E$.
\end{proposition}

The following lemma is well-known.

\begin{lemma} \label{meas}
  Let $E \subseteq \R^n$ be compact. Let $\al>1$. 
  There exists a Borel measurable map $b : \R^n\setminus E \to E$ such that $|x - b(x)| < \al d(x,E)$ for all 
  $x \in \R^n\setminus E$.
\end{lemma}

\begin{proof}
  Let $\{x_k\}_{k \in \N}$ be a dense subset of $E$. Define 
  $m : \R^n \setminus E \to \N$ by $m(x) := \min\{k : |x-x_k| < \al d(x,E)\}$ and $x : \N  \to E$ by 
  $x(k) = x_k$. Then both $m$ and $x$ are Borel measurable, hence so is $b := x \o m$.  
\end{proof}

\begin{proposition}
  \label{prop:existenceofext}
  Let $\seq M\le  \seq N$ and $\seq S$ be positive sequences such that  $m_k^{1/k}$, $n_k^{1/k}$, and $s_k^{1/k}$ tend to 
  $\infty$ and  
  \begin{gather}
    \E C_1\ge 1 \A t>0 : \ol \Ga_{\seq n} (C_1 t) \le \ul \Ga_{\seq m}(t), \label{eq:countcomp}
    \\
    \E C_2\ge 1 \A j \in \N : n_{j+1} \le C_2^{j+1}s_j. \label{ass2}
  \end{gather}  
  Let $E \subseteq \R^n$ be compact. Assume that $f=(f^\al)_\al \in \cB^{\seq M}_{C_0}(E)$ satisfies
  \begin{align}
     \label{eq:growth1}
     &\A \al \in  \N^n \A x \in E : |f^{\al}(x)|\le C C_0^{|\al|}M_{|\al|},\\
     \label{eq:growth2}
     \begin{split}
     &\A j \in \N \A a_1, a_2 \in E \A z \in \C^n:\\ 
     &\qquad |T^j_{a_1} f(z)-T^j_{a_2} f(z)|\le C C_0^{j+1}  m_{j+1}(|a_1-z|+|a_1-a_2|)^{j+1},
     \end{split}
   \end{align}
   for suitable constants $C,C_0>0$.
  Then there exists an extension $F \in C^\infty_c(\C^n)$ of $f$ such that 
  \begin{equation} \label{toshow}
    \A z \in \C^n : |\ol \p F(z)| \le A h_{\seq s} ( 12n C_0 C_1 d(z,E)),
  \end{equation}
  where $A= A(C,C_0,C_1,C_2,n)$.
\end{proposition}

\begin{proof}
    By \Cref{meas}, there is a Borel measurable map $z \mapsto \hat z$ such that 
    \begin{equation}
      d(z) \le |z-\hat z| < 2 d(z),
    \end{equation}
    where $d(z)=d(z,E)$.
    Then 
  \[
    G(z) := T_{\hat z}^{p(z)}  f(z), \quad z \in \C^n \setminus E,
  \]
  where
  \[
    p(z) := \ul \Ga_{\seq m} (8nC_0 \, d(z)),
  \]
  is Borel measurable and locally bounded. Indeed,
  \begin{equation} \label{eq:dzedz}
    d(z) \le 2d(\ze) \le 3 d(z) \quad \text{  for  } \ze \in B(z,d(z)/2),
  \end{equation} 
  and hence $p(\ze) = \ul \Ga_{\seq m}( 8n C_0 \, d(\ze)) 
  \le \ul \Ga_{\seq m}(4n C_0 \, d(z))$.

  Let $\ps \in C^\infty(\C)$ be a non-negative, rotationally invariant function satisfying 
  $\int \ps \, d \cL^{2} = 1$ such that  
  $\Ps(z) := \ps(z_1) \cdots \ps(z_n)$ has support in the unit ball in $\C^n$.
  Define 
  \[
    F(z) := 
      \frac{(2c_2)^{2n}}{\de(z)^{2n}} \int \Ps\Big(\frac{2c_2(\ze-z)}{\de(z)}\Big) G(\ze) \, d \cL^{2n}(\ze)  
      \quad \text{ for  } z \in \C^n \setminus E.
  \]   
  Here $\de$ is the regularized distance for $E \subseteq \C^n \cong \R^{2n}$ from \Cref{prop:regdist} and $c_2$ 
  is chosen as in \Cref{prop:regdist}(1). If we do not specify the domain of integration, as above, 
  it should be understood as $\C^n$.
  It is not hard to see that $F$ is $C^\infty$ on $\C^n \setminus E$.

    For each $\al = (\al_1,\al_2) \in \N^n \times \N^n$ we define $F^\al : \C^n \to \C$ by setting 
    \begin{align*}
      F^\al(a) := \p_z^{\al_1} \p_{\ol z}^{\al_2} F(a) \quad \text{ if } a \not\in E         
    \end{align*}   
    and if $a \in E$ then $F^\al(a)$ is uniquely determined by the identity 
    \begin{align*}
      \sum_{\al = (\al_1,\al_2)\in \N^n \times \N^n} \frac{F^\al(a)}{\al!} Z^{\al_1} \ol Z^{\al_2} = 
      \sum_{\be \in \N^n}  \frac{f^\be(a)}{\be!} Z^\be \quad \text{ in } \C[[Z,\ol Z]].  
    \end{align*}
    Then for all $j \in \N$, $a \in E$, and $z = x+iy \in \C^n$,
    \begin{equation}  \label{jets}
      T^j_a F(z) := \sum_{\substack{\al = (\al_1,\al_2)\\ |\al|\le j}} \frac{F^\al(a)}{\al!} (z-a)^{\al_1} (\ol z -a )^{\al_2}
      = T^j_a f(z).
    \end{equation}
    We will write $F=(F^\al)_\al$; this should not cause too much confusion with the \emph{function} $F$. 
  We will prove the following two claims from which the theorem follows easily:
  \begin{enumerate}
    \item \eqref{toshow} holds for all $z \in \C^n \setminus E$. 
    \item $F^0$ is $C^\infty$ on $\C^n$ and $F^\al = \p_z^{\al_1} \p_{\ol z}^{\al_2} F^0$ for all 
    $\al=(\al_1,\al_2) \in \N^n \times \N^n$.
  \end{enumerate}

  Let us first show (1).
  Using \Cref{prop:regdist}, it is not hard to see that   
  \begin{equation} \label{eq:kernel}
    \Big|\p_z^{\al_1} \p_{\ol z}^{\al_2} \Big(\frac{1}{\de(z)^{2n}} 
    \Ps\Big(\frac{2c_2(\ze-z)}{\de(z)}\Big)\Big)\Big| \le \frac{K_\al}{d(z)^{2n+|\al|}}, 
  \end{equation}
  for all $\ze, z \notin E$ and $\al=(\al_1,\al_2) \in \N^{n} \times \N^n$.
  For any polynomial $P\in \C[z]$, we have 
  \[
    \frac{(2c_2)^{2n}}{\de(z)^{2n}} \int \Ps\Big(\frac{2c_2(\ze-z)}{\de(z)}\Big) P(\ze) \, d \cL^{2n}(\ze) = P(z), 
  \]
  which follows from the Cauchy integral formula,
  \begin{align*}
    &\frac{(2c_2)^{2n}}{\de(z)^{2n}} \int \Ps\Big(\frac{2c_2(\ze-z)}{\de(z)}\Big) \ze^\al \, d \cL^{2n}(\ze) =
    \int_{B(0,1)^n} \Ps(\ze) \Big(\tfrac{\de(z)}{2c_2} \ze + z\Big)^\al  \, d \cL^{2n}(\ze) 
    \\
    &=\prod_{j=1}^n \int_{B(0,1)} \ps(\ze_j) \Big(\tfrac{\de(z)}{2c_2} \ze_j + z_j\Big)^{\al_j}  \, d \cL^{2}(\ze_j)
    = z^\al.
  \end{align*}
  Thus, if $z \in \C^n \setminus E$, $z_0 \in E$, and $j \in \N$, we get 
  \begin{equation}
  \label{eq:taylorapp}
    F(z) = T_{z_0}^j  f(z) + \frac{(2c_2)^{2n}}{\de(z)^{2n}} \int \Ps\Big(\frac{2c_2(\ze-z)}{\de(z)}\Big) 
    \big(G(\ze) - T_{z_0}^j  f(\ze)\big)  
    \, d \cL^{2n}(\ze).
  \end{equation}
  Hence, by choosing $z_0 = \hat z$,
  \[
    \ol \p F(z) =   \int \ol \p \Big(\frac{(2c_2)^{2n}}{\de(z)^{2n}} \Ps\Big(\frac{2c_2(\ze-z)}{\de(z)}\Big)\Big) 
    \big(G(\ze) - T_{\hat z}^j  f(\ze)\big)  
    \, d \cL^{2n}(\ze).
  \]
  By \eqref{eq:kernel}, for all $j \in \N$,
  \begin{align}
  \label{eq:keyest} \nonumber
    |\ol \p F(z)| 
    &\le \frac{K}{d(z)^{2n+1}}   \int_{B(z,\frac{\de(z)}{2c_2})}  
    |G(\ze) - T_{\hat z}^j  f(\ze)|  
    \, d \cL^{2n}(\ze) 
    \\
    &\le  \frac{K}{d(z)}  \sup_{\ze \in B(z, d(z)/2)}  
    |G(\ze) - T_{\hat z}^j  f(\ze)|,
  \end{align}
  where $K$ denotes a generic constant.
  Now
  \[
    |G(\ze) - T_{\hat z}^j  f(\ze)| \le |T_{\hat z}^j  f(\ze) - T_{\hat \ze}^j  f(\ze)| 
    + |T_{\hat \ze}^j  f(\ze) - T_{\hat \ze}^{p(\ze)}  f(\ze)|.
  \]
  We estimate the summands separately. 
  So fix some arbitrary $z \in \C^n \setminus E$, take $\ze \in B(z,d(z)/2)$ and set 
  \[
  j+1 : =\ol \Ga_{\seq n}(12n C_0 C_1 \, d(z)).
  \]
  Since 
  $|\hat{z}-\ze| +|\hat{z}-\hat{\ze}|\le 9d(z)$,
  \eqref{eq:growth2}  
  and the definition of $j+1$ give 
  \begin{align*}
  |T_{\hat z}^j f(\ze) - T_{\hat \ze}^j f(\ze)| 
  &\le C (9 C_0 \, d(z))^{j+1} \, m_{j+1}\\
  &\le C(12 nC_0 C_1\, d(z))^{j+1} \, n_{j+1}
  = C h_{\seq n}(12 nC_0 C_1\, d(z)).
  \end{align*}
  By \eqref{eq:countcomp}, \eqref{eq:dzedz}, and \Cref{basic}(2), $j+1\le \ul \Ga_{\seq m}(12n C_0  \, d(z)) 
  \le \ul \Ga_{\seq m}(8n  C_0 \, d(\ze)) = p(\ze)$.
  Thus (using that there are $\binom{k+n-1}{n-1} \le 2^{k+n-1}$ many $\be \in \N^n$ such that $|\be| =k$) 
  \begin{align*}
    |T_{\hat \ze}^j  f(\ze) - T_{\hat \ze}^{p(\ze)}  f(\ze)|
    & = \Big| \sum_{j<|\be| \le p(\ze)} 
    f^{\be}(\hat \ze) \frac{(\ze - \hat \ze)^\be}{\be!} \Big|  
    \\ 
    &\le C  \sum_{j<|\be| \le p(\ze)} 
    (2n C_0  d(\ze))^{|\be|} m_{|\be|} \quad \text{ by } \eqref{eq:growth1} 
    \\ 
    &\le 2^{n-1} C \sum_{k = j+1}^{p(\ze)} 
    (8 n C_0 \, d(\ze))^{k} m_{k} 2^{-k} 
    \\
    &\le 2^{n-1} C (8nC_0 \, d(\ze))^{j+1} m_{j+1} \qquad \text{ by  \Cref{basic}\eqref{eq:ulGa3}} 
    \\ 
    &\le 2^{n-1} C (8nC_0 \, d(\ze))^{j+1} n_{j+1}  \qquad \text{ since } \seq m \le \seq n
    \\
    &\le 2^{n-1} C (12 nC_0 C_1 \, d(z))^{j+1} n_{j+1}  
    \\
    &= 2^{n-1} C\,  h_{\seq n}(12 nC_0 C_1 \, d(z)) \qquad \text{ by } \eqref{counting2}.
  \end{align*}
  Combining the estimates, we get 
  \begin{align*}
    |\ol \p F(z)| 
    &\le  \frac{K}{d(z)} h_{\seq n}(12 nC_0C_1 \, d(z)).
  \end{align*}
  By \eqref{ass2} and the definition of $h_{\seq n}$, we have $h_{\seq n}(t)/t \le C_2 h_{\seq s}(t) $, which implies
  \[
  |\db F(z)| \le Kh_{\seq s}(12 nC_0C_1 \, d(z)).
  \]
  Thus claim (1) is proved.

  Let us show (2). To this end we prove that for all $j \in \N$,  $\al=(\al_1,\al_2) \in \N^n \times \N^n$ with $|\al| \le j$, 
  $z \in \C^n$, and $a \in E$,
  \begin{equation} \label{eq:key}
    |F^\al(z) - \p_z^{\al_1} \p_{\ol z}^{\al_2} T^j_{a} F(z)| = o(|z-a|^{j-|\al|}) \quad \text{ as } |z-a| \to 0. 
  \end{equation}  
  This implies (2): First of all it implies that all $F^\al$ are continuous on $\C^n$. 
  If $a\in E$ and $z \in \C^n \setminus E$,  
  then, for $j > |\al|$, where $e_i$ denotes the $i$-th standard unit vector in $\R^n$, 
  \begin{align*}
     &|F^\al(z) - F^\al(a) - \sum_{i=1}^{n} (z_i - a_i) F^{(\al_1+e_i,\al_2)}(a) 
     | = o(|z-a|)  \quad \text{ as } |z-a| \to 0,
  \end{align*}
  by \eqref{eq:key} and the fact that $T^j_{a} F(z) = T^j_a f(z)$ is a polynomial. 
  Notice that, by \eqref{jets}, $\p_z^{\al_1} \p_{\ol z}^{\al_2} T^j_{a} F(z) = F^{(\al_1,\al_2)}(a) = 0$ 
  whenever $\al_2 \ne 0$.
  It follows that $F^\al$ is $C^1$, $\p_{z}^{e_i} F^\al = F^{(\al_1+e_i,\al_2)}$, and 
  $\p_{\ol z}^{e_i} F^\al = F^{(\al_1,\al_2+e_i)}$.
 
  Now \Cref{lem:taylorest} implies, for $a_1, a_2 \in E$,
  \begin{align} \label{eq:taylor2}
  |\p_z^{\al_1} \p_{\ol z}^{\al_2} T^j_{a_1}F(z)-\p_z^{\al_1} \p_{\ol z}^{\al_2} T^j_{a_2}F(z)| = O\big((|a_1-a_2| +|z-a_1|)^{j-|\al|+1}\big).
\end{align}
  In particular, it suffices to show \eqref{eq:key} for $a = \hat z$, since $|\hat z - a| \le 3 |z-a|$. 
  The estimates for $| G(\ze) -  T_{\hat z}^j  f(\ze)|$ above also yield that for $\ze \in B(z,d(z)/2)$ we have  
  \[
    | G(\ze) -  T_{\hat z}^j  f(\ze)| = O(d(z)^{j+1}).
  \]
  Since 
  $T^j_{\hat z} f(z) = T^j_{\hat z} F(z)$ by \eqref{jets}, 
  we may conclude with \eqref{eq:taylorapp} for $z_0 = \hat z$ and \eqref{eq:kernel} that
  \[  
    |\p_z^{\al_1} \p_{\ol z}^{\al_2} \big(F -  T^j_{\hat z} F\big)(z)| = o(|z-\hat z|^{j-|\al|}) \quad \text{ as } |z-\hat z| \to 0,
  \]
  if $z \in \C^n\setminus E$. Thus \eqref{eq:key} is proved.
\end{proof}

\begin{proof}[Proof of \Cref{thm:Rchar}]
The theorem now follows easily from
\Cref{prop:restriction}, \Cref{prop:whitext}, \Cref{lem:taylorest}, and \Cref{prop:existenceofext}.
\end{proof}

\begin{proof}[Proof of \Cref{thm:Bchar}]
  Suppose that $f \in \cB^{(\fM)}(U)$.
  Let $\seq S \in \fM$ and $\rh>0$. Since $\fM$ is B-regular,
  there exist $\seq M,\seq N \in \fM$ such that \eqref{eq:countcomp} and \eqref{ass2} hold. 
  By \Cref{prop:whitext} and \Cref{lem:taylorest},
  we have \eqref{eq:growth1} and \eqref{eq:growth2} for $C_0 = \rh/(12n C_1)$. 
  So \Cref{prop:existenceofext} yields an extension $F\in C^\infty_c(\C^n)$ of $f$ 
  such that 
  \[
   |\ol \p F(z)| \le A h_{\seq s} ( \rh d(z,\ol U)).
  \] 
  Hence $f$ is $(\fM)$-almost analytically extendable.
  The converse follows from \Cref{prop:restriction}. 
\end{proof}

\subsection{A stronger result}

Assume that $\seq M$ is a strongly log-convex (i.e.\ $\mu_k/k$ is increasing) weight sequence such that $m_k^{1/k} \to \infty$. 
Then we can choose the same extension $F$ of $f \in B^{(\seq M)}(U) = \bigcap_{\rh>0} B^{\seq M}_\rh(U)$ for every $\rh$.

\begin{theorem}
   Let $\seq M$ be a strongly log-convex weight sequence with $m_k^{1/k} \to \infty$ and 
   $(M_{k+1}/M_k)^{1/(k+1)}$ bounded. 
   Let $U \subseteq \R^n$ be a bounded quasiconvex domain. 
   Then $f \in \cB^{(\seq M)}(U)$ if and only if $f$ admits an extension $F \in C^1_c(\C^n)$ such that
   \[
      \forall \rh>0 \E C \ge 1 \A z \in \C^n : |\ol \p F(z)| \le C h_{\seq m}(\rh d(z,\ol U)).
   \]
\end{theorem}

\begin{proof}
  Use \cite[Lemma 6]{Komatsu79b} (or \Cref{Komatsu} below) and the Roumieu result.
\end{proof}

We do not know if a similar statement holds in the general case.

\section{Applications to classes defined by weight functions}

In this section we fully characterize \emph{when} the classes $\cB^{\{\om\}}$ and $\cB^{(\om)}$ 
admit a description by almost analytic extensions. It turns out that this feature is equivalent to several other 
pertinent properties of the classes. 

First we recall the description by associated weight matrices.

\subsection{Weight functions and the associated weight matrix} \label{weightfunction}

Two weight functions $\om$ and $\si$ are said to be \emph{equivalent} if $\om(t) = O(\si(t))$ and $\si(t) = O(\om(t))$ 
as $t \to \infty$. 
For each weight function $\om$ there is an equivalent weight function 
$\tilde \om$ such that $\om(t) = \tilde \om(t)$ for large $t>0$ 
and $\tilde \om |_{[0,1]} =0$. It is thus no restriction to assume that $\om |_{[0,1]} =0$ when necessary.

For weight functions $\om$ and $\si$ we have $\cB^{[\om]} \subseteq \cB^{[\si]}$ if and only if
$\si(t) = O(\om(t))$ as $t \to \infty$,
cf.\ \cite{Bjoerck66}, \cite{BMT90}, or \cite[Corollary 5.17]{RainerSchindl12}; 
in particular, $\om$ and $\si$ are equivalent if and only if 
$\cB^{[\om]} = \cB^{[\si]}$.

\begin{definition}[Associated weight matrix]
Following \cite[5.5]{RainerSchindl12} we 
associate with any weight function $\om$ a weight matrix $\fW = \{\seq W^x\}_{x>0}$ by setting 
\begin{equation*}
  W^x_k := \exp(\tfrac{1}{x}\vh^*(x k)), \quad k \in \N.
\end{equation*} 
Moreover, we define
\begin{equation*}
  \vt^x_k := \frac{W^x_k}{W^x_{k-1}}.   
\end{equation*}  
\end{definition}

\begin{lemma}[{\cite[Lemma 2.4]{RainerSchindl17}}] \label{lemma4} 
  We have:
  \begin{enumerate}
     \item Each $\seq W^x$ is a weight sequence (in the sense of \Cref{weights}).
     \item $\vt^x \le \vt^y$ if $x \le y$, which entails $\seq W^x \le \seq W^y$. \label{lemma4(2)}
     \item For all $x>0$ and all $j,k \in \N$,  $W^x_{j+k} \le W^{2x}_{j} W^{2x}_{k}$ and $w^x_{j+k} \le w^{2x}_{j} w^{2x}_{k}$.
     \label{eq:mW}
     \item For all $x>0$ and all $k \in \N_{\ge 2}$, $\vt^x_{2k} \le   \vt^{4x}_{k}$. 
     \item $\A \rh>0 \E H\ge 1 \A x >0 \E C \ge 1 \A k \in \N : \rh^k W^x_k \le C W^{Hx}_k$. \label{5.10}
     \item If $\om(t) = o(t)$ as $t \to \infty$ then $(w^x_k)^{1/k} \to \infty$ and $\vt^{x}_k/k \to \infty$ for all $x>0$.
   \end{enumerate} 
\end{lemma}

\begin{theorem}[{\cite[Corollaries 5.8 and 5.15]{RainerSchindl12}}] \label{representation}
  Let $\om$ be a weight function and let $\fW = \{\seq W^x\}_{x>0}$ be the associated weight matrix. 
  Then, as locally convex spaces,
  \[
    \cB^{[\om]}(U) = \cB^{[\fW]}(U) \quad \text{ and } \quad \cE^{[\om]}(U) = \cE^{[\fW]}(U).
  \]
  We have $\cB^{[\om]}(U) = \cB^{[\seq W^x]}(U)$ (or $\cE^{[\om]}(U) = \cE^{[\seq W^x]}(U)$) for all $x>0$ if and only if 
  \begin{align}
     \E H\ge 1 \A t\ge 0 : 
      2\om(t) \le \om(Ht) + H.   \label{om6}
  \end{align} 
  Moreover, \eqref{om6} holds if and only if some (equivalently each) $\seq W^x$ has moderate growth.  
\end{theorem}

\begin{remark}\label{rem:Comparision}
  Let us emphasize that the fact that $\cE^{[\om]} = \cE^{[\seq M]}$ for some weight sequence $\seq M$ 
  if and only if $\om$ satisfies \eqref{om6} is due to \cite{BMM07}.
\end{remark}

\subsection{Concave weight functions}

We will see that the classes $\cB^{[\om]}$ that admit description by almost analytic extension 
are precisely those determined by a concave weight function $\om$. 
The proof depends on the following result obtained in \cite{Rainer:aa}.

\begin{proposition} \label{prop:strongmatrix}
  Let $\om$ be a weight function satisfying $\om(t) = o(t)$ as $t \to \infty$ which is equivalent 
  to a concave weight function. For each $x>0$ there exist constants $A,B,C >0$ such that 
  \begin{equation} \label{eq:strong}
      A^{-1} w^{x/B}_k \le \ul w^x_k \le w^x_k \le C^k \ul w^{Bx}_k \quad \text{ for all } k \in \N. 
  \end{equation} 
  The weight matrix
  $\fS := \{\seq S^x = (k! \ul w^x_k)_k : x >0\}$ is regular.
\end{proposition}

\begin{proof}   
  Only the regularity of $\fS$ was not yet observed in \cite{Rainer:aa}. 
  Notice that $w^x_{j+1} \le C^j w^y_j$ for all $j$ implies  
  $\ul w^x_{j+1} \le C^j \ul w^y_j$ for all $j$  
  which is clear by the properties of the log-convex minorant, since
  $\om_{(C^j w^y_j)_j}(t) = \om_{\seq w^y}(t/C)$ and hence
   $\ul {(C^j  w^y_j)_j} = (C^j \ul w^y_j)_j$.   
  Since $\ul {\seq w}^x$ is log-convex, $\ol \Ga_{\ul {\seq w}^x} = \ul \Ga_{\ul {\seq w}^x}$.  
  Evidently, $(\ul w^x_k)^{1/k} \to \infty$ for all $x>0$, by \Cref{lemma4} and \eqref{eq:strong}.  
\end{proof}

\begin{corollary} \label{cor:strongmatrix}
   Let $\om$ be a weight function satisfying $\om(t) = o(t)$ as $t \to \infty$.
   The weight matrix $\fW$ associated with $\om$
   is always semiregular. 
   If additionally $\om$ is equivalent to a concave weight function, then $\fW$ is equivalent to a regular 
   weight matrix.    
\end{corollary}

We will now prove a version of almost analytic extension in the Beurling case $\cB^{(\om)}$ for \emph{strong} 
weight functions $\om$ 
which is stronger than provided by the general \Cref{thm:Bchar}. 
Recall that a weight function $\om$ is called \emph{strong} if 
\begin{equation} \label{eq:strongweight}
    \E C>0 \A t>0 :  \int_1^\infty \frac{\om(tu)}{u^2} \, du \le C\om(t) + C. 
\end{equation}
Evidently, a strong weight function $\om$ is non-quasianalytic.
In fact, \eqref{eq:strongweight} is equivalent to the validity of the Whitney extension theorem in the classes $\cB^{[\om]}$; 
see \cite{BBMT91}.
Moreover, a strong weight function $\om$ is equivalent to a concave weight function, see \cite[Proposition 1.3]{MeiseTaylor88}, 
and satisfies $\om(t) = o(t)$ 
as $t \to \infty$, see \cite[Corollary 1.4]{MeiseTaylor88}; cf.\ also \cite{BBMT91} and \cite[Section 3.5]{RainerSchindl17}.

This stronger results depends on \cite[Lemma 4.4]{BBMT91} which should be compared with 
\Cref{lem:concavedescend} and \Cref{rem:notconvex} below.

\begin{theorem} \label{thm:strongerconcave}
   Let $\om$ be a strong weight function 
   and let $\fW$ be the associated weight matrix.  
   Let $U \subseteq \R^n$ be a bounded quasiconvex domain. 
   Then $f \in \cB^{(\om)}(U)$ if and only if $f$ admits an extension $F \in C^1_c(\C^n)$ such that
   \begin{equation} \label{eq:strongerconcave}
   	\forall \seq M \in \fW \A \rh>0 \E C \ge 1 \A z \in \C^n : |\ol \p F(z)| \le C h_{\seq m}(\rh d(z,\ol U)).
   \end{equation}
\end{theorem}

\begin{proof}
  If $f$ admits an extension satisfying \eqref{eq:strongerconcave} then $f \in \cB^{(\om)}$, 
  by \Cref{prop:restriction} and \Cref{representation}.
  Conversely, let $f \in \cB^{(\om)}(U)$. Set 
  \[
  	L_k := \max\Big\{ \sup_{x \in U, |\al| \le k} |\p^\al f(x)|, k!\Big\}
  \] 
  Let us proceed as in the proof of \cite[Theorem 4.5]{BBMT91}:
  Define $g: [0,\infty) \to \R$ by 
  \[
  	g(t) := \log L_k, \quad \text{ for } k \le t < k+1. 
  \]
  The arguments in \cite[Theorem 4.5]{BBMT91} show that there exists a convex function
  $h_0 : [0,\infty) \to [0,\infty)$ such that $g\le h_0$ 
  and $h:=h_0^*(\max\{0,\log t\})$ satisfies $\om(t) = o(h(t))$ as $t \to \infty$. 
  We may apply \cite[Lemma 4.4]{BBMT91} which yields a strong weight function $\si$ such that 
  $\om(t) = o(\si(t))$ and $\si(t) = o(h(t))$. 
  Hence $g\le h_0 = h_0^{**} \le (\si(e^t))^* + A$ for some constant $A>0$, whence 
  $f \in \cB^{\{\si\}}(U)$. 
  Since $\si$ is equivalent to a concave weight function, 
  there is a regular weight matrix $\fS$ such that $\cB^{[\si]} = \cB^{[\fS]}$. 
  \Cref{thm:Rchar} implies that there is an extension $F \in C^1_c(\C^n)$ of $f$ and some $\seq S \in \fS$ and $C,\rh>0$ 
  such that
  \[
  	|\ol \p F(z)| \le C h_{\seq s}(\rh d(z,\ol U)), \quad z \in \C^n.
  \]
  Since $\om(t) = o(\si(t))$ as $t \to \infty$ and hence $\seq S \{\lhd) \fW$, cf.\ \cite[Lemma 5.16]{RainerSchindl12}, 
  \eqref{eq:strongerconcave} follows. 
\end{proof}

\subsection{A characterization theorem}

The next theorem characterizes \emph{when}
the classes $\cB^{[\om]}$ 
admit a description by almost analytic extensions.

\begin{theorem} \label{thm:omegachar}
Let $\om$ be a weight function satisfying $\om(t) = o(t)$ as $t \to \infty$. The following are equivalent.
\begin{enumerate}
  \item \label{char9} $\cB^{\{\om\}}$ can be described by almost analytic extensions, i.e., there is an R-regular 
  weight matrix $\fS$ such that  
  $f \in \cB^{\{\om\}}(U)$ if and only if $f$ is $\{\fS\}$-almost analytically extendable, for every bounded quasiconvex 
  domain $U \subseteq \R^n$.
  \item \label{char10} $\cB^{(\om)}$ can be described by almost analytic extensions, i.e., there is a B-regular 
  weight matrix $\fS$ such that  
  $f \in \cB^{(\om)}(U)$ if and only if $f$ is $(\fS)$-almost analytically extendable, for every bounded quasiconvex 
  domain $U \subseteq \R^n$. 
  \item \label{char3} $\cB^{\{\om\}}$ is stable under composition.
  \item \label{char4} $\cB^{(\om)}$ is stable under composition.  
  \item \label{char1} $\om$ is equivalent to a concave weight function. 
  \item \label{char2} $\exists C>0 \E t_0 >0 \A \la \ge 1 \A t \ge t_0 : \om(\la t)\le C \la \, \om(t)$.
  \item \label{char5} There is a weight matrix $\fS$ consisting of strongly log-convex weight sequences such that
  $\cB^{\{\om\}} = \cB^{\{\fS\}}$.
  \item \label{char6} There is a weight matrix $\fS$ consisting of strongly log-convex weight sequences such that
  $\cB^{(\om)} = \cB^{(\fS)}$. 
  \item \label{char11} There is a weight matrix $\fM$  satisfying
  $\forall \seq M \in \fM \E \seq N \in \fM \E C\ge 1 \A 1 \le j \le k : \mu_j/j \le C \nu_k/k$
  and such that
  $\cB^{\{\om\}} = \cB^{\{\fM\}}$. (Recall that $\mu_k:=M_k/M_{k-1}$ and $\nu_k:=N_k/N_{k-1}$.) 
  \item \label{char12} There is a weight matrix $\fM$  satisfying
  $\forall \seq N \in \fM \E \seq M \in \fM \E C\ge 1 \A 1 \le j \le k : \mu_j/j \le C \nu_k/k$
  and such that
  $\cB^{(\om)} = \cB^{(\fM)}$.
  \item \label{char7} There is an R-regular weight matrix $\fM$ such that $\cB^{\{\om\}} = \cB^{\{\fM\}}$.
  \item \label{char8} There is a B-regular weight matrix $\fM$ such that $\cB^{(\om)} = \cB^{(\fM)}$.
\end{enumerate} 
If $\om$ is a strong weight function, then 
the extension of $f \in \cB^{(\om)}(U)$ in \eqref{char10} may be taken independent of $\seq S \in \fS$ and $\rh>0$, 
  as in \Cref{thm:strongerconcave}.
\end{theorem}

Notice that the conditions in the theorem are furthermore equivalent to stability of the class $\cB^{[\om]}$ 
under inverse/implicit functions and solving ODEs
and, in terms of the associated weight matrix $\fW$, to 
\[
  \A x>0 \E y>0 : (w^x_j)^{1/j} \le C\, (w^y_k)^{1/k} \quad \text{ for } j \le k
\]
as well as
\[
  \A y>0 \E x>0 : (w^x_j)^{1/j} \le C\, (w^y_k)^{1/k} \quad \text{ for } j \le k
\]
see \cite{RainerSchindl14}.

\begin{proof}
   \eqref{char9} $\Rightarrow$ \eqref{char3} and \eqref{char10} $\Rightarrow$ \eqref{char4} follow from \Cref{prop:composition}. 
  Indeed, $h_{\seq s} \le h_{\seq t}$ if $\seq S \le \seq T \in \fS$. 
  For the Beurling case notice that 
  for any given $\seq S \in \fS$ and $\rh>0$ we know that $g$ has an $(h_{\seq s},\rh)$-almost analytic extension 
  $G$ and $f$ has an $(h_{\seq s},\rh/\on{Lip}(G))$-almost analytic extension $F$. Hence, by \Cref{prop:composition}, 
  $F \o G$ is a $(h_{\seq s},\rh)$ -almost analytic extension of $f \o g$.

  The equivalence of the conditions \eqref{char3}--\eqref{char12} was proved in \cite{Rainer:aa};
  for partial results see also \cite[Lemma 1]{Peetre70}, \cite{FernandezGalbis06} and \cite{RainerSchindl14}.

  That \eqref{char1} implies \eqref{char7} and \eqref{char8} is a consequence of 
  \Cref{lem:m1} and \Cref{prop:strongmatrix}.

  The implications \eqref{char7} $\Rightarrow$ \eqref{char9} and \eqref{char8} $\Rightarrow$ \eqref{char10} follow from
  \Cref{thm:Rchar} and \Cref{thm:Bchar}, respectively. 

  The supplement follows from \Cref{thm:strongerconcave}.
\end{proof}

In the next theorem we make the connection to \Cref{PetzscheVogt} which is due to \cite{PetzscheVogt84}. 

\begin{theorem} \label{PetzscheVogtnew}	
Let $\om$ be a concave weight function satisfying $\om(t) = o(t)$ as $t \to \infty$. 
Let $U \subseteq \R^n$ be a bounded quasiconvex domain. 
Then:
	\begin{enumerate}
		\item $f \in \cB^{\{\om\}}(U)$ if and only if there exist $F \in C^1_c(\C^n)$ and $\rh>0$ such that $F|_U = f$ and 
		\begin{equation} \label{eq:condomstar}
			\sup_{z \in \C^n \setminus \ol U} |\ol \p F(z)| \exp( \rh \om^\star (d(z,\ol U)/\rh)) < \infty.
		\end{equation}
		\item $f \in \cB^{(\om)}(U)$ if and only if for all $\rh>0$ there exists 
		$F \in C^1_c(\C^n)$ such that $F|_U = f$ and \eqref{eq:condomstar}. 
	\end{enumerate}
If $\om$ is a strong weight function, then the extension $F$ in (2) may be taken independent of $\rh>0$.   	
\end{theorem}

\begin{proof}
	Let $\fW$ be the associated weight matrix of $\om$.
	For each $\seq M \in \fW$ there exists a constant $C\ge 1$ such that 
   \begin{equation} \label{eq:staromega}
   	  \om^\star(t) \le C \om_{\seq m}\Big(\frac{C}t\Big) \quad \text{ and } \quad
      \om_{\seq m}(t) \le C \om^\star \Big(\frac{1}{e Ct}\Big) + C
   \end{equation}
   for all $t>0$; see \cite[Corollary 3.11]{RainerSchindl17}. Here $\om_{\seq m}(t) = -\log h_{\seq m}(1/t)$, cf.\ \eqref{ep:omegam}. 
   By \Cref{cor:strongmatrix}, there is a regular weight matrix $\fS$ which is equivalent to $\fW$.
   Hence for each $\seq S \in \fS$ there exists $C\ge 1$ such that \eqref{eq:staromega} holds with $\om_{\seq m}$ 
   replaced by $\om_{\seq s}$. 
   In view of \Cref{thm:omegachar} the conclusion follows easily.	
\end{proof}

\section{The ultradifferentiable wave front set} \label{sec:WF}

In this section we define and study the wave front set for ultradifferentiable classes given by weight matrices.
This extends the results of H\"ormander \cite{H_rmander_1971} who considered only Roumieu classes defined by a single weight sequence. 
In particular we observe that our definition coincides with the one of Albanese--Jornet--Oliaro \cite{Albanese:2010vj} in the case that the classes are given by
a weight function.
We will follow primarily the presentation given in \cite[section 8.4-8.6]{Hoermander83I}.

In this section weight matrices are just assumed to be R- or B-semiregular. In \Cref{sec:WFreg}
below we will present stronger results for R- and B-regular matrices.

From now on $\Om$ denotes a non-empty open set in $\R^n$
and we shall write $\cE(\Om):=C^\infty(\Om)$ from time to time. 
We will use $D_j := - i \p_j$.

\subsection{The ultradifferentiable wave front set}

Our first preliminary result is the local characterization of ultradifferentiable functions by the Fourier transform. 

\begin{proposition} \label{FourierChar}
  Let $p_0\in\Omega$ and $u\in\D^\prime(\Omega)$.
\begin{enumerate}
\item If $\fM$ is an R-semiregular weight matrix,
   then $u\in\cE^{\{\fM\}}$ near $p_0$ if and only if for some neighborhood $V$ of $p_0$ there exist a bounded sequence
  $(u_N)_N\subseteq \cE^\prime(\Omega)$ with $u\vert_V=u_N\vert_V$ and
  some $\seq M\in\fM$ and $Q>0$ such that
  \begin{equation}\label{FourierCharacterization}
  \sup_{\substack{\xi\in\R^n\\ N\in\N}}
  \frac{|\xi|^N\bigl|\widehat{u}_N(\xi)\bigr|}{Q^N M_N}<\infty.
  \end{equation}
\item If $\fM$ is a B-semiregular weight matrix, then $u\in\cE^{(\fM)}$ near $p_0$ 
if and only if for some neighborhood $V$ of $p_0$ there exists a bounded sequence
$(u_N)_N\subseteq \cE^\prime(\Omega)$ with $u\vert_V=u_N\vert_V$ and
 such that \eqref{FourierCharacterization} holds for all $\seq M\in\fM$ and $Q>0$.
\end{enumerate}
\end{proposition}

\begin{proof}
	It suffices to slightly modify the proof of \cite[Proposition 8.4.2]{Hoermander83I}. 
	Fix $\seq M \in \fM$. Suppose that for some $r>0$ and some constants $C,h >0$
	\[
		|D^\al u(x)| \le C h^{|\al|} M_{|\al|}\quad \text{ for all } \al \text{ and } |x-x_0|<3r. 
	\]	
   There exist smooth cut-off functions $\ch_N$ with support in $|x-x_0|\le 2r$, equal $1$ when $|x-x_0|<r$, 
   and satisfying
   \begin{equation} \label{eq:testchi}
       |D^\al \ch_N| \le (C_1 N)^{|\al|}, \quad \text{ for } |\al| \le N;
    \end{equation} 
   cf.\ the proof of \cite[Proposition 8.4.2]{Hoermander83I}.    
   Then  
	the sequence $u_N := \ch_N u$ is bounded in $\cE'(\Om)$ and, thanks to \eqref{strictInclusion2}
  and \Cref{lem:basicM}(1), satisfies, for $|\al| = N$,
	\begin{align*}
	 	|D^\al u_N| 
  &\leq\sum_{\beta\leq\alpha}\binom{\alpha}{\beta}
  C_1^{|\beta|}N^{|\beta|}Ch^{| \alpha-\beta|}M_{|\alpha-\beta|}\\
  &\leq C\sum_{\beta\leq\alpha}\binom{\alpha}{\beta} C_h^{\tfrac{|\beta|}{N}}\bigl(C_1h\bigr)^{|\beta|}
  M_N^{\tfrac{|\beta|}{N}}h^{|\alpha-\beta|}
  M_N^{\tfrac{|\alpha-\beta|}{N}}
  \leq CC_h(C_2 h)^N M_N,
	 \end{align*} 
	 for some constant $C_2$. This easily implies \eqref{FourierCharacterization}.

	For the converse recall that, 
	since $(u_N)_N$ is bounded in $\cE^\prime(\Om)$, the Banach--Steinhaus theorem implies that there are constants $C,\mu>0$ such that
  \begin{equation} \label{BanachSteinhaus}
  |\widehat{u}_N(\xi)|\leq C\bigl(1+|\xi|\bigr)^\mu \quad \text{ for all } N.
  \end{equation}
  In $V$ we have $D^\alpha u(x)
  =(2\pi)^{-n}\int_{\R^n}\!e^{ix\xi}\xi^\alpha\widehat{u}_N(\xi)\,d\xi$
  for $N=|\alpha| +n+1$, since then \eqref{FourierCharacterization} implies that $\xi^\alpha\widehat{u}_N$ is integrable. 
  Estimating the integrals over $|\xi|\le Q \sqrt[N]{M_N}$ and  $|\xi|\ge Q \sqrt[N]{M_N}$ separately, using 
  \eqref{BanachSteinhaus} and \eqref{FourierCharacterization}, we conclude
  \begin{align*}
  | D^\alpha u(x)|
  &\leq C\Big(\Big(1+Q\sqrt[N]{M_N}\Big)^\mu\Big(Q\sqrt[N]{M_N}\Big)^{|\alpha|+n} 
  +Q^N M_N  \int_{Q\sqrt[N]{M_N}}^\infty\!\negthickspace t^{-2}\,dt\Big)\\
  &\leq CQ^{N-1}\Big(\Big(\sqrt[N]{M_N}\Big)^{|\alpha|+\mu+n}+
   M_N^{(N-1)/N}\Big)\\
&\leq CQ^{|\alpha|}\bigl(\sqrt[N]{M_N}\bigr)^{|\alpha| +n+\mu}
  \end{align*}
  where $C$ is a generic constant independent from $N$. 
  Repeated use of \Cref{def:Rregular}\eqref{R-Derivclosed1} or \Cref{def:Rregular}\eqref{B-Derivclosed1} shows $u \in \cE^{[\fM]}(V)$. 
\end{proof}

\begin{definition}\label{WF-Definition}
Let $\fM$ be a weight matrix.
Let $u\in\D^\prime(\Omega)$ and $(x_0,\xi_0)\in T^\ast\Omega\!\setminus\!\{0\}$.
\begin{enumerate}
  \item We say that $u$ is \emph{microlocally ultradifferentiable of class $\{\fM\}$} 
  at $(x_0,\xi_0)$ iff
  there exist a neighborhood $V$ of $x_0$, a conic neighborhood $\Gamma$ of $\xi_0$, 
  and a bounded sequence $(u_N)_N\subseteq\cE^\prime(\Omega)$ 
  with $u_N\vert_V=u\vert_V$  such that for some $\seq M\in\fM$ and a constant $Q>0$ we have
  \begin{equation}\label{WF-defining}
\sup_{\substack{\xi\in\Gamma\\ N\in\N}}\frac{|\xi|^N\bigl|\widehat{u}_N(\xi)\bigr|}{Q^N M_N}<\infty.
\end{equation}
  \item  $u$ is called \emph{microlocally ultradifferentiable of class $(\fM)$} at $(x_0,\xi_0)$ iff 
  there exist a neighborhood $V$ of $x_0$, a conic neighborhood $\Gamma$ of $\xi_0$, and  
  a bounded sequence $(u_N)_N\subseteq\cE^\prime(\Omega)$ with $u_N\vert_V=u\vert_V$
  such that  \eqref{WF-defining} is satisfied
  for all $\seq M\in\fM$ and all $Q>0$.
\end{enumerate}
The ultradifferentiable \emph{wave front set} $\WF_{[\fM]}u$ of $u$  
is the complement of the set of all $(x,\xi)\in T^\ast\Omega \setminus \{0\}$, where $u$ 
is microlocally ultradifferentiable of class $[\fM]$. 
For a weight function $\om$ and the associated weight matrix $\fW$ we set 
\[
  \WF_{[\om]} u := \WF_{[\fW]} u. 
\]
This coincides with the definition given in \cite{Albanese:2010vj} thanks to \Cref{representation}; see also 
\cite{RainerSchindl12}.
For the weight sequence $(k!)_k$ (resp.\ the weight function $t \mapsto t$) 
we get the analytic wave front set also denoted by $\WF_A u$.
\end{definition}

Notice that, in \Cref{WF-Definition}, $\fM$ is deliberately an \emph{arbitrary} weight matrix, since occasionally we want to compare 
$\WF_{\{\fM\}} u$ with $\WF_{(\fM)} u$. Most of the time we will assume semiregularity of the particular type. 

The distributions $u_N$ in \Cref{WF-Definition} can be chosen of the form
$\ch_Nu$ where $\ch_N$ is a bounded sequence of test functions as shown by the next lemma.

\begin{lemma}\label{WF-M Charakterisierung} 
Let $\fM$ be a weight matrix, $K\subseteq \Omega$ compact,
 $u\in\D^\prime(\Omega)$ of order $\mu$ in $K$,
and $F$ a closed cone.
\begin{enumerate}
	\item Suppose that 
$\WF_{\{\fM\}}u\cap (K\times F)=\emptyset$. If $\chi_N\in\D(K)$ and for each
$\alpha$ there exist $\seq M^\alpha\in\fM$ and $C_\alpha>0$ such that
\begin{equation}\label{TestfEst2}
\bigl| D^{\alpha+\beta}\chi_N\bigr|\leq C_\alpha
\big(C_\alpha \sqrt[N]{M^\alpha_N}\big)^{|\beta|},\qquad \vert\beta|\leq N=1,2,\dotsc,
\end{equation}
then $\chi_N u$ is bounded in $\cE^{\prime,\mu}$ and there are  $\seq M^\prime\in\fM$ and $Q,C>0$ such that
\begin{equation} \label{TestfEst3}
|\xi|^N \left|\widehat{\chi_N u}(\xi)\right|\leq C Q^N M^\prime_N,
\qquad N\in\N,\;\xi\in F.
\end{equation}
\item Suppose that $\WF_{(\fM)}u\cap (K\times F)=\emptyset$.
If $\chi_N\in\D(K)$ satisfies \eqref{TestfEst2} for some
totally ordered collection of positive sequences $\seq M^\al$ such that 
$\{\seq M^\al\} \{\lhd)\fM$,
then $\chi_Nu$ is bounded in $\cE^{\prime,\mu}$ and for all $\seq M^\prime\in\fM$ 
and all $Q>0$ there is a constant $C$ such that \eqref{TestfEst3} holds.
\end{enumerate}
\end{lemma}

It is not hard to see that there exist $\ch_N$ which satisfy \eqref{TestfEst2}; cf.\ \eqref{eq:testchi}. 
We emphasize that in (2) the sequences $\seq M^\al$ are not assumed to be weight sequences in the sense of 
\Cref{weights} (and do not belong to $\fM$).  

\begin{proof}
	The proof of (1) follows closely the arguments in \cite[Lemma 8.4.4]{Hoermander83I} with the 
   only difference that here we have to deal with more than just one weight sequence; we provide 
   details for later reference.

	The boundedness of $\chi_N u$ is evident.	
	Let $x_0\in K$, $\xi_0\in F\setminus\{0\}$ and choose $V$, $\Gamma$ and $u_N$ 
according to \Cref{WF-Definition}. 
Obviously, if $\supp\chi_N\subseteq V$, then $\chi_Nu=\chi_Nu_N$. 
By assumption, $u_N$ satisfies \eqref{BanachSteinhaus} and \eqref{WF-defining} in $\Gamma$ for some  $\seq M^\prime\in\fM$ and $Q>0$.
For convenience we set $\ell=\mu+n+1$.
Observe that, for $\et \in \R^n$ and $k \ge 0$,
\[
   |\et|^{\ell+k} \le  \Big(\sum_{j=1}^n |\et_j|\Big)^{\ell+k} = 
    \sum_{|\ga| = \ell +k} \binom{\ell +k}{\ga} |\et^\ga|.  
\]   
Together with \eqref{TestfEst2} we get, for $k \le N$,
\begin{align*}
   |\et|^{\ell+k} |\widehat \ch_N(\et)| 
   &\le 
    \sum_{|\ga| = \ell +k} \binom{\ell +k}{\ga} |\et^\ga \widehat \ch_N(\et)|
   \\
   &\le C(n,\ell)^k \sum_{|\al| \le \ell, |\be| = k} |\widehat {D^{\al+\be}\ch_N}(\et)|
   \\
   &\le C(n,\ell)^k \sum_{|\al| \le \ell, |\be| = k} C_\al (C_\al \sqrt[N]{M_N^\al})^{k}
   \\
   &\le C(n,\ell)^k M_N^{k/N}
\end{align*}
for some $\seq M \in \fM$.
This implies that, for all $N$,
\begin{equation} \label{H8.4.4}
\big|\widehat{\chi}_N(\eta)\big| \le C^{N+1} M_N \left(|\et| + \sqrt[N]{M_N}\right)^{-N} (1+|\eta|)^{-\mu-n-1}
\end{equation}
for some $C>0$.
We have
$\widehat{\chi_N u}(\xi)=(2\pi)^{-n}\int\!\widehat{\chi}_N(\eta)\widehat{u}_N(\xi-\eta)\,d\eta$.
Let $0<c<1$ and consider the integrals over $|\et| \le c |\xi|$ and $|\et| \ge c |\xi|$ separately.
Since $|\et| \ge c |\xi|$ implies
$|\xi-\eta|\leq(1+c^{-1})|\eta|$, we find with \eqref{BanachSteinhaus} (cf.\ \cite[(8.1.3)]{Hoermander83I}) 
\begin{equation} \label{eq:convolution}
|\widehat{\chi_N u}(\xi)| \le \lVert\widehat{\chi}_N\rVert_{L^1}
\sup_{|\xi -\eta|\leq c|\xi|}|\widehat{u}_N(\eta)|
 + 
C\left(1+c^{-1}\right)^\mu\int_{|\eta|\geq c|\xi|}\bigl|
\widehat{\chi}_N(\eta)\bigr|(1+|\eta|)^\mu\,d\eta.
\end{equation}
If $\xi_0\in\Gamma_1\subseteq\Gamma\cup\{0\}$ is a closed cone, then we can choose
$c$ such that $\eta\in\Gamma$ if $\xi\in\Gamma_1$ and $|\xi-\eta|\leq c|\xi|$. 
In this case $|\eta|\geq (1-c)|\xi|$.
Combining all this we obtain 
\begin{align*}
\sup_{\xi\in\Gamma_1}|\xi|^N\bigl|\widehat{\chi_Nu}(\xi)\bigr|
&\leq (1-c)^{-N} \lVert\widehat{\chi}_N\rVert_{L^1}\sup_{\eta\in\Gamma}
\bigl|\widehat{u}_N(\eta)\bigr||\eta|^N 
\\
&\quad +C\left(1+c^{-1}\right)^{\mu+N}\int\!(1+|\eta|)^\mu|\eta|^N
\bigl|\widehat{\chi}_N(\eta)\bigr|\,d\eta. 
\end{align*}
In view of \eqref{WF-defining} and \eqref{H8.4.4}
we have
\begin{equation*}
\sup_{\xi\in\Gamma_1}|\xi|^N\bigl|\widehat{\chi_Nu}(\xi)\bigr|
\leq Ch^NM^{\prime\prime}_N
\end{equation*}
for some $\seq M^{\prime\prime}\in\fM$ and some constants $C,h>0$.
Since $\xi\in F \setminus \{0\}$ was chosen arbitrarily, we see that $F$ can be covered by
a finite number of conic neighborhoods like $\Gamma_1$ and therefore 
\eqref{TestfEst3} is proven for $F$ and $\supp\chi_N\subseteq U$, where $U$ is a small enough neighborhood of $x_0$.
But $K$ is compact and $x_0$ was also chosen arbitrarily. Hence $K$ can be covered by 
finitely many sets $U_j$ in which \eqref{TestfEst3} holds.
Now let $\chi_N\in\D(K)$ satisfy \eqref{TestfEst2}.
 As in the proof \cite[Lemma 8.4.4]{Hoermander83I} we can choose a partition of unity 
 $\chi_{j,N}\in\D(U_j)$ for each $N$ and each $\chi_{j,N}$ satisfies \eqref{TestfEst2} 
 with $\seq M^\al$ independent of $j$. 
 Then \eqref{TestfEst2} holds also for $\lambda_{j,N}=\chi_{j,N}\chi_N$.
 The statement follows since $\sum_{j}\lambda_{j,N}=\chi_N$.

For part (2) observe that the proof of \eqref{H8.4.4} remains unchanged and then the 
condition $\{\seq M^\al\} \{\lhd)\fM$ easily implies the statement.
\end{proof}

The basic features of the ultradifferentiable wave front set are collected 
in the following proposition (cf.\ \cite{Hoermander83I} and \cite{Albanese:2010vj}).

\begin{proposition} \label{WFproperties} 
Let $\fM,\fN$ be weight matrices and $u\in\D^\prime(\Omega)$. Then:
\begin{enumerate}
  \item $\WF_{[\fM]} u$ is a closed and conic subset of $\COT$.
  \item $\WF_{\{\fM\}}u\subseteq\WF_{(\fM)}u$.
  \item $\WF u\subseteq\WF_{[\fN]}u\subseteq \WF_{[\fM]}u$ if $\fM[\preceq]\fN$.
  \item $\WF_{(\mathfrak{N})}u \subseteq \WF_{\{\fM\}}u$ if $\fM \{\lhd)\mathfrak{N}$.
  \item $(x,\xi)\notin\WF_{[\fM]}u\Leftrightarrow (x,-\xi)\notin\WF_{[\fM]}\ol{u}$.
  \item If $\fM$ is $[$semiregular$]$ then $\pi_1(\WF_{[\fM]}u)=\singsupp_{[\fM]}u$.
  \item 
  If $\fM$ is $[$semiregular$]$ then $\WF_{[\fM]}Pu\subseteq\WF_{[\fM]}u$ 
  for all linear partial differential operators $P$ with 
  $\cE^{[\fM]}$-coefficients.
\end{enumerate}
\end{proposition}

All these properties also hold for $\WF_{[\om]} u$, in particular,
$\WF_{[\om]} u \subseteq \WF_{[\si]} u$ if $\om(t) = O(\si(t))$ 
and  
$\WF_{(\om)} u \subseteq \WF_{\{\si\}} u$ if $\om(t) = o(\si(t))$ as $t \to \infty$.

\begin{proof}
  The proof of (1)--(5) is straightforward. 

  (6) If we use \Cref{FourierChar} and \Cref{WF-M Charakterisierung}, then 
  this follows along the lines of the proof of \cite[Theorem 8.4.5]{Hoermander83I}.
  
  (7)
  We first prove the Roumieu case.
  If $\fM$ is R-semiregular, then \Cref{def:Rregular}\eqref{R-Derivclosed1} 
  implies $\WF_{\{\fM\}}\partial_j u\subseteq\WF_{\{\fM\}}u$. 
  Hence it suffices to show
  that $\WF_{\{\fM\}} au \subseteq \WF_{\{\fM\}} u$, where $a\in\cE^{\{\fM\}}$. 
  If $(x_0,\xi_0)\notin\WF_{\{\fM\}} u$, then by (1) there are a compact neighborhood
  $K$ of $x_0$ and a closed conic neighborhood $\Ga$ of $\xi_0$ such that
  $(K\times\Gamma)\cap \WF_{\{\fM\}}u=\emptyset$.
  Suppose that $\chi_N\in\D(K)$ satisfies \eqref{TestfEst2}
  and let $\seq M\in\fM$ be such that $a\vert_K\in\cE^{\{\seq M \}}(K)$.   
  Observe that, by \Cref{def:Rregular}\eqref{R-Derivclosed1}, 
  for each $k$ there is $\seq M^{(k)} \in \fM$ such that $M_{k+j} \le C_k^{j+1} M^{(k)}_j$ for all $j$.
  Moreover, for each $\seq M \in \fM$, $M_k^{1/k}$ is increasing. 
  Thus, for $\lvert\beta\rvert\leq N$ and arbitrary $\alpha$, (the constants change from line to line)
  \begin{align*}
  \bigl\lvert D^{\alpha+\beta}\bigl(a\chi_N\bigr)\bigr\rvert
  &\leq 2^{\lvert\alpha\rvert+\lvert\beta\rvert}
  \sum_{\gamma\leq\alpha+\beta}\bigl\lvert  D^{\alpha+\beta-\gamma}
  a\bigr\rvert\bigl\lvert D^\gamma\chi_N\bigr\rvert\\
  &\leq c_\alpha^{\lvert\beta\rvert+1}\sum_{\ga = \ga'+\ga'', \gamma^\prime\leq\alpha,
  \gamma^{\prime\prime}\leq\beta}
\bigl\lvert  D^{\alpha+\beta-\gamma}
a\bigr\rvert\bigl\lvert D^{\gamma^\prime+\ga^{\prime\prime}}\chi_N\bigr\rvert\\
&\leq c_\alpha^{\lvert\beta\rvert+1} \sum_{\substack{\gamma^\prime\leq\alpha\\
		\gamma^{\prime\prime}\leq\beta}}
	h^{\lvert\al\rvert+\lvert\be\rvert-\lvert\ga\rvert}
	M_{\lvert\al\rvert+\lvert\be\rvert-\lvert\ga\rvert}
	\, C_{\ga^\prime}
\left(C_{\ga^\prime}\sqrt[N]{M_N^{\ga^\prime}}\right)^{\lvert\ga^{\prime\prime}\rvert}
\\
	&\leq c_{\alpha}^{\lvert\beta\rvert+1}\sum_{\substack{\gamma^\prime\leq\alpha\\
			\gamma^{\prime\prime}\leq\beta}}
	h_{\alpha}^{\lvert\beta\rvert-\lvert\gamma^{\prime\prime}\rvert+1}
	M^{(|\al|)}_{\lvert\beta-\ga^{\prime\prime}\vert}
	C_{\ga^\prime}\left(C_{\ga^\prime}\sqrt[N]{M_N^{\ga^\prime}}\right)^{\lvert\ga^{\prime\prime}\rvert}\\
	&\leq c_{\alpha}^{\lvert\beta\rvert+1} \left(M_N^\prime\right)^{\tfrac{\lvert\beta\rvert}{N}},
  \end{align*}
  where $\seq M^\prime=\max\{\seq M^{(|\al|)},\seq M^{\gamma'} : \gamma'\leq\alpha\}$.
  Therefore $\lambda_N=a\chi_N\in\D(K)$ also satisfies \eqref{TestfEst2}. 
  Hence \eqref{TestfEst3} holds for $\lambda_N u = \ch_N au$ and some
  $\seq M^{\prime\prime}\in\fM$, by \Cref{WF-M Charakterisierung}, 
  that is, $(x_0,\xi_0)\not\in\WF_{\{\fM\}} au$.

  Let us prove the Beurling case.  
  If $\fM$ is B-semiregular, then \Cref{def:Rregular}\eqref{B-Derivclosed1} 
  implies $\WF_{(\fM)}\partial_j u\subseteq\WF_{(\fM)}u$. 
  We claim that $\WF_{(\fM)} au \subseteq \WF_{(\fM)} u$ if $a\in\cE^{(\fM)}$. 
  If $(x_0,\xi_0)\notin\WF_{(\fM)} u$, then there are a compact neighborhood
  $K$ of $x_0$ and a closed conic neighborhood $\Ga$ of $\xi_0$ such that
  $(K\times\Gamma)\cap \WF_{(\fM)}u=\emptyset$.
  By semiregularity, we have $(N!)_N \{ \lhd )\fM$ and there exist $\ch_N \in \cD(K)$ 
  which satisfy \eqref{TestfEst2} for 
  $\sqrt[N]{M^\al_N}$ replaced by $N$. 
  Since $a \in \cB^{(\fM)}(K)$ we are in the situation of \Cref{lem:Bchi} below 
  which provides a collection of sequences suitable to perform the 
  above computation.  
  It follows that for each $\al$ there is a sequence $\seq L^\al \{\lhd) \fM$ such that 
  $\lambda_N=a\chi_N\in\D(K)$ satisfies
  \eqref{TestfEst2} (with $\seq L^\al$ instead of $\seq M^\al$).
  An analogous statement holds for
  the collection $\{\seq {\tilde L}^m\}_m$, where $\tilde L^m_k:= \max_{|\al| \le m}	L^\al_k$,
  which is totally ordered and satisfies $\seq L^m \{\lhd) \fM$ for all $m$. 
 Thus \Cref{WF-M Charakterisierung}(2) implies 
  the analogue of 
  \eqref{TestfEst3} for $\lambda_N u = \ch_N au$
  for all $\seq M \in \fM$ and $Q>0$.	
  Hence	$(x_0,\xi_0)\not\in\WF_{(\fM)} au$.    
\end{proof}

\begin{lemma} \label{lem:Bchi}
	Let $\fM$ be a B-regular weight matrix, and let $a \in \cB^{(\fM)}(K)$ for some compact $K \subseteq \R^n$.
	Then there exists a collection of positive sequences $\fL$ with the following properties:
	\begin{enumerate}
			\item For each $\seq L \in \fL$ there exists $\seq L' \in \fL$ such that $(L_{k+1}/L'_k)^{1/(k+1)}$ is bounded.
			\item $\fL \{\lhd) \fM$ and $a \in \cB^{\{\fL\}}(K)$.
			\item For each $\seq L \in \fL$ there exists a sequence $\seq {L''} \ge \seq L$ (not necessarily in $\fL$) such that 
			$\seq L'' \{\lhd ) \fM$ and $(L_k'')^{1/k}$ is increasing. Let $\fL'':= \{\seq L'': \seq L \in \fL\}$.
			\item If $\fF \subseteq \fL \cup \fL''$ is finite, then $\seq F := \max \fF$ defined by $F_k := \max_{\seq L \in \fF} L_k$  
			satisfies $\seq F  \{ \lhd ) \fM$. 
		\end{enumerate}	
\end{lemma}

\begin{proof}
	Let us define $\seq L$ by 
	\[
		L_k := \max \Big\{ \sup_{x \in K, |\al| \le k} |\p^\al a(x)|, k!\Big\}, \quad   k \in \N. 
	\]
	For $\nu \ge 1$ set $L^\nu_k := L_{k +\nu}$ and $\fL := \{\seq L^\nu\}_{\nu \in \N}$ with $\seq L^0:=\seq L$. 
	Then $\fL$ satisfies (1). 
	
	Clearly, $a \in \cB^{\{\seq L^0\}}(K)$. Let $\seq M \in \fM$ and $\nu \in \N$. 
	Since $\fM$ is B-regular, there exists $\seq M' \in \fM$ and $C>0$
	such that $M'_{k+\nu} \le C^{k+\nu} M_k$ for all $k$. Then  
	\[
		\frac{(L^\nu_k)^{1/k}}{M_k^{1/k}} = \frac{L_{k+\nu}^{1/k}}{M_k^{1/k}}  \le C^{1+\nu/k}  \frac{L_{k+\nu}^{1/k}}{(M'_{k+\nu})^{1/k}} 
		= C^{1+\nu/k}  \Big(\frac{L_{k+\nu}^{1/(k+\nu)}}{(M'_{k+\nu})^{1/(k+\nu)}} \Big)^{1 + \nu /k}
	\] 
	tends to $0$ as $k \to \infty$, since $\seq L \{ \lhd ) \fM$ by assumption. This implies (2).

	Given $\seq L \in \fL$ we define $\seq L''$ by setting $L''_0 := 1$ and
	\[
		(L''_k)^{1/k} := \max\{L_j^{1/j} : j \le k\}, \quad k \ge 1.
	\]
	Then $(L''_k)^{1/k}$ is increasing and $\seq L''\ge \seq L$. For $M \in \fM$ and $\ep >0$ there exists $j_0$ such that 
	$\frac{(L_j)^{1/j}}{M_j^{1/j}} < \ep$ for all $j \ge j_0$, since $L \lhd M$. Then,  for $k \ge j_0$,
	\begin{align*}
		\frac{(L''_k)^{1/k}}{M_k^{1/k}} 
		&= \max \Big\{\frac{\max\{L_j^{1/j} : j < j_0\}}{M_k^{1/k}}, \frac{\max\{L_j^{1/j} : j_0 \le j \le k\}}{M_k^{1/k}}\Big\}
		\\
		&\le \max \Big\{\frac{\max\{L_j^{1/j} : j < j_0\}}{M_k^{1/k}}, \max\Big\{\frac{L_j^{1/j}}{M_j^{1/j}} : j_0 \le j \le k\Big\}\Big\}
		\\
		&\le \max \Big\{\frac{\max\{L_j^{1/j} : j < j_0\}}{M_k^{1/k}}, \ep\Big\}
	\end{align*}
	which equals $\ep$ if $k$ is large enough, since $M_k^{1/k} \nearrow \infty$. This shows $\seq L'' \{\lhd) \fM$ and hence (3). 

	(4) follows easily from  $\fL \{\lhd) \fM$ and $\fL'' \{\lhd) \fM$.
\end{proof}

\begin{proposition}\label{WF-Intersection}
	Let $\fM$ be a weight matrix satisfying \Cref{def:Rregular}(0) and $u\in\D^\prime(\Omega)$.
\begin{enumerate}
\item We have 
\[
\WF_{\{\fM\}} u = \bigcap_{\seq M \in \fM} \WF_{\{\seq M\}} u 
\quad \text{ and } \quad
\WF_{(\fM)} u = \overline{\bigcup_{\seq M \in \fM} \WF_{(\seq M)} u}.
\]
\item If for all $\seq M \in \fM$ there is $\seq M' \in \fM$ such that $\seq M \lhd \seq M'$ then
\[
\WF_{\{\fM\}} u = \bigcap_{\seq M \in \fM} \WF_{(\seq M)} u. 
\]
\item If for all $\seq M \in \fM$ there is $\seq M' \in \fM$ such that $\seq M' \lhd \seq M$ then
\[
\WF_{(\fM)} u = \overline{\bigcup_{\seq M \in \fM} \WF_{\{\seq M\}} u}.
\]
\end{enumerate}
\end{proposition}

\begin{proof}
(1) The first identity is clear from the definition.
So is the inclusion 
$\overline{\bigcup_{\seq M \in \fM} \WF_{(\seq M)} u} \subseteq \WF_{(\fM)} u$, since the wave front set is closed. 
Now assume that $(x_0,\xi_0) \not\in \overline{\bigcup_{\seq M \in \fM} \WF_{(\seq M)} u}$. 
Then there exist a compact neighborhood $K$ of $x_0$ and a closed conic neighborhood $\Ga$ of $\xi_0$ such 
that 
\[
(K \times \Ga) \cap \overline{\bigcup_{\seq M \in \fM} \WF_{(\seq M)} u} = \emptyset 
\]
and hence $(K \times \Ga) \cap \WF_{(\seq M)} u = \emptyset$ for all $\seq M \in \fM$.
That $\fM$ satisfies \Cref{def:Rregular}(0) guarantees that 
$(N!)_N  \lhd  \seq M$ for all $\seq M \in \fM$.
Let $\ch_N \in \cD(K)$ satisfy \eqref{TestfEst2} for $\sqrt[N]{M^\al_N}$ replaced by $N$.
 Then, by \Cref{WF-M Charakterisierung},  
for all $\seq M \in \fM$ and all $Q>0$
\[
\sup_{\xi \in \Ga, N \in \N} \frac{|\xi|^N |\widehat { \ch_N u} (\xi) | }{Q^N M_N} < \infty,
\]
i.e., $(x_0,\xi_0) \not\in \WF_{(\fM)}u$.
This shows (1). Now (2) and (3) follow easily from (1) and \Cref{WFproperties}(2)\&(4).
\end{proof}

\subsection{Description of the wave front set by boundary values of holomorphic functions}

Let $\Gamma\subseteq\R^n$ be an open convex cone and  
set $\Gamma_r:=\{y\in\Gamma :  | y|<r\}$ for $r>0$.
A function $g\in C^1(\Omega\times\Gamma_r)$ is said to be of \emph{slow growth} if
there exist $c>0$ and $k\geq 0$ such that
\begin{equation*}
| g(x,y)|\leq c| y|^{-k}, \quad \text{ for } y \in \Ga_r. 
\end{equation*} 
If $g$ is of slow growth, then $\lim_{\Gamma_r\ni\ep\rightarrow 0} g(\cdot,\ep)$ exists in the sense 
of distributions.
We call this limit the \emph{boundary value $b_\Gamma g$} of $g$. 

Let us define
\[
  I(\xi) :=\int_{| \omega|=1}\! e^{-\omega\xi}\,d\omega.
\]
For $n = 1$ we have $I(\xi) = 2 \cosh \xi$ and, for $n>1$, $I(\xi) = I_0 (\langle \xi,\xi\rangle^{1/2})$, where
\begin{align*}
I_0(\rho) =c_{n-1}\int_{-1}^{1}\!\bigl(1-t^2\bigr)^{(n-3)/2}e^{t\rho}\,d\rho
\end{align*}
and $c_{n-1}$ denotes the area of $S^{n-2}$.
Finally, set
\[
  K(z):=(2\pi)^{-n}\int e^{iz\xi}/I(\xi)\,d\xi, \quad z \in X := \{z \in \C^n : |\Im z|< 1\}. 
\]
We recall the content of \cite[Lemma 8.4.9 and Lemma 8.4.10]{Hoermander83I}:
$I_0$ is an even entire function such that for every $\ep>0$ we have
\begin{equation}\label{Section1AuxEst}
I_0(\rho)=(2\pi)^{(n-1)/2} e^{\rho}\rho^{-(n-1)/2}\bigl(1+O(1/\rho)\bigr) \quad \text{ as } \rh \to \infty, |\arg \rho|<\pi/2-\ep.  
\end{equation}
There is a constant $C>0$ such that
\begin{equation*}
| I_0(\rho)|\leq C(1+|\rho|)^{-(n-1)/2}e^{|\real\rho|},\quad \text{ for } \rho\in\C.
\end{equation*}
The function $K$ is analytic in the connected open set
\begin{equation*}
 \tilde{X} :=\bigl\{z\in\C^n : \langle z,z\rangle\notin(-\infty,-1]\bigr\} \supseteq X.
\end{equation*}
For any closed cone $\Gamma\subseteq\tilde{X}$ such that $\langle z,z\rangle$ is never $\leq 0$ for $z\in\Gamma\!\setminus\!\{0\}$
there is some $c>0$ such that $K(z)=O(e^{-c| z|})$ as $z\rightarrow\infty$ in $\Gamma$.
We have for real $x$ and $y$
\begin{equation*}
| K(x+iy)|\leq K(iy)=(n-1)!(2\pi)^{-n}(1-| y|)^{-n}(1+O(1-| y|)),\quad \text{ as }  | y| \nearrow 1.
\end{equation*}

The following theorem is a generalization of \cite[Theorem 8.4.11]{Hoermander83I}.

\begin{theorem}\label{BasicDecomposition} \label{WFThmChar1}
If $u\in\S^\prime(\R^n)$ and $U=K\ast u$, then $U$ is analytic in 
$X$ and there exist $C,a,b$ such that
  \begin{equation}\label{convest}
  | U(z)|\leq C\bigl(1+| z|\bigr)^{a}\bigl(1-|\imag z|)^{-b},\quad z\in X.
  \end{equation} 
  The boundary values $U(\cdot+i\omega)$ are continuous functions of $\omega\in S^{n-1}$ with values in $\S^\prime(\R^n)$, and
  \begin{equation}\label{DefConv}
  \langle u,\varphi\rangle=\int_{S^{n-1}}\!\langle U(\cdot+i\omega),\varphi\rangle\,d\omega,\quad \varphi\in \S.
  \end{equation}
On the other hand, if $U$ is given satisfying \eqref{convest}, then the formula \eqref{DefConv} defines a distribution $u\in\S^\prime$ with $U=K\ast u$.

  For all
  $[$semiregular$]$ weight matrices $\fM$ we have
  \begin{equation*}
  \bigl(\R^n\times S^{n-1}\bigr)\cap \WF_{[\fM]} u= \bigl\{(x,\omega) : |\omega|=1,\, 
  U\text{ is not in }\cE^{[\fM]} \text{ at } x-i\omega\bigr\}.
  \end{equation*}
\end{theorem}

This follows from a straightforward modification of the proof in \cite{Hoermander83I}
using [semiregularity] of $\fM$.
The same applies to the following corollary.

\begin{corollary}\label{DecompCor}
Let $\Gamma^1,\dotsc\Gamma^N\subseteq\R^n\!\setminus\!\{0\}$ be closed cones such that $\bigcup_j\Gamma^j=\R^n\!\setminus\!\{0\}$.
Any $u\in\S^\prime(\R^n)$ can be written $u=\sum u_j$, where $u_j\in\S^\prime$ and
\begin{equation}\label{WFdecomp}
\WF_{[\fM]}u_j\subseteq\WF_{[\fM]}u\cap\bigl(\R^n\times\Gamma^j\bigr).
\end{equation}
If $u=\sum u^\prime_j$ is another such decomposition, then $u^\prime_j=u_j+\sum_k u_{jk}$, where $u_{jk}\in\S^\prime$, $u_{jk}=-u_{kj}$ and
\begin{equation*}
\WF_{[\fM]}u_{jk}\subseteq \WF_{[\fM]}u\cap\bigl(\R^n\times\bigl(\Gamma^j\cap\Gamma^k\bigr)\bigr).
\end{equation*}
\end{corollary}

The next theorem generalizes \cite[Theorem 8.4.15]{Hoermander83I}; it suffices to follow the 
arguments in \cite{Hoermander83I}; recall that $\Ga^\circ := \{\xi \in \R^n : \langle y, \xi \rangle \ge 0 \text{ for all } y \in \Ga\}$ 
denotes the dual cone of $\Ga$.

\begin{theorem}\label{weakBound2}
Let $\fM$ be a $[$semiregular$]$ weight matrix. 
Let $\Gamma\subseteq\R^n\!\setminus\!\{0\}$ be an open convex cone and $u\in\D^\prime(\Omega)$ such that $\WF_{[\fM]}u\subseteq \Omega\setminus\Gamma^\circ$.
If $V\Subset\Omega$ and $\Gamma^\prime \subset \Gamma$ is an open convex cone with closure in $\Ga \cup \{0\}$, 
then there is a function $F$ holomorphic in $V+i\Gamma_\gamma^\prime$ of slow growth and $u\vert_V-b_{\Gamma^\prime}F \in \cE^{[\fM]}(V)$.
\end{theorem}

Combining \Cref{weakBound2} with \Cref{DecompCor} and \cite[Theorem 8.4.8]{Hoermander83I} yields:

\begin{corollary}\label{SemiBVChar} 
Let $\fM$ be a $[$semiregular$]$ weight matrix.
Let $u\in\D^\prime(\Omega)$ and $(x_0,\xi_0)\in T^\ast\Omega \setminus \{0\}$.
Then $(x_0,\xi_0)\notin\WF_{[\fM]} u$ if and only if there exist a neighborhood $V$ of $x_0$, $f\in\cE^{[\fM]}(V)$,
open cones $\Gamma^1,\dotsc,\Gamma^q$ with the property $\xi_0\Gamma^j<0$ for all $j$, and holomorphic functions
$F_j\in\O(\Omega+i\Gamma^j_{\delta})$ of slow growth such that
\begin{equation*}
u\vert_V=f+\sum_{j=1}^qb_{\Gamma^j}F_j.
\end{equation*}
\end{corollary}

Since $\cE^{[\fM]}$ is stable by pullback with real analytic mappings, see \Cref{AnalyticClosed}, 
we can follow the proof of \cite[Theorem 8.5.1]{Hoermander83I} to obtain the following statement.

\begin{theorem}\label{Invar1}
Let $\fM$ be a $[$semiregular$]$ weight matrix.
Let $F:\Omega_1\rightarrow\Omega_2$ be a real analytic mapping, where $\Om_i \in \R^{n_i}$ are open.
If $u\in\D^\prime(\Omega_2)$ and $N_F\cap\WF_{[\fM]}u=\emptyset$, then
\begin{equation*}
\WF_{[\fM]}F^\ast u\subseteq F^\ast \WF_{[\fM]}u. 
\end{equation*}
Here $N_F = \{ (f(x),\et) \in \Om_2 \times \R^{n_2} : F'(x)^T \et = 0 \}$ is the set of normals of $F$.
\end{theorem}

\begin{remark}
If the map $F$ in \Cref{Invar1} is a real analytic diffeomorphism  
then for all distributions 
$u\in\D^\prime(\Omega_2)$
\begin{equation*}
\WF_{[\fM]}F^\ast u=F^\ast \WF_{[\fM]}u.
\end{equation*}
Hence the ultradifferentiable wave front set can be defined for distributions on real analytic manifolds.
\end{remark}

The following result can be proved in analogy to \cite[Theorems 8.5.4 and 8.5.4']{Hoermander83I}.

\begin{theorem}
Let $\fM$ be a $[$semiregular$]$ weight matrix.
Let $X\subseteq\R^n$ and $Y\subseteq\R^m$ be open sets and $K\in\D^\prime(X\times Y)$ be a distribution such that the projection $\supp K\rightarrow X$ is proper. If $u\in\cE^{[\fM]}(Y)$ then
\begin{equation*}
\WF_{[\fM]}\mathcal{K}u\subseteq\bigl\{(x,\xi)\in X\times\R^n\!\setminus\!\{0\} : (x,y,\xi,0)\in\WF_{[\fM]}(K)
\text{ for some }y\in \supp u\bigr\},
\end{equation*}
where $\mathcal{K}$ is the linear operator with kernel $K$.
\end{theorem}

\subsection{Toward a quasianalytic Holmgren uniqueness theorem}

We want to close this section with the proof of a generalization of \cite[Theorem 7.1]{H_rmander_1993} which will be 
needed for a version of the Holmgren uniqueness theorem in \Cref{thm:Holmgren}.

\begin{proposition}\label{1dimUniq}
Let $\fM$ be a quasianalytic R-semiregular weight matrix.
Let $u\in\D^\prime(I)$ be a distribution on an open interval $I$ of $\R$.
If $x_0\in I$ is a boundary point of $\supp u$, then $(x_0,\pm 1)\in\WF_{\{\fM\}}u$.
\end{proposition}

Since $\WF_{\{\fM\}}u\subseteq \WF_{(\fM)}u $, by \Cref{WFproperties}, only the Roumieu case is interesting.

\begin{proof}
  By \Cref{weakBound2}, we have a decomposition $u = u_+ + u_-$, where $u_+ \in \cE^{\{\fM\}}$. 
      Set $v_\pm := u_\pm \circ f$ with $f(x):= \frac{\delta x}{\sqrt{1+x^2}}$ and $\de>0$. 
      By  
      \Cref{AnalyticClosed}, $v_{+} \in \cE^{\{\fM\}}$, i.e.
      \begin{gather}
      \label{vderest}
      |v^{(j)}_+(x)|\le C Q^j M_j, \quad \text{ for all } j \in \N, x \in \R, 
      \end{gather}
      for some $C,Q>0$ and some $\seq M \in \fM$. 
      Now it suffices to follow the arguments in the proof of \cite[Theorem 6.1]{H_rmander_1993} 
      which show that the weight sequence $\seq M$ is non-quasianalytic. 
      (These arguments do not require that $\seq M$ is derivation closed.) 
\end{proof}

A straightforward modification of the proof 
of \cite[Theorem 8.5.6]{Hoermander83I} yields the following version in several variables.

\begin{theorem}\label{ndimUniq}
Let $\fM$ be a quasianalytic R-semiregular weight matrix. 
Let $u\in\D^\prime(\Omega)$ 
and let $F : \Om \to \R$ be real analytic. If $x_0\in\supp u$ is
such that $dF(x_0)\neq 0$ and $F(x)\leq F(x_0)$ for all $x\in\supp u$, then
$(x_0,\pm dF(x_0))\in\WF_{\{\fM\}} u$.
\end{theorem}

\section{Characterization of the ultradifferentiable wave front set} \label{sec:WFreg}

In this section all weight matrices are [regular]. We need 
a microlocalized version of the almost analytic extension.
Now we say that a smooth function $F \in  \cE(\Omega\times\Gamma_r)$ is \emph{$(h,Q)$-almost analytic} 
if there is a constant $C\ge 1$ such that
\begin{equation}
\biggl|\frac{\partial F}{\partial \bar{z}_j}(x,y)\biggr| 
\leq C\,h (Q | y|) \qquad (x,y)\in \Omega\times\Gamma_{r},
\;j=1,\dotsc,n,
\end{equation}
where $z_j = x_j + i y_j$. Let $\Om \subseteq \R^n$ be a bounded open set. 

\begin{definition}
Let $\fM$ be a weight matrix. 
Let $u\in\D^\prime(\Omega)$ and $\Gamma \subseteq \R^n$ an open convex cone. We say that
\begin{enumerate}
\item $u$ is \emph{$\{\fM\}$-almost analytically extendable into $\Gamma$}
if there exist $\seq M\in\fM$, $Q>0$, $r>0$, and 
an $(h_{\seq m},Q)$-almost analytic function $F\in\cE(\Omega\times\Gamma_r)$ of slow growth 
such that $u=b_\Gamma F$.
\item $u$ is \emph{$(\fM)$-almost analytically extendable into $\Gamma$} 
if for all $\seq M\in\fM$ and all $Q>0$
there exist $r>0$ 
and 
an $(h_{\seq m},Q)$-almost analytic function $F\in\cE(\Omega\times\Gamma_r)$ of slow growth 
such that $u=b_\Gamma F$.              
\end{enumerate}
\end{definition}

\subsection{Almost analytic description of the ultradifferentiable wave front set}

\begin{theorem}\label{Theorem-M-BVWF}
Let $\fM$ be a $[$semiregular$]$ weight matrix.
If $u\in\D^\prime(\Omega)$ is 
$[\fM]$-almost analytically extendable into $\Gamma$, then
\begin{equation}
\WF_{[\fM]}u \subseteq \Omega\times\Gamma^\circ\!\setminus\{0\}.
\end{equation}
\end{theorem}

\begin{proof}
Assume that $u = b_\Ga F$, where $F\in\cE(\Omega\times\Gamma_r)$ is an $(h_{\seq m},Q)$-almost analytic function 
of slow growth, i.e., there exist $c,k>0$ such that
\begin{equation*}
| F(x,y)|\leq c| y|^{-k},\quad  x \in \Om, y \in \Ga_r. 
\end{equation*}
Let $Y_0 \in \Ga$ and
let $(x_0,\xi_0)\in T^\ast\Omega \setminus \{0\}$ with $Y_0\xi_0<0$.
Choose bounded neighborhoods $V_1$ and $V_2$ of $x_0$ such that $\overline{V_1}\subseteq V_2$
and a sequence $(\varphi_N)_N\subseteq\D(\Omega)$ such that $\supp\varphi_N\subseteq V_2$, 
$\varphi_N\vert_{V_1}= 1$,
and
\begin{equation}\label{Estimate}
\bigl|\partial^{\alpha}\varphi_N(x)|\leq Q_1^{|\alpha|}(N+1)^{|\alpha|},
\quad \text{ for } |\alpha|\leq N+1, 
\end{equation}
where $Q_1$ is a constant independent of $N$.
We set 
\begin{equation*}
\Phi_N(x,y)=\sum_{|\alpha|\leq N}\frac{\partial^\alpha\varphi_N}{\partial x^\alpha}(x)\frac{(iy)^\alpha}{\alpha !},
\quad \text{ for } N \ge k,
\end{equation*}
and recall from \cite[8.4.8]{Hoermander83I} that the estimate \eqref{Estimate} yields
\begin{equation}\label{EstimateBeta}
\left|\sum_{|\alpha|=\mu}\frac{\partial^\alpha\varphi_N}{\partial x^\alpha}\frac{(iY_0)^\alpha}{\alpha !}\right|
\leq Q_1^\mu\frac{| Y_0|_1^\mu}{\mu !}(N+1)^\mu,
\quad \text{ for } 0\leq\mu\leq N+1.
\end{equation}
Here $| Y_0|_1=\sum_{j=0}^n | Y_{0,j}|$.
For $N\geq k$ we have (see e.g.\ \cite{Fuerdoes1})
\begin{align*}
\widehat{\varphi_Nu}(\xi) 
&=\int_\Omega  F(x,Y_0)e^{-i\langle x+iY_0,\xi\rangle} \Phi_N (x,Y_0)\,dx
\\
&\quad + 2i  \int_{\Omega} \int_0^1  \bigl\langle \ol{\partial} F(x,\tau Y_0),Y_0\bigr\rangle e^{-i\langle x+i\tau Y_0,\xi\rangle}
\Phi_N(x,\tau Y_0)\,d\tau dx\\
&\quad+(N+1) \int_{\Omega} \int_0^1 F(x,\tau Y_0)\tau^Ne^{-i\langle x+i\tau Y_0,\xi\rangle}
\sum_{|\alpha|=N+1}\partial^\alpha \varphi_N(x)
\frac{(iY_0)^\alpha}{\alpha !}\,d\tau dx.
\end{align*}
If $Y_0\xi<0$ we know from \cite[p.\ 285]{Hoermander83I} that the first and third integral above can be estimated by 
\begin{equation*}
Q_1^{N+1}\Bigl(e^{Y_0\xi}+(N-k)! (-Y_0\xi)^{k-N-1}\Bigr).
\end{equation*}
Since $F$ is $(h_{\seq m},Q)$-almost analytic, the second integral is estimated by (cf.\ \cite{Fuerdoes1}) 
\begin{equation*}
Q_1^{N+1}Q^{N-k}\int_0^1 m_{N-k}\tau^{N-k}e^{\ta Y_0\xi}\,d\tau \le Q_1^{N+1}Q^{N-k}m_{N-k}(N-k)!(-Y_0\xi)^{k-N-1},
\end{equation*}
where $Q_1$ is a suitable constant.
We set $u_N=\varphi_{k+N-1}u$ and observe
that there are an open conic neighborhood  $V$ of $\xi_0$ and a constant $\gamma>0$ such that 
$Y_0\xi\leq -\gamma|\xi|$ for all $\xi\in V$. For such $\xi$ we conclude (using $e^{-\ga |\xi|} \le N!(\ga |\xi|)^{-N}$)
\begin{align*}
\bigl|\widehat{u}_N(\xi)\bigr|&\leq C\left(Q_1^{N}\big(e^{-\gamma|\xi|}+(N-1)!|\xi|^{-N}\big)+(Q_1Q)^NM_{N-1}|\xi|^{-N}\right)\\
&\leq CQ_1^N \bigl( N! \gamma^{-N}+(N-1)!+Q^NM_{N-1}\bigr)|\xi|^{-N}\\
&\leq C(Q_1Q)^NM_N|\xi|^{-N},
\end{align*}
by \eqref{strictInclusion2}. This shows that $(x_0,\xi_0)\notin\WF_{[\fM]}u$. 

Since $Y_0\in\Gamma$ was chosen arbitrarily the statement of the theorem follows.
\end{proof}

Combining \Cref{Theorem-M-BVWF} with \Cref{thm:Rchar}, \Cref{thm:Bchar}, and \Cref{SemiBVChar} 
we obtain the following characterization of the ultradifferentiable wave front set.

\begin{corollary}\label{RegWFLocalChar2}
Let $\fM$ be a $[$regular$]$ weight matrix.
Let $u\in\D^\prime(\Omega)$ and $(x_0,\xi_0)\in T^\ast\Omega\!\setminus\!\{0\}$. Then
$(x_0,\xi_0)\notin\WF_{[\fM]}u$ if and only if there are open convex cones $\Gamma^1,\dotsc,\Gamma^d$ 
with $\xi_0\Gamma^j<0$,
an open neighborhood $V$ of $x_0$ and distributions $u_j\in\D^\prime(V)$ such that $u_j$ 
is $[\fM]$-almost analytically extendable into $\Gamma^j$ for $j=1,\dotsc, d$ and 
\begin{equation*}
u\vert_V=\sum_{j=1}^du_j.
\end{equation*}
\end{corollary}

\subsection{Invariance by pullback with ultradifferentiable mappings}

We are now ready to show that the ultradifferentiable wave front set is compatible with 
the pullback by ultradifferentiable mappings. As a consequence the ultradifferentiable wave front set
can be defined for distributions on ultradifferentiable manifolds.

\begin{theorem} \label{thm:pullback}
Let $\fM$ be a $[$regular$]$ weight matrix.
Let $F:\Omega_1\rightarrow \Omega_2$ be an $\cE^{[\fM]}$-mapping. 
If $u\in\D^\prime(\Omega_2)$ and $\WF_{[\fM]}u\cap N_F=\emptyset$ then
\begin{equation*}
\WF_{[\fM]}F^\ast u\subseteq F^\ast \WF_{[\fM]} u.
\end{equation*}
Here $N_F = \{ (F(x),\et) \in \Om_2 \times \R^{n_2} : F'(x)^T \et = 0 \}$ is the set of normals of $F$.
\end{theorem}

\begin{proof} 
First assume that $u$ is $[\fM]$-almost analytically extendable into an open convex cone $\Gamma$.
By \Cref{Theorem-M-BVWF}, $\WF_{[\fM]} u \subseteq\Omega\times\Gamma^\circ \setminus \{0\}$. 
Since $\WF_{[\fM]}u\cap N_F=\emptyset$, we have 
$F'(x)^T\eta\neq 0$ for all $\et \in \Gamma^\circ \setminus \{0\}$.
Hence $F'(x)^T \Gamma^\circ$ is a closed convex cone for all $x\in\Omega_1$.
We claim that for $x_0\in\Omega_1$ we have
\begin{equation} \label{claim}
\WF_{[\fM]} (F^{\ast}u)\vert_{x_0}\subseteq \bigl\{(x_0,F'(x_0)^T\eta):\; \eta\in\Gamma^\circ\!\setminus\{0\}\!\bigr\}.
\end{equation}
We can write (see \cite[page 296]{Hoermander83I})
\begin{equation*}
F'(x_0)^T\Gamma^\circ =\bigl\{\xi\in\R^n :\langle h,\xi\rangle\geq 0 \text{ if } F^\prime(x_0)h\in\Gamma\bigr\}.
\end{equation*}
Let $\Phi \in \cE(\Omega_2\times\Gamma_r)$ be an $(h_{\seq m},Q)$-almost analytic function such that $u=b_\Gamma\Phi$.
Let $X_1\subseteq \Omega_1$ be a relatively compact quasiconvex neighborhood of $x_0$ and 
denote by $\widetilde{F}\in\cE(X_1\times\R^n,\Om_2 \times \R^n)$ an $(h_{\seq n},\rho)$-almost analytic extension of $F$, which exists by 
\Cref{thm:Rchar} and \Cref{thm:Bchar}.
Since $h_{\seq m} \le h_{\seq n}$ if ${\seq M} \le \seq N$ and since $h_{\seq m}$ is increasing,
we can assume that $\seq M=\seq N$ and $Q=\rho$.

Let $h \in \R^n$ and $F^\prime(x_0)h\in\Gamma$.
Then 
\begin{equation*}
\imag \widetilde{F}(x+i\ep h)\in \Gamma\quad \text{ for small } \ep >0 \text{ if } x \in X_0, 
\end{equation*}
where $X_0$ is a small neighborhood of $x_0$.

From the proof of the existence of the boundary value of an almost analytic function (see e.g.\ \cite{Fuerdoes1}, for the special case
of boundary values of holomorphic functions see \cite{Hoermander83I})
we observe that the map
\begin{equation*}
\R_{\geq 0}\times\bigl(\Gamma\cup\{0\}\bigr)\ni(\ep, y)\longmapsto \widetilde{\Phi}(\ep,y)
:=\Phi\bigl(\widetilde{F}(\cdot+i\ep h)+iy\bigr)\in\D^{\prime}(X_0)
\end{equation*}
is continuous. Now 
\begin{gather*}
  \widetilde{\Phi}(\ep,y)\stackrel{\ep \to 0}{\longrightarrow} \widetilde{\Phi}(0,y)=\Phi(\widetilde{F}(\cdot+0i)+iy)
  \stackrel{y \to 0}{\longrightarrow} F^* u \quad \text{ and } 
  \\
  \widetilde{\Phi}(\ep,y)\stackrel{y \to 0}{\longrightarrow} \widetilde{\Phi}(\ep,0)=\Phi(\widetilde{F}(\cdot+i\ep h))
  \quad \text{ in } \cD'(X_0). 
\end{gather*} 
 Hence by continuity
 \begin{equation*}
 F^\ast u=\lim_{\ep\rightarrow 0} \Phi\bigl(\widetilde{F}(\cdot+i\ep h)\bigr) \quad \text{ in } \cD'(X_0).
 \end{equation*}
 Now $\Phi\circ\widetilde{F}$ is $(h_{\seq m},CQ)$-almost analytic, where the composition is defined 
 and $C$ is the Lipschitz constant of $\widetilde{F}$ (cf.\ \Cref{prop:composition}).
Thus the proof of \Cref{Theorem-M-BVWF} 
implies
 \begin{equation*}
 \WF_{[\fM]} F^\ast u\vert_{x_0}\subseteq \{(x_0,\xi) : \langle h,\xi\rangle\geq 0\}.
 \end{equation*}
This proves \eqref{claim}.

Now suppose that $(F(x_0),\eta_0)\notin\WF_{[\fM]}u$.
By \Cref{RegWFLocalChar2} there are an open neighborhood $V$ of $x_0$, distributions $u_1,\dotsc, u_d\in\D^\prime (V)$
and open convex cones $\Gamma_1,\dotsc,\Gamma_d$ such that $\eta_0\Gamma_j<0$ and $u_j$ is 
$[\fM]$-almost analytically extendable into $\Gamma_j$ for all $j=1,\dotsc,d$ and
\begin{equation*}
u\vert_V=\sum_{j=1}^d u_j.
\end{equation*}
By assumption, $F'(x)^T\eta\neq 0$ when $(F(x),\eta)\in\WF_{[\fM]} u$ for $x\in F^{-1}(V)$.
Hence we can assume that $F'(x)^T\eta\neq 0$ for $\eta\in\Gamma_j^\circ \setminus \{0\}$ for all $j=1,\dotsc, d$ 
and $x\in F^{-1}(V)$, since in the proof of \Cref{RegWFLocalChar2} the cones 
$\Gamma_j$ can be chosen such that the set $\Gamma^\circ_j\cap S^{n-1}$ 
has small measure and $\Gamma^\circ_j\cap\WF_{[\fM]}u\vert_{F(x)}\neq\emptyset$ for $x\in V$.
By the arguments above we have for a smaller neighborhood $V_0$ of $x_0$ that
\begin{equation*}
F^{\ast}u\vert_{V_0}=\sum_{j=1}^N{F^\ast u_j}\vert_{V_0}
\end{equation*}
and 
$\WF_{[\fM]} (F^\ast u_j)\vert_{x_0}\subseteq \bigl\{(x_0,F'(x_0)^T\eta) : \eta\in\Gamma_j^\circ\!\setminus\!\{0\}\bigr\}$
for all $j=1,\dotsc,d$.
However, since $\eta_0\Gamma_j<0$ it follows that $(x_0,F'(x_0)^T\eta_0) \notin \WF_{[\fM]} (F^\ast u_j)$ 
and therefore $(x_0, F'(x_0)^T\eta_0) \notin \WF_{[\fM]}(F^\ast u)$.
\end{proof}

\begin{remark}\label{WeightDiffeoWF}
If the mapping $F$ in \Cref{thm:pullback}
is a diffeomorphism of class $\cE^{[\fM]}$, then
\begin{equation*}
\WF_{[\fM]} F^\ast u=F^\ast \WF_{[\fM]}u, \quad u \in \cD'(\Om).
\end{equation*}
Hence the ultradifferentiable wave front set $\WF_{[\fM]}u$ can be defined for distributions on
ultradifferentiable manifolds of 
class $\cE^{[\fM]}$.  
\end{remark}

\subsection{An ultradifferentiable version of Bony's theorem} \label{sec:Bony}
 
Bony \cite{MR0650834} showed that the analytic wave front can be described either by the Fourier transform, by holomorphic extensions,
or by the FBI transform. The latter can be viewed as a nonlinear version of the Fourier transform and was introduced by \cite{MR0399494}.

We use here the generalized FBI transform defined by \cite{Berhanu:2012aa} as
\begin{equation*}
\mathfrak{F}u(t,\xi)=c_p\bigl\langle u(x),e^{i\xi(t-x)}e^{-|\xi| p(t-x)}\bigr\rangle,
\quad u\in\cE^\prime(\Omega), 
\end{equation*}
where $p$ is a real homogeneous positive elliptic polynomial of degree $2k$ and $c_p^{-1}:=\int e^{-p(x)}dx$,
i.e., $c |x|^{2k} \le p(x) \le C |x|^{2k}$ for constants $0<c<C$.

\begin{theorem} \label{Bony}
Let $\fM$ be a $[$regular$]$ weight matrix.
Let $u\in\D^\prime(\Omega)$ and $(x_0,\xi_0)\in T^\ast\Omega\!\setminus\!\{0\}$. 
Then 
\begin{enumerate} 
\item $(x_0,\xi_0)\notin\WF_{\{\fM\}} u$ if and only if there exist a test function $\psi\in\D(\Omega)$ with $\psi\equiv 1$ 
near $x_0$, a conic neighborhood $U\times\Gamma$ of $(x_0,\xi_0)$, a weight sequence $\seq M\in\fM$, and a constant $\gamma>0$ such that 
\begin{equation}\label{M-FBIestimate}
\sup_{(t,\xi)\in U\times\Gamma}e^{\omega_{\seq M}(\gamma|\xi|)}
\bigl|\mathfrak{F}(\psi u)(t,\xi) \bigr|<\infty.
\end{equation}
\item $(x_0,\xi_0)\notin\WF_{(\fM)} u$ if and only if there exist a test function $\psi\in\D(\Omega)$ with $\psi\equiv 1$ near $x_0$, 
a conic neighborhood $U\times\Gamma$ of $(x_0,\xi_0)$ such that
\eqref{M-FBIestimate} is satisfied for all weight sequences $\seq M\in\fM$ and all $\gamma>0$.
\end{enumerate}
\end{theorem}

Note that \Cref{Bony:intro} is a direct consequence, since a weight function $\om$ 
and the associated weight matrix $\fW = \{\seq W^x\}_{x>0}$ satisfy
\[
   \A x>0 \E C_x>0 \A t >0 : x\, \om_{\seq W^x}(t) \le \om(t) \le 2 x\, \om_{\seq W^x}(t) + C_x,
\]  
see \cite[Lemma 2.5]{JSS19} and \cite[Lemma 5.7]{RainerSchindl12}, 
and $\om$ and all $\om_{\seq W^x}$ satisfy \eqref{om1}.

\begin{proof}
First let $(x_0,\xi_0)\notin\WF_{[\fM]}u$. 
W.l.o.g.\ we can assume that $x_0=0$. 

Suppose that $u$ is locally the boundary value of an $(h_{\seq m},\rho)$-almost
analytic function $F\in\cE(V\times\Gamma_\delta)$, i.e.\ $u\vert_V=b_\Gamma F$, where $V$ is a neighborhood of the origin and 
$\xi_0\Gamma<0$ is an open convex cone. 
We assume that this holds either for some $\seq M \in \fM$ and some $\rh>0$ or for all $\seq M \in \fM$ and all $\rh>0$, 
depending on the case we treat.
We will show that this implies \eqref{M-FBIestimate} for the same $\seq M$ and either some $\ga>0$ or all $\ga>0$, respectively. 
By \Cref{RegWFLocalChar2}, one direction of the theorem follows. 

Choose $r>0$ such that $B_{2r} = \{x : |x| <2r\} \Subset V$ and let $\psi\in\D(B_{2r})$ be such that $\psi\vert_{B_r}\equiv 1$. 
Take $v\in\Gamma_\delta$ and define
\begin{equation*}
Q(t,\xi,x):=i\xi(t-x)-|\xi| p(t-x).
\end{equation*}
Then
\begin{equation} \label{eq:FBI}
\mathfrak{F}(\psi u)(t,\xi)=\lim_{\tau\rightarrow 0+}\int_{B_{2r}} e^{Q(t,\xi,x+i\tau v)}\psi(x)F(x+i\tau v)\,dx.
\end{equation}
As in the proof of \cite[Theorem 4.2]{Berhanu:2012aa} we  put $z=x+iy$, $\psi(z)=\psi(x)$, and
\begin{equation*}
D_\tau :=\bigl\{ x+i\sigma v\in\C^n : x\in B_{2r},\; \tau\leq\sigma\leq\lambda\bigr\},
\end{equation*}
for some $\lambda>0$ to be determined later,
 and consider the $n$-form
\begin{equation*}
e^{Q(t,\xi,z)}\psi(z)F(z)\,dz_1\wedge \dots\wedge dz_n.
\end{equation*}
Stokes' theorem implies
\begin{align}\label{FBI-neccesity1}
&\int_{B_{2r}} e^{Q(t,\xi,x+i\tau v)}\psi(x)F(x+i\tau v)\,dx 
\nonumber \\
&=\int_{B_{2r}} e^{Q(t,\xi,x+i\lambda v)}\psi(x)F(x+i\lambda v)\,dx 
\nonumber \\
&\quad +\sum_{j=1}^n\int_{D_\tau}  e^{Q(t,\xi,z)}\frac{\partial}{\partial \bar{z}_j}\bigl(\psi(z)F(z)\bigr)
\,d\bar{z}_j\wedge dz_1\wedge\dots\wedge dz_n 
\nonumber \\
&=\int_{B_{2r}} e^{Q(t,\xi,x+i\lambda v)}\psi(x)F(x+i\lambda v)\,dx
\nonumber \\
&\quad+\sum_{j=1}^n\int_{B_{2r}} \int_\tau^\lambda  e^{Q(t,\xi,x+i\sigma v)}
\frac{\partial\psi}{\partial \bar{z}_j}(x+i\sigma v)F(x+i\sigma v)\,d\sigma dx
\nonumber \\
&\quad+\sum_{j=1}^n\int_{B_{2r}} \int_\tau^\lambda  e^{Q(t,\xi,x+i\sigma v)}\psi(x+i\sigma v)
\frac{\partial F}{\partial \bar{z}_j}(x+i\sigma v)\,d\sigma dx.
\nonumber \\
&=: I_1 + I_2 + I_3.
\end{align}
Since $\xi_0 v<0$ there is an open cone $\Gamma_1$ containing $\xi_0$ such that 
$\xi v\leq -c_0|\xi|| v|$ for all $\xi\in \Gamma_1$ and some constant $c_0>0$.
For $\xi\in \Gamma_1$ and $t$ in some bounded neighborhood $W$ of the origin we have
\begin{align*}
\real Q(t,\xi,x+i\lambda v)&=\lambda (\xi v)-|\xi|\real p(t-x-i\lambda v)\\
&=\lambda (\xi v)-|\xi|\bigl(\real p(t-x) +O(\lambda^2)| v|^2\bigr)\\
&\leq \lambda(\xi v)-c|\xi|\bigl(| t-x|^{2k}+O(\lambda^2)| v|^2\bigr)\\
&\leq -c_0\lambda| v||\xi| +O\bigl(\lambda^2\bigr)|\xi|.
\end{align*}
Hence for $\lambda$ small enough
\begin{equation}\label{exponentialdecay1}
\real Q(t,\xi,x+i\lambda v)\leq -\frac{c_0}{2}\lambda | v||\xi|,
\quad  \xi\in \Gamma_1, x\in B_{2r}, t \in W.  
\end{equation}
 We conclude that there are constants $\gamma_1,C_1>0$ such that
\begin{equation*}
|I_1| \leq 
C_1e^{-\gamma_1|\xi|}, \quad \xi\in \Gamma_1,  t \in W.  
\end{equation*}
We recall that \Cref{def:Rregular}(0) implies that $\omega_{\seq M}(t)=O(t)$ as $t\rightarrow \infty$ 
(cf.\ e.g.\ \cite{Komatsu73}, \cite{BMM07}, or \cite{RainerSchindl12}). Hence
there are constants $\gamma_1,C_1>0$ such that, for all $\rh >0$,  
\begin{equation*}
|I_1| \leq C_1
e^{-\omega_{\seq M}(\ga_1 \rho|\xi|)}, \quad \xi\in \Gamma_1,  t \in W.
\end{equation*}

For $I_2$ we estimate
\begin{align*}
\real Q(t,\xi,x+i\sigma v)&\leq \sigma (\xi v)- c| t-x|^{2k}|\xi|
+ O\bigl(\lambda^2\bigr)|\xi|\\
&\leq -c| t -x|^{2k}|\xi|+ O\bigl(\lambda^2\bigr)|\xi|.
\end{align*} 
If $x\in\supp (\partial \psi /\partial\bar{z}_j)$ then $| x|\geq r$.
Therefore, for $| t|\leq r/2$ and $\lambda$ small enough,  there is a 
constant $\gamma_2>0$ such that
\begin{equation*}
\real Q(t,\xi,x+i\sigma v)\leq -\gamma_2|\xi|, \quad \xi\in\Gamma_1.
\end{equation*}
Hence, for all $\rh >0$,
\begin{equation*}
|I_2|
\leq C_2e^{-\gamma_2|\xi|}\leq C_2e^{-\omega_{\seq M}(\gamma_2\rho| \xi|)},
\quad  \xi\in\Gamma_1,  
| t|\leq r/2, 0<\tau<\lambda.
\end{equation*}

By 
\eqref{exponentialdecay1}, we have for a generic constant $C_3>0$ and all $k\in\N$
\begin{align*}
|I_3|
&\leq C_3\int_0^\infty\! e^{-c_0\sigma | v||\xi|}h_{\seq m}(\rho\sigma | v|)\,
d\sigma 
\leq C_3\int_0^\infty\!e^{-c_0\sigma | v||\xi|} \rho^k\sigma^k| v|^k m_k\, 
d\sigma\\
&\le C_3\rho^k c_0^{-k}|\xi|^{-k}k! m_k\
=C_3\bigl(c_0^{-1}\rho\bigr)^kM_k|\xi|^{-k}
\end{align*}
and thus
\begin{align*}
|I_3|
&\leq C_3h_{\seq M} \bigl(\rho c_0^{-1}|\xi|^{-1}\bigr)
\leq C_3 e^{-\omega_{\seq M}(c_0\rho^{-1} |\xi|)}.
\end{align*}
In the Roumieu case this holds for some $\seq M \in \fM$ and some $\rh>0$, in the Beurling case for all 
$\seq M \in \fM$ and all $\rh>0$. 
Since the appearing constants do not depend on $\ta$, we may conclude \eqref{M-FBIestimate} 
in view of \eqref{eq:FBI} and \eqref{FBI-neccesity1}.

Let us now prove the converse implication. Fix $(x_0 = 0,\xi_0)$ and
assume that \eqref{M-FBIestimate} holds 
either for some $\seq M \in \fM$ and some $\ga>0$ or for all $\seq M \in \fM$ and all $\ga>0$.
We will prove that $(0,\xi_0)\notin \WF_{[\fM]} v$ where $v=\psi u$. We invoke the inversion formula 
for the FBI transform \cite{Berhanu:2012aa}
\begin{equation*}
v =\lim_{\ep\rightarrow\infty}\int_{\R^n\times\R^n}\!\!e^{i\xi(x-t)}e^{-\ep|\xi|^2}
\mathfrak{F}v(t,\xi)|\xi|^{\tfrac{n}{2k}}\,dtd\xi.
\end{equation*}
Let $v^\ep(z)$ denote the above integral for $x$ replaced by $z \in \C^n$. 
Then $v^\ep(z)$ is an entire function which we split as $v^\ep(z) = v^\ep_1(z) + v^\ep_2(z) +v^\ep_3(z) +v^\ep_4(z)$, where
\begin{align*}
v_1^\ep(z)&= \text{ the integral over } \{ \xi \in \R^n , |t| \le a\},
\\
v_2^\ep(z)&= \text{ the integral over } \{ |\xi| \le B, a \le |t| \le A\}, 
\\
v_3^\ep(z)&= \text{ the integral over } \{ \xi \in \R^n , |t| \ge A\},
\\
v_4^\ep(z)&= \text{ the integral over } \{ |\xi| \ge B, a \le |t| \le A\}
\end{align*}
for certain constants $a$, $A$ and $B$ to be determined. 
Following \cite{MR2397326} or \cite{Berhanu:2012aa} we see that $v_2^\ep$, $v_3^\ep$, and $v_4^\ep$ converge to holomorphic functions in a neighborhood 
of the origin as $\ep\rightarrow 0$.

It remains to look at $v_1^\ep$. Suppose that $a$ is small enough such that
$B_a\subseteq U$.
Let $\CC_j$, $1\leq j\leq N$, be open acute cones such that
\begin{equation*}
\R^n=\bigcup_{j=1}^N\overline{\CC}_j
\end{equation*}
and the intersection $\overline{\CC}_j\cap\overline{\CC}_k$ has measure zero for $j\neq k$. 
We may assume that
$\xi_0\in\CC_1$, 
$\CC_1\subseteq\Gamma$, and $\xi_0\notin\overline{\CC}_j$ for $j> 1$.
In particular, by \eqref{M-FBIestimate} we have
\begin{equation}\label{M-FBIestimate3}
\bigl|\mathfrak{F}(v)(x,\xi)\bigr|\leq C e^{-\omega_{\seq M}(\gamma|\xi|)}
\qquad x\in B_a, \xi\in \CC_1.
\end{equation}
For $j=2,\dotsc, N$ we can choose open cones $\Gamma_j$ such that $\xi_0\Gamma_j<0$ and
\begin{equation}\label{dualconeFBI}
\langle y,\xi\rangle\geq c | y | |\xi|\quad \text{ for }  y \in\Gamma_j, \xi\in\CC_j, 
\end{equation}
for some constant $c>0$.
For $j\in\{2,\dots, N\}$ and $\ep >0$
we set
\begin{equation*}
f_j^\ep(x+iy)=\int_{\CC_j}\int_{B_a} e^{i\xi(x+iy -t)-\ep|\xi|^2}
\mathfrak{F}v(t,\xi)|\xi|^{\tfrac{n}{2k}}\,dtd\xi.
\end{equation*}
Each $f_j^\ep$ is entire and for $\ep \to 0$
the functions $f_j^\ep$ converge uniformly on compact subsets of the wedge $\R^n+i\Gamma_j$
to the holomorphic function
\begin{equation*}
f_j(x+iy)=\int_{\CC_j}\int_{B_a} e^{i\xi(x+iy -t)}
\mathfrak{F}v(t,\xi)|\xi|^{\tfrac{n}{2k}}\,dtd\xi
\end{equation*}
on $\R^n\times i\Gamma_j$ thanks to \eqref{dualconeFBI}.
Similarly we define
\begin{align*}
f_1^\ep(x)&=\int_{\CC_1}\int_{B_a} e^{i\xi(x -t)-\ep|\xi|^2}
\mathfrak{F}v(t,\xi)|\xi|^{\tfrac{n}{2k}}\,dtd\xi\\
\end{align*}
and
\begin{align*}
f_1(x)&=\int_{\CC_1}\int_{B_a} e^{i\xi(x -t)}\mathfrak{F}v(t,\xi)|\xi|^{\tfrac{n}{2k}}\,dtd\xi.
\end{align*}
The functions $f_1^\ep$, $\ep>0$, extend to entire functions, 
whereas $f_1$ is smooth, by \eqref{M-FBIestimate3}, since $e^{-\omega_{\seq M}}$ is rapidly decreasing.
This decrease also shows that $f_1^\ep$ converges uniformly to $f_1$ in a neighborhood of $0$, since
\begin{align*}
\bigl| f_1(x)-f_1^\ep(x)\bigr|
&\leq C\int_{\CC_1}\!|\xi|^{\tfrac{n}{2k}}e^{-\omega_{\seq M}(\gamma|\xi|)}
\bigl| 1-e^{-\ep|\xi|^2}\bigr|\,d\xi \to 0
\end{align*}
by the monotone convergence theorem. Moreover,
\begin{align*}
\bigl| D^\alpha f_1(x)\bigr| 
&\leq\int_{\CC_1} \int_{B_a} |\xi|^{\tfrac{n}{2k}}
\bigl|\xi^{\alpha}\mathfrak{F}v(t,\xi)\bigr|\,dtd\xi\\
&\leq C\int_{\CC_1}|\xi|^{\tfrac{n}{2k}+
|\alpha|} e^{-\omega_{\seq M}(\gamma |\xi|)}\,d\xi 
= C\int_{\CC_1} |\xi|^{\tfrac{n}{2k}+
|\alpha|} h_{\seq M}(\tfrac1{\gamma |\xi|})\,d\xi 
\\
&\leq C\gamma^{-2n+|\alpha|}M_{2n+|\alpha|}
\int_{\CC_1}|\xi|^{\tfrac{n}{2k}-2n}\,d\xi
\leq C'\gamma^{|\alpha|}M^\prime_{|\alpha|},
\end{align*}
for a suitable $\seq M' \in \fM$. Here we use the [semiregularity] of $\fM$. 
Thus $f_1 \in \cE^{[\fM]}$.

So we have shown that on an open neighborhood $V$ of the origin and some open cones $\Gamma_j$, 
$j=2,\dotsc, N$ that satisfy $\xi_0\Gamma_j<0$ we can write
\begin{equation*}
v\vert_V= v_0+\sum_{j=2}^N b_{\Gamma_j} f_j
\end{equation*}
with $v_0\in\cE^{[\seq M]}(V)$ and $f_j$ holomorphic on $V +i\Gamma_j$ for $j=2,\dotsc,N$.
This completes the proof, by \Cref{SemiBVChar}.
\end{proof}

\section{Elliptic regularity} \label{sec:elliptic}

The smooth elliptic regularity theorem, cf.\ \cite[Theorem 8.3.1]{Hoermander83I}, states that 
a linear differential operator $P$ with smooth coefficients satisfies
\begin{equation*}
\WF u\subseteq\WF Pu\cup \Char P, \quad   u \in \cD'. 
\end{equation*}
In particular, if $P$ is elliptic then it is microhypoelliptic, i.e., $\WF Pu=\WF u$. 
Analogous results hold in the analytic category (see \cite{MR0650826}).
Recall that $$\Char P=\{(x,\xi) \in\COT : P_m(x,\xi)=0\}$$ 
is the characteristic set of $P=\sum_{|\alpha|\leq m} a_{\alpha}(x) D^\alpha$
with principal symbol $P_m(x,\xi)=\sum_{|\alpha|=m} a_{\alpha}(x)\xi^\alpha$.

In the ultradifferentiable case an elliptic regularity theorem was proven in \cite{H_rmander_1971} for Roumieu classes given by weight sequences 
and  operators with real analytic coefficients.
In \cite{Albanese:2010vj} an elliptic regularity theorem was obtained for operators with ultradifferentiable coefficients of type $\cE^{[\om]}$. 

In this section we prove an elliptic regularity theorem in the general setting of ultradifferentiable classes defined by weight matrices.
As \cite{Albanese:2010vj} 
we follow the pattern of proof of \cite{H_rmander_1971} and we try to find the weakest possible conditions on the weights.
The results of \cite{H_rmander_1971} and \cite{Albanese:2010vj} follow as special cases of our theorem.

\subsection{The ultradifferentiable elliptic regularity theorem}

We will need a condition with generalizes \emph{moderate growth} of a sequence:
\begin{equation}
  \A \seq M\in\fM \E \seq N \in\fM \E C>0 \A j,k \in\N :   M_{j+k} \leq C^{j+k} N_j N_k.\label{R-mg}
\end{equation}
Note that this is the ``Roumieu variant'' which will be sufficient for our purpose. 

Recall that for an R-semiregular weight matrix condition \Cref{rem:regular}\eqref{aiR} 
is equivalent to
\begin{equation}
\A \seq M \in \fM \E \seq N \in \fM \E C>0 \A k \in \N :   \max_{\substack{\alpha_1+\cdots +\alpha_j=k\\ 
\alpha_\ell >0}} m_jm_{\alpha_1}\dots m_{\alpha_j} \le C^{k+1} n_k. \label{R-FdB}
\end{equation}
Let us point out that the weight matrix $\fW$ associated with a weight function $\om$ always satisfies 
\eqref{R-mg} (see \Cref{lemma4}), and $\fW$ fulfills \Cref{rem:regular}\eqref{aiR}
if and only if $\om$ is equivalent to a concave weight function (see \Cref{thm:omegachar}).

\begin{theorem}\label{elliptic-regThm}
Let $\fM$ be an R-semiregular weight matrix that satisfies \eqref{R-mg} and \eqref{R-FdB} 
and $P(x,D) = \sum_{|\al| \le m} a_\al(x) D^\al$ a linear partial differential operator with $\cE^{\{\fM\}}(\Omega)$-coefficients.
Then we have the following statements.
\begin{enumerate}
\item If $\fL$ is a R-semiregular weight matrix such that $\fM \{\preceq\}\fL$ then
\begin{equation}\label{EllipticReg}
\WF_{\{\fL\}}u \subseteq \WF_{\{\fL\}} Pu\cup\Char P
\end{equation}
for all $u\in\D^\prime(\Omega)$. If $P$ is elliptic, then $\WF_{\{\fL\}}u = \WF_{\{\fL\}} Pu$.
\item If $\fL$ is B-semiregular and $\fM \{\lhd)\fL$ then
\begin{equation}
\WF_{(\fL)} u \subseteq \WF_{(\fL)} Pu\cup \Char P
\end{equation}
for all $u\in\D^\prime(\Omega)$. If $P$ is elliptic, then $\WF_{(\fL)}u = \WF_{(\fL)} Pu$.
\end{enumerate}
\end{theorem}

\begin{proof}
It suffices to show that $(x_0,\xi_0)\notin\WF_{[\fL]} Pu\cup\Char P$ for $\xi_0 \ne 0$ implies $(x_0,\xi_0)\notin\WF_{[\fL]}u$.
Therefore we can assume that there are a compact neighborhood $K$ of $x_0$ and a closed conic neighborhood $V$ of $\xi_0$
such that the principal symbol $P_m(x,\xi) = \sum_{|\al| = m} a_\al(x) \xi^\al$ is non-zero in $K \times V$  
and 
\begin{gather*}
(K\times V)\cap\WF_{[\fL]} Pu=\emptyset.
\end{gather*}

By \cite[Theorem 1.4.2]{Hoermander83II} there is a sequence $(\lambda_N)\subseteq\D(K)$ with $\lambda_N\vert_U\equiv 1$
on some fixed neighborhood $U$ of $x_0$ such that for all $\alpha\in\N^n$ there are constants
$C_\alpha,\,h_\alpha>0$ such that
\begin{equation}\label{testfunctestimate}
\bigl| D^{\alpha+\beta}\lambda_N\bigr|\leq C_\alpha \bigl(h_\alpha N\bigr)^{|\beta|} \quad \text{ for } |\be|\le N = 1,2, \ldots 
\end{equation}
Now the sequence $u_N=\lambda_{2N} u$ is bounded in $\cE^\prime (K)$ and each of its elements
is equal to $u$ on $U$. Hence it suffices to show that the sequence $(u_N)_N$ satisfies \eqref{WF-defining} 
\begin{itemize}
  \item for some $Q>0$ and some $\seq L\in\mathfrak{L}$ in the Roumieu case,
  \item for all $Q>0$ and all $\seq L \in \fL$ in the Beurling case.  
\end{itemize}
The first part of the proof is valid in both cases.

Following the approach of H\"ormander \cite[Theorem 8.6.1]{Hoermander83II} we first want to solve the equation $Qg=e^{-ix\xi}\lambda_{2N}$, 
where $Q g = \sum (-D)^\al( a_\al g)$ is the formal adjoint of $P$. 
The ansatz $g=e^{-ix\xi}P_m(x,\xi)^{-1}w$ leads to the equation
\begin{equation}\label{w-eq}
w-Rw=\lambda_{2N}
\end{equation}
where $R=R_1+\dots +R_m$ and $R_j|\xi|^j$ is a differential operator of order $\leq j$ with $\cE^{\{\fM\}}$-coefficients  
which are homogeneous of degree $0$ in $\xi$ if $x\in K$ and $\xi\in V$.
A formal solution of \eqref{w-eq} would be
$w=\sum_{k=0}^\infty R^k\lambda_{2N}$, but this series may diverge in general and we cannot consider
derivatives of $\lambda_{2N}$ of arbitrary high order. Hence we set
\begin{equation*}
w_N=\sum_{j_1+\dots +j_k\leq N-m} R_{j_1}\dots R_{j_k}\lambda_{2N}
\end{equation*}
and calculate
\begin{equation*}
w_N-Rw_N=\lambda_{2N}-\sum_{\sum_{s=1}^kj_s>N-m\geq\sum_{s=2}^kj_s}R_{j_1}\dots R_{j_k}\lambda_{2N}
=:\lambda_{2N}-\rho_N.
\end{equation*}
Therefore 
\begin{equation*}
Q\left(e^{-ix\xi} P_m(x,\xi)^{-1}w_N(x,\xi)\right)=e^{-ix\xi}\left(\lambda_{2N}(x)-\rho_N(x,\xi)\right).
\end{equation*}
We obtain 
\begin{align}
\widehat{u}_N(\xi) &=
\bigl\langle u , e^{-i\langle\cdot,\xi\rangle}\lambda_{2N}\bigr\rangle  
\notag\\
&=\bigl\langle Pu, e^{-i\langle \cdot,\xi\rangle} P_m^{-1}(\cdot , \xi ) w_N(\cdot , \xi)\bigr\rangle
+\bigl\langle u,e^{-i\langle \cdot ,\xi\rangle} \rho_N (\cdot ,\xi)\bigr\rangle, \quad \xi\in V. 
\label{WF-M-EllipticRegEq1}
\end{align}

In order to proceed we make the following claim which will be proved in \Cref{Lemma1} below:
\emph{There exist $\seq M\in\fM$, $h>0$, and constant $C>0$ (only depending on $R$, $\fM$, $h$ and the sequence 
$(\lambda_N)_N$) such that, if $j=j_1+\dots +j_k$ and $j+|\beta|\leq 2N$, then
\begin{equation}\label{WF-M-RegEst1}
\left| D^{\beta}\bigl(R_{j_1}\dots R_{j_k}\lambda_{2N}\bigr)\right| 
\leq Ch^{j+|\beta|} M^{\tfrac{j+|\beta|}{N}}_{N}|\xi|^{-j},
\qquad \xi\in V.
\end{equation}}

We use this to estimate the terms on the right-hand side of \eqref{WF-M-EllipticRegEq1} for $\xi\in V$,
where $|\xi|$ is large. We begin with the second term $II:=\bigl\langle u,e^{-i\langle \cdot,\xi\rangle} \rho_N (\cdot ,\xi)\bigr\rangle$.

Since $u$ is of finite order, say $\mu$, near $K$, there is a constant $C_u$ that only depends on $K$ and $u$ such that for all
$\psi\in\D(\Omega)$ with $\supp \psi\subseteq K$ we have
\begin{equation*}
\bigl|\langle u,\psi\rangle\bigr|
\leq C_u\sum_{|\alpha|\leq \mu}\sup_K\bigl| D^{\alpha}\psi\bigr|.
\end{equation*}
Note that $\supp_x\rho_N(\cdot,\xi)\subseteq K$ for all $\xi\in V$ and $N\in\N$. Thence
\begin{equation*}
II 
\leq C \sum_{|\alpha| \leq \mu} \sum_{\beta\leq\alpha}
|\xi|^{|\alpha|-|\beta|} 
\sup_{x\in K}\bigl| D^{\beta}_x\rho_N(x,\xi)\bigr|\\
\leq C\sum_{|\alpha|\leq \mu}
|\xi|
^{\mu-|\alpha|} \sup_{x\in K}\bigl| D_x^{\alpha}\rho_N(x,\xi)\bigr|,
\end{equation*}
for $\xi\in V$ with $|\xi|\geq 1$ and $N\in \N$.
There are at most $2^N$ terms 
in  $\rho_N$ and each term can be estimated by
\eqref{WF-M-RegEst1} (since $N\geq j>N-m$), whence
\begin{equation*}
\bigl| D_x^{\alpha}\rho_N(x,\xi)\bigr|\leq Ch^{N}2^N 
| \xi|^{m-N}M_{N}^{\tfrac{N+\mu}{N}}
\end{equation*}
for $x\in K$ and $\xi\in V$ with $|\xi|>1$.
Thus, by \Cref{def:Rregular}\eqref{R-Derivclosed1}, there exists $\seq M \in \fM$ and $h_1 >0$ such that
\begin{equation}\label{EllipticREGULAR1}
II
\leq C h_1^N |\xi|^{\mu+m-N} M_N.
\end{equation}

(1) Let us consider the Roumieu case and assume that $\fM\{\preceq\}\fL$. Then, by \eqref{EllipticREGULAR1}, 
there exists $\seq L \in \fL$ and $h >0$ such that
\begin{equation}\label{EllipticREGULAR2}
II
\leq C h^N |\xi|^{\mu+m-N} L_N.
\end{equation}
The first term $I := \bigl\langle Pu, e^{-i\langle \cdot,\xi\rangle} P_m^{-1}(\cdot , \xi ) w_N(\cdot , \xi)\bigr\rangle$ 
in \eqref{WF-M-EllipticRegEq1} is more difficult to estimate.
For $N>m$, $|\beta|\leq N$ and $\xi\in V$ with $|\xi| > M_{N}^{1/N}$,
\eqref{WF-M-RegEst1} gives
\begin{equation*}
\bigl| D^{\beta}w_N(x,\xi)\bigr|
\leq C\sum_{j=0}^{N-m}h^{j+|\beta|}M_{N}^{\tfrac{j+|\beta|}{N}}|\xi|^{-j}
\leq Ch^{|\beta|}M_{N}^{\tfrac{|\beta|}{N}}\sum_{j=0}^{N-m}h^j
\le C_1 h_1^N M_{N}^{\tfrac{|\beta|}{N}}
\end{equation*}
for suitable $C_1$ and $h_1$.
Analogously, one obtains a similar bound for 
$\widetilde{w}_N(x,\xi)=w_N(x,\xi)|\xi|^{m} P_m^{-1}(x,\xi)$. 
Let 
\begin{equation*}
\widehat{\widetilde{w}}_N(\eta,\xi)=\int_\Omega \! e^{-ix\eta}w_N(x,\xi)\,dx
\end{equation*}
be the partial Fourier transform of $\tilde{w}_N(\cdot,\xi)$. 
Then, by the above, there exist $\seq M\in\fM$ and $h>0$ such that
\begin{equation} \label{ellipticHelp1}
\bigl|\eta^\beta\widehat{\tilde{w}}_N(\eta,\xi)\bigr|\leq C h^N M_{N}^{\tfrac{|\beta|}{N}}
\end{equation}
for all $N>m$, $|\beta|\leq N$, $\xi\in V$ with
$|\xi|> M_{N}^{1/N}$ and $\eta\in\R^n$.
So, for some $q>0$, 
\begin{equation}\label{ellipticHelp2}
 \Big(|\eta|+M_{N}^{\tfrac{1}{N}}\Big)^N\left|\widehat{\tilde{w}}_N(\eta,\xi)\right|
 \leq C(\sqrt{n}h)^N\sum_{k=0}^N\binom{N}{k}M_{N}^{\tfrac{k}{N}}M_{N}^{\tfrac{N-k}{N}} 
 \leq C q^N M_{N}.
 \end{equation}

Now set $f=Pu$ and recall that by assumption $\WF_{\{\fL\}}f\cap (K\times V)=\emptyset$.
By \Cref{WF-M Charakterisierung}, we find a sequence $(f_N)_N$ which is bounded in 
 $\cE^{\prime ,\mu}$, equals $f$ in some neighborhood of $K$, and there exist $\seq L \in \fL$ and $Q>0$ such that  
 \begin{equation}\label{ellipticHelp0}
 \bigl|\widehat{f}_N(\eta)\bigr|\leq C \frac{Q^NL_N}{| \eta|^N}, \quad \text{ for } N \in \N, \eta\in W,
 \end{equation}
 where $W$ is a conic neighborhood of $V$. 
Then $\tilde{w}_Nf=\tilde{w}_Nf_{N^\prime}$ for
 $N^{\prime}=N-\mu -n$. 
 In analogy with \eqref{eq:convolution} we find, for $0 < c< 1$,
 \begin{align*}
(2\pi)^n\Bigl| \widehat{\tilde{w}_Nf}(\xi)\Bigr| 
&\leq (1-c)^{-N'} \left\lVert \widehat{\tilde{w}}_N(\cdot,\xi)\right\rVert_{L^1}
\sup_{\et \in W}\bigl| \widehat{f}_{N^\prime}(\eta)\bigr| |\xi|^{-N'} |\et|^{N'} 
\\
&\quad
+
C\int\limits_{|\eta|>c|\xi|}\!
\left| \widehat{\tilde{w}}_{N}(\eta,\xi)\right|\bigl(1+c^{-1}\bigr)^\mu \bigl(1+|\eta|)^\mu\,d\eta.
 \end{align*}
By \eqref{ellipticHelp2}, if $N>n+\mu+m$, then
\begin{align*}
\left\lVert \widehat{\tilde{w}}_N(\cdot ,\xi)\right\rVert_{L_1}
&\leq Cq^N M_{N}\int_{\R^n}\Bigl(|\eta|+\sqrt[N]{M_{N}}\Bigr)^{-N}\,d\eta\\
&\le  C_1 q^N M_{N}\int_0^\infty\! \Bigl(r+\sqrt[N]{M_{N}}\Bigr)^{-N}r^{n-1}\,dr\\
&\leq C_1 q^N M_{N}\int_0^\infty\! \Bigl(r+\sqrt[N]{M_{N}}\Bigr)^{-N^\prime -1}\,dr\\
&= C_1q^N M_{N}\int_{\sqrt[N]{M_{N}}}^\infty s^{-N^\prime-1}\,ds\\
&= C_1q^N M_{N}^{1-N'/N}/N'\\
&\leq C_1q^N M_{N}^{\tfrac{\mu +n}{N}}.
\end{align*} 
Together with \eqref{ellipticHelp0} and \eqref{ellipticHelp1}, and
since $\sqrt[N]{M_N}$ is increasing,  
we conclude that for $\xi \in V$ with $|\xi| > \sqrt[N]{M_N}$,
\begin{align}\label{EllipticOuterEst1}
\Bigl|\widehat{\tilde{w}_Nf}(\xi)\Bigr| 
&\leq C_1 \left(\frac{q}{1-c}\right)^N M_{N}^{\tfrac{n+\mu}{N}} Q^{N^\prime}L_{N^\prime}|\xi|^{-N^{\prime}}
\notag \\
&\quad
+C_2(\sqrt{n}h)^N M_{N}\int_{|\eta|>c|\xi|}\negthickspace
|\eta|^{-N^\prime-n}\,d\eta \notag \\
&\leq C_1 q_0\left(\frac{qQ}{1-c}\right)^NL_{N}^{\tfrac{n+\mu}{N}}L_{N}^{\tfrac{N^\prime}{N}}|\xi|^{-N^\prime}
+C_2(\sqrt{n}q_0h)^N
 L_{N}c^{-N^\prime}|\xi|^{-N^\prime}
\notag \\
&\leq C_3 q_3^N L_{N}
|\xi|^{-N^\prime},
\end{align}
where we used the fact that there is a constant
$q_0$ such that $M_N^{1/N}\leq q_0L_N^{1/N}$. 

Now setting $N^\ast=N+n+\mu+m$ and $v_N=u_{N^\ast}$ 
we may conclude from \eqref{EllipticREGULAR2} and \eqref{EllipticOuterEst1} that there exist $\seq L \in \fL$ and $h>0$ such that 
\begin{equation*}
| \xi|^N\bigl|\widehat{v}_N(\xi)\bigr|\leq C h^N L_N, \quad \text{ for } \xi \in V \text{ with } |\xi| > M_{N^\ast}^{1/N^\ast}. 
\end{equation*}
The boundedness of the sequence $(v_N)_N$ in 
$\cE^{\prime,\mu}$ implies an estimate analogous to \eqref{BanachSteinhaus} and hence   
we have
\begin{equation}\label{EllipticRegular}
|\xi|^N\bigl| \widehat{v}_N(\xi)\bigr|\leq CM_{N^\ast}^{\tfrac{N+\mu}{N^\ast}},
\quad \text{ for } |\xi|\leq M_{N^\ast}^{1/N^\ast}. 
\end{equation}
This completes the proof of (1).

(2) Let us treat the Beurling case. The assumption $\fM \{\lhd) \fL$ and \eqref{EllipticREGULAR1} yield that 
\eqref{EllipticREGULAR2} holds for all $\seq L \in \fL$ and all $h>0$. 
Moreover, $f = Pu$ now 
satisfies $\WF_{(\fL)}f\cap (K\times V)=\emptyset$, by assumption, and hence \eqref{ellipticHelp0} holds 
for all $\seq L \in \fL$ and all $Q>0$. 
Together with $\fM \{\lhd) \fL$ this allows us to finish the proof in analogy to the Roumieu case in (1).  
\end{proof}

It remains to establish the claim \eqref{WF-M-RegEst1}:

\begin{lemma}\label{Lemma1}
There exist $\seq M\in\fM$, $h>0$, and constant $C>0$ (only depending on $R$, $\fM$, $h$ and the sequence 
$(\lambda_N)_N$) such that, if $j=j_1+\dots +j_k$ and $j+|\beta|\leq 2N$, then
\begin{equation}\label{WF-M-RegEst1again}
\left| D^{\beta}\bigl(R_{j_1}\dots R_{j_k}\lambda_{2N}\bigr)\right| 
\leq Ch^{j+|\beta|} M^{\tfrac{j+|\beta|}{N}}_{N}|\xi|^{-j},
\qquad \xi\in V.
\end{equation}
\end{lemma}

\begin{proof}
  Since both sides of \eqref{WF-M-RegEst1again} are homogeneous of degree $-j$ in $\xi\in V$ 
  it suffices to prove the lemma for $|\xi|=1$.
  The set $\mathcal{R}\subseteq \cE^{\{\fM\}}(K)$ of all coefficients of the operators $R_1,\dotsc,R_m$ is finite.
Hence there are constants $h$ and $C$ and a weight sequence
$\seq M\in\fM$ such that
\begin{equation}\label{Rcoeff}
| D^\alpha a(x)|\leq Ch^{|\alpha|}M_{|\alpha|}, 
\quad \text{ for } a \in \cR, x \in K, \al \in \N_0^n.  
\end{equation}
Thus the assertion is a consequence of the next lemma.
\end{proof}

\begin{lemma}\label{Lemma2}
  Let $K\subseteq \Omega$ be compact, $(\lambda_N)_N\subseteq\D (K)$ a sequence satisfying
  \eqref{testfunctestimate} and $a_1,\dotsc,a_{j-1}\in\mathcal{R}$. 
  Then there exist $\seq M\in\fM$ and $C,h>0$ (independent of $N$)
  such that
  \begin{equation}\label{Est26}
  \bigl| D_{i_1}(a_1D_{i_2}(a_2 \cdots D_{i_{j-1}}(a_{j-1}D_{i_j}\lambda_{2N})\cdots))\bigr|\leq C h^j M_N^{\tfrac{j}{N}},
  \quad \text{ for } j \le 2N.
  \end{equation} 
\end{lemma}

\begin{proof}
By \eqref{testfunctestimate} and \eqref{strictInclusion2},
for each $q>0$ and each 
$\seq M \in \fM$ there exists $C' \ge 1$ such that 
\begin{align}
  | D^{\beta}\lambda_{2N}| 
  \le 
  C' q^{|\beta|}M_N^{\tfrac{|\beta|}{N}}  
  \quad \text{ for }  |\be| \le 2 N.
\end{align}

  The left-hand side of \eqref{Est26}  
  is a sum of terms of the form
  $(D^{\alpha_1}a_1)\dotsb (D^{\alpha_{j-1}}a_{j-1})D^{\alpha_j}\lambda_{2N}$ for
  $|\alpha_1|+\dots +|\alpha_j|=j$.
  If $C_{k_1,\dots,k_j}$ is the number of terms with $|\alpha_1| =k_1,\dots,|\alpha_j|=k_j$, then, 
  thanks to \eqref{testfunctestimate}, \eqref{Rcoeff}, and \eqref{R-FdB}, there exists $\seq M' \in \fM$ such that 
  the left-hand side of \eqref{Est26} is bounded by
  \begin{align}\label{Est28}
  C_0&\sum C^{j-1} h^{j-k_j}C_{k_1,\dots,k_j}
  m_{k_1}\cdots m_{k_{j-1}}k_1!\cdots k_{j-1}!h_0^{k_j}N^{k_j} \notag \\
  &\leq C_0 C^{j-1} \sum h^{j-k_j}m^\prime_{j-k_j}C_{k_1,\dotsc,k_j}k_1!\cdots k_{j-1}!
   h_0^{k_j}N^{k_j} \notag\\
  &\leq C_0 C^j\sum h^{j-k_j}h_0^{k_j}C_{k_1,\dotsc,k_j}\frac{k_1!\cdots k_{j-1}!}{(j-k_j)!}M^\prime_{j-k_j}N^{k_j}.
  \end{align}
    By \eqref{R-mg}, there exist $\seq M'' \in \fM$ and a constant $q_2>0$ such that
    $M'_{j-k_j}\leq q_2^{j-k_j}M''_{\sigma_1}M''_{\sigma_2}$
    if $\sigma_1+\sigma_2=j-k_j$.
  By \eqref{strictInclusion2}, 
  there exists $C_2 >0$ such that
  \begin{equation*}
h^{j-k_j}h_0^{k_j}  M^\prime_{j-k_j}N^{k_j}
\leq C_2^{\tfrac{k_j}{N}} (h q_2)^j
 M''_{\sigma_1}M''_{\sigma_2}(M''_N)^{\tfrac{k_j}{N}}\\
\leq C_2^{\tfrac{k_j}{N}} (hq_2)^j
(M''_N)^{\tfrac{j}{N}}
  \end{equation*}
  since $\sqrt[N]{M_N}$ is increasing. 
  As noted in \cite{Albanese:2010vj} and \cite[p. 308]{Hoermander83I} one has
  \begin{equation*}
  \sum C_{k_1,\dotsc,k_j} \frac{k_1!\cdots k_{j-1}!}{(j-k_j)!} 
  \le \frac{2^j}{j!} \sum C_{k_1,\dotsc,k_j}k_1!\cdots k_j!=\frac{2^j(2j-1)!!}{j!} \le 4^j.
  \end{equation*}
The lemma follows.
\end{proof}

\subsection{Stronger versions in special cases} \label{sec:openproblem}

As a special case of \eqref{EllipticReg} we obtain 
\[
	\WF_{\{\fM\}} u \subseteq \WF_{\{\fM\}} Pu \cup \Char P, \quad u \in \cD', 
\]
for any $P$ with $\cE^{\{\fM\}}$-coefficients, where $\fM$ satisfies the assumptions of \Cref{elliptic-regThm}.
We do not know if an analogous statement holds in this generality in the Beurling case, 
but we have two important partial results \Cref{cor:elliptic1} and \Cref{cor:strong}.

\begin{theorem} \label{cor:elliptic1}
	Let $\seq M$ be a strongly log-convex weight sequence of moderate growth with $m_k^{1/k} \to \infty$ and 
	$P(x,D) = \sum_{|\al| \le m} a_\al(x)  D^\al$ a linear partial differential operator with 
   $\cE^{(\seq M)}$-coefficients. 
	Then
	\[
		\WF_{(\seq M)} u \subseteq \WF_{(\seq M)} Pu \cup \Char P, \quad u \in D'.   
	\]
	If $P$ is elliptic, then $\WF_{(\seq M)} u = \WF_{(\seq M)} Pu$.
\end{theorem}

\begin{proof}
	As in the proof of \Cref{elliptic-regThm} we fix a compact $K \subseteq \Om$.
	Let 
	\[
	L_k := \max \Big\{ \max_{|\be| = k} \max_{|\al| \le m} \sup_{x \in K} |\p^\be a_\al(x)|, k!\Big \}. 
	\]
	Then $\seq L \lhd \seq M$. By \Cref{Komatsu} below, there exists a strongly log-convex weight sequence of moderate growth 
   $\seq M'$ 
	such that $\seq L \le \seq M' \lhd \seq M$. Thus we may apply (the proof of) \Cref{elliptic-regThm}(2) and the statement follows. 
\end{proof}

\begin{lemma} \label{Komatsu}
   Let $\seq L,\seq M$ be positive sequences satisfying $\seq L \lhd \seq M$ and $L_0 = M_0 =1$. 
   Suppose that $\seq M$ is strongly log-convex and satisfies $m_k^{1/k} \to \infty$.
   Then there exists a strongly log-convex sequence $\dot {\seq M}$ with $\dot m_k^{1/k} \to \infty$ such that 
   $\seq L\leq \dot {\seq M}\lhd  \seq M$. 
   If $\seq M$ has moderate growth, then so does $\dot {\seq M}$.  
\end{lemma}

\begin{proof}
The first assertion simply follows from \cite[Lemma 6]{Komatsu79b}. 
Since $\seq L \lhd \seq M$, for all $\rh >0$ there is $C >0$ such that
\begin{equation} \label{eq:Komatsu1}
   L_k \le  C \rh^k M_k \quad \text{ for all } k.
\end{equation}
Let $C_\rh$ be the infimum of all $C>0$ such that \eqref{eq:Komatsu1} holds.
Consider the sequence $\bar {\seq L} =(\bar L_k)$ defined by 
   \begin{equation*} 
      \bar L_k := \inf_{\rh>0} C_\rh \rh^k M_k.  
   \end{equation*}
Notice that $c_k := \mu_k/\bar \la_k$, where  $\mu_k := M_k/M_{k-1}$, $\bar \la_k := \bar L_k/\bar L_{k-1}$, 
is increasing.  
Then $\dot M_k = \dot \mu_1 \dot \mu_2 \cdots \dot \mu_k$, $\dot M_0 :=1$ with 
\[
   \frac{\dot \mu_k}k := \max\Big\{\sqrt{\frac{\mu_k}k}, \max_{1 \le j \le k} \frac{\bar \la_j}j \Big\}, 
\] 
satisfies the first part of the assertion; for details see \cite[Lemma 6]{Komatsu79b}.
Let us check that $\dot {\seq M}$ has moderate growth if that is true for $\seq M$.    
   By \cite[Lemma 2.2]{RainerSchindl16a}, $\seq M$ has moderate growth if and only if $\mu_{2k}\lesssim \mu_k$. In that case
   \begin{align*}
   \frac{\dot \mu_{2k}}{2k} = \max\Big\{\sqrt{\frac{\mu_{2k}}{2k}}, \max_{1 \le j \le 2k} \frac{\bar \la_j}j \Big\} \lesssim 
   \frac{\dot \mu_k}{k},
   \end{align*}
   because for $k < j \le 2k$ we have 
   \begin{align*}
   \frac{\bar \la_j}j = \frac{\mu_j}{j c_j} \le \frac{\mu_{2k}}{2k c_k} \lesssim \frac{\mu_k}{k c_k} =  \frac{\bar \la_k}k 
   \end{align*}
   since $c_k$ and $\mu_k/k$ are increasing.
   It follows that $\dot {\seq M}$ has moderate growth. 
\end{proof}

We get a similar result for concave weight functions
which is a strengthened version of 
\cite[Theorem 4.1]{Albanese:2010vj} with operator and 
   wave front set of the same Beurling class.
It depends crucially on the following lemma.

We recall that a weight function $\om$ is equivalent to a concave weight function 
if and only if
\begin{equation} \label{eq:concavecond}
      \E C\ge 1 \E t_1>0 \A t\ge t_1 \A \la\ge 1 : \frac{\om(\la t)}{\la t} \le C \frac{\om(t)}{t};
   \end{equation}
see \Cref{thm:omegachar}.    

\begin{lemma} \label{lem:concavedescend}
   Let $\om : [0,\infty) \to [0,\infty)$ be continuous, increasing, surjective and such that $\om(t) = o(t)$ as $t\to \infty$.
   Assume that $\om$ satisfies \eqref{eq:concavecond}.
   Let $h : [0,\infty) \to [0,\infty)$ be a function such that $\om(t) = o(h(t))$ as $t \to \infty$.
   Then there exists a continuous, increasing, surjective function $\si : [0,\infty) \to [0,\infty)$  
   such that $\om(t) = o(t)$ as $t\to \infty$ and 
   \begin{enumerate}
      \item $\om(t) = o(\si(t))$ as $t \to \infty$,
      \item $\si(t) = o(h(t))$ as $t \to \infty$,
      \item $\si(\la t) \le \la  \si(t)$ for all $\la \ge 1$ and $t\ge t_1$ (with the same $t_1$ as above).
   \end{enumerate}
\end{lemma}

\begin{proof}
   Note that \eqref{eq:concavecond} can be reformulated as follows
   \begin{equation} \label{eq:concavecond2}
      \E C\ge 1 \E t_1>0 \A s \ge t\ge t_1  : \frac{\om(s)}{s} \le C \frac{\om(t)}{t}.
   \end{equation}
   Let us define
   \[
      \om_1(t) := t \sup_{s \ge t} \frac{\om(s)}{s}, \quad t \ge t_1, 
   \]
   and extend $\om_1$ to $[0,t_1]$ in such a way that $\om_1 : [0,\infty) \to [0,\infty)$ is
   continuous, increasing, surjective and such that $\om(t) = o(t)$ as $t\to \infty$; that this 
   is possible follows from the fact that $\om(t) = O(\om_1(t))$ and $\om_1(t) = O(\om(t))$ as $t \to \infty$
   which is a consequence of \eqref{eq:concavecond2}. 
   By definition $\om_1(t)/t$ is decreasing for $t \ge t_1$. 
   Moreover, $\om_1(t) = o(h(t))$ as $t \to \infty$. 

   We define $\si : [0,\infty) \to [0,\infty)$ and $0=t_0 < t_1 <t_2 < \cdots \to \infty$ as follows:
   If $t_j$ with odd $j$ is already chosen, take $t_{j+1} > t_j$ to be the smallest solution of $t j \om_1(t_j) = (j+2) t_1  \om_1(t)$
   which exists since $\om_1(t)/t \to 0$ as $t \to \infty$.
   If $t_j$ with even $j$ is already chosen, choose $t_{j+1}>t_j$ such that 
   \begin{equation} \label{eq:concavecond3}
      \max \Big\{\frac{\om_1(t)}{t}, \frac{\om_1(t)}{h(t)} \Big\}  \le \frac{1}{(j+1)(j+3)} \quad \text{ for all } t\ge t_{j+1}.
   \end{equation} 
   This is possible since $\om(t) = o(t)$  and $\om_1(t) = o(h(t))$ as $t \to \infty$. Now set
   \begin{align*}
      \si(t) := \begin{cases}
         j \om_1(t) & \text{ if } t \in [t_{j-1},t_j) \text{ and $j\ge 1$ is odd,}  
         \\ 
         (j-1) t \om_1(t_{j-1})/t_{j-1} & \text{ if } t \in [t_{j-1},t_j) \text{ and $j\ge 1$ is even.} 
      \end{cases}
   \end{align*}
   Then $\si$ is continuous, increasing, and surjective. 

   That $\om_1(t) = o(\si(t))$ as $t \to \infty$ follows easily from the fact that $\om_1(t)/t$ is decreasing for $t \ge t_1$.

   Observe that for each odd $j$ we have $\si(t) \le (j+2) \om_1(t)$ for all $t \in [t_j, t_{j+2}]$, by the choice of $t_{j+1}$.
   Together with \eqref{eq:concavecond3} this implies
   $\si(t) = o(t)$ and $\si(t) = o(h(t))$ as $t \to \infty$.

   By construction $\si(t)/t$ is decreasing for $t \ge t_1$. This completes the proof. 
\end{proof}

\begin{theorem} \label{cor:strong}
  Let $\om$ be a concave weight function
  and let $P(x,D) = \sum_{|\al| \le m} a_\al(x)  D^\al$ be a linear partial differential operator with $\cE^{(\om)}$-coefficients.
  Then
  \[
    \WF_{(\om)} u \subseteq \WF_{(\om)} Pu \cup \Char P, \quad u \in \cD'.   
  \]
  If $P$ is elliptic, then $\WF_{(\om)} u = \WF_{(\om)} Pu$.
\end{theorem}

\begin{proof}
  Let $\seq L$ be the sequence defined in the proof of \Cref{cor:elliptic1}. 
  We may proceed as in the proof of \Cref{thm:strongerconcave} which is based on \cite[Theorem 4.5]{BBMT91}
  and obtain a function $h : [0,\infty) \to [0,\infty)$ such that 
  $\om(t) = o(h(t))$ as $t \to \infty$. 
  Then \Cref{lem:concavedescend} provides a 
  `weight' function $\si$ such that $\om(t) = o(\si(t))$ and 
  $\si(t) = o(h(t))$ as $t \to \infty$. 
  As in the proof of \Cref{thm:strongerconcave} we conclude that $a_\al|_K \in \cB^{\{\si\}}(K)$. 
  Since $\si$ is equivalent to a concave `weight' function, 
  we may apply (the proof of) \Cref{elliptic-regThm}(2) and \Cref{thm:omegachar}.
\end{proof}

\begin{remark} \label{rem:notconvex}
   We remark that formally $\si$ is not a weight function, since it is not clear that $t \mapsto \si(e^t)$ is convex 
  (see \eqref{om4}). But this is not needed in this context, 
  since the properties of $\si$ suffice to guarantee that the associated weight matrix satisfies  
  \eqref{R-mg} and \eqref{R-FdB}; cf.\ \cite[Section 3.1]{Schindl14}.

  In contrast, the proof of \Cref{prop:strongmatrix} depends crucially on \eqref{om4}; see \cite[Proposition 3]{Rainer:aa}
  and \cite[Lemmas 2.5 \& 3.6]{JSS19}. 
  Therefore, we cannot use \Cref{lem:concavedescend} in the proof of \Cref{thm:strongerconcave}.
\end{remark}

Let $\seq G^s$, $s>1$, be the Gevrey sequence defined by $G^s_k := k!^s$.
It is immediate from \Cref{elliptic-regThm} that  
\[
  \WF_{\{\seq G^s\}} u \subseteq \WF_{\{\seq G^s\}} Pu \cup \Char P, \quad u \in \cD',
\]
if $P$ has $\cE^{\{\seq G^s\}}$-coefficients,
and from \Cref{cor:elliptic1} that 
\[
  \WF_{(\seq G^s)} u \subseteq \WF_{(\seq G^s)} Pu \cup \Char P, \quad u \in \cD',
\]
if $P$ has $\cE^{(\seq G^s)}$-coefficients. 
Now \Cref{Gevrey} follows easily in view of \Cref{WF-Intersection}(3).

\begin{remark}
If we modify the proof of \Cref{elliptic-regThm} following the lines of \cite{Fuerdoes1}, 
then we obtain \eqref{EllipticReg} for 
distributions $u\in\D^\prime(\Omega,\C^\nu)$, where $P$ is a square matrix of 
partial differential operators with ultradifferentiable coefficients.
\end{remark}

\subsection{Holmgren's uniqueness theorem}

Kawai \cite{MR0420735} and H\"ormander \cite{H_rmander_1971} separately showed 
that the elliptic regularity theorem can be used to prove Holmgren's uniqueness 
theorem \cite{zbMATH02662678}.
This scheme of proof was applied by the first author \cite{Fuerdoes1}
to extend Holmgren's uniqueness theorem to operators with coefficients in quasianalytic 
Roumieu classes defined by regular weight sequences of moderate growth.
The only other ingredient necessary for the proof was an appropriate version
of \Cref{ndimUniq}.

The same proof gives the following.

\begin{theorem} \label{thm:Holmgren}
Let $\fM$ be a quasianalytic R-semiregular weight matrix 
that satisfies \eqref{R-mg} and \eqref{R-FdB}.
Let $P$ be a linear partial 
differential operator with coefficients in $\cE^{\{\fM\}}(\Omega)$.
If $X$ is a $C^1$-hypersurface in $\Omega$ that is non-characteristic at $x_0$ and 
$u\in\D^\prime(\Omega)$ a solution of $Pu=0$ that vanishes on one side of $X$ near $x_0$, then
$u\equiv 0$ in a full neighborhood of $x_0$.
\end{theorem}

In particular, this theorem applies to operators with $\cE^{\{\om\}}$-coefficients
for concave quasianalytic weight functions $\om$.  
(Note that the Beurling version of the theorem follows trivially but is of no interest, since 
we always have $\cE^{(\fM)} \subseteq \cE^{\{\fM\}}$).
In \Cref{sec:QtransverseG} we give an example of a concave weight function
	$\om_0$ such that
	$\cE^{\{\om_0\}}$ is not included in $\cE^{\{\fG\}}$. 
	Hence \Cref{thm:Holmgren} applies to a wider class of operators than the
	quasianalytic Holmgren theorem given in \cite{Fuerdoes1}
   (in fact a class $\cE^{\{\seq M\}}$ with regular $\seq M$ of moderate growth is contained in some 
   Gevrey class, see \cite{Matsumoto84}).
	
	Therefore we can also extend the quasianalytic versions given in \cite{Fuerdoes1} of the 
   generalizations and applications of the analytic Holmgren theorem given by
	Bony \cite{MR0474426}, H\"ormander \cite{MR0320486}, Sj\"ostrand \cite{MR699623} and Zachmanoglou \cite{MR0299925}; 
	in fact, the assumption \eqref{R-FdB} guarantees that the classes are stable by solving ordinary 
	differential equations (with parameters), see \cite{RainerSchindl14}.

\subsection{Quasianalytic classes transversal to all Gevrey classes} \label{sec:QtransverseG}

We give here examples of quasianalytic classes that are not contained in $\cE^{\{\fG\}}$, 
but satisfy many of the regularity properties discussed before.
More precisely: 
\begin{enumerate}
   \item We will construct a quasianalytic strongly log-convex weight sequence $\seq Q$
which is derivation-closed and satisfies $q_k^{1/k} \to \infty$ such that $\cE^{\{\seq Q\}}\nsubseteq\cE^{\{\fG\}}$.
   \item We will show that $\omega_{\seq Q}$ is a weight function equivalent to a concave quasianalytic
 weight function and $\cE^{\{\omega_{\seq Q}\}}\nsubseteq\cE^{\{\fG\}}$.
\end{enumerate}
Note that $\seq Q$ cannot be of moderate growth (cf.\ \cite{Matsumoto84}).

We are going to define $\seq Q$ by $Q_0=1$ and
$Q_k=k!\prod_{j=1}^k\rho_j$
for $k\geq 1$ and a suitable sequence $\rho=(\rho_k)_k$ to be constructed.
In order to define $\rho$ accordingly we need three more auxiliary sequences
$(\alpha_j)_j$, $(\beta_j)_j\subseteq \N$ and $(\tau_j)_j\subseteq \R$ which will be chosen iteratively.
Let $\al_1 = \ta_1 = 1$. If $\al_j$ and $\ta_j$, $j\ge 1$, are already chosen, we pick $\beta_j\in\N$ such 
that
\begin{equation}\label{Appendix1}
\be_j \ge e^{\ta_j} \al_j,
\end{equation}
and set 
\begin{equation}\label{Appendix3}
\alpha_{j+1}:= \frac{(\beta_j-1)\beta_{j}}{2} \quad \text{ and }\quad  \ta_{j+1} := e^{\al_{j+1}}. 
\end{equation}
Clearly, $\al_1< \be_1<\al_2<\be_2 <\al_3 <  \cdots$.
We define
\[
   \rh_k:= \begin{cases}
\tau_j,  \quad\text{ if } k \in A_j:= \{k \in \N : \alpha_j\leq k< \beta_j\},\\
e^k,  \quad\text{ if } k \in B_j := \{k \in \N : \beta_j\leq k< \alpha_{j+1}\}.
\end{cases}
\]
By construction, $\rh_k$ is increasing and hence $\seq Q$ is strongly log-convex.
We also have $\rho_k \to \infty$ and hence $q_k^{1/k} \to \infty$, by the arguments in 
\cite[p.~104]{RainerSchindl12}. 
The sequence $\seq Q$ is derivation-closed, 
since $\rho_k\leq e^k$ for all $k$.

In order to see that $\seq Q$ is quasianalytic we have to show that 
\begin{equation}\label{Appendix4}
\sum_{k=1}^\infty\frac{1}{k\rho_k}\geq 
\sum_{j=0}^\infty\sum_{k\in A_j}\frac{1}{k\rho_k}
=\sum_{j=1}^\infty \frac{1}{\tau_j}\sum_{k\in A_j}\frac{1}{k}
\end{equation}
diverges.
Recall that, if $\gamma$ is the Euler constant, we have
\begin{equation*}
\sum_{k=1}^p\frac{1}{k}=\log p+\gamma+\ve_p
\end{equation*}
and $\ve_p\rightarrow 0$ for $p\rightarrow\infty$.
Thus, for $j \ge 2$,
\begin{equation*}
\sum_{k\in A_j}\frac{1}{k}
=\log\left(\frac{\beta_j-1}{\alpha_j-1}\right)+\ve_{\beta_j-1}-\ve_{\alpha_j-1}.
\end{equation*}
By \eqref{Appendix1}, 
$\log\left(\frac{\beta_j-1}{\alpha_j-1}\right)\geq \tau_j$,
for $j\geq 2$, which implies that 
\eqref{Appendix4} diverges.

Finally, we note that 
$\cE^{\{\seq Q\}}\subseteq\cE^{\{\fG\}}$ if and only if there exists $s> 0$
such that 
\begin{equation*}
\sup_{k\geq 1}\frac{\sqrt[k]{q_k}}{k^s}<\infty.
\end{equation*}
However, by \eqref{Appendix3},
\begin{align*}
q_{\alpha_{j+1}}^{1/\alpha_{j+1}}&\geq \Big(\prod_{k\in B_j}
\rho_k\Big)^{1/\alpha_{j+1}}=\exp\Big(\frac{1}{\alpha_{j+1}}\sum_{k=\beta_j}^{\alpha_{j+1}-1}k\Big)
= \exp\Big(\frac{\alpha_{j+1}-1}{2}-1\Big),
\end{align*}
and hence $q_{\alpha_{j+1}}^{1/\alpha_{j+1}}/ a_{j+1}^s$ is unbounded for all $s$.
This ends the proof of $(1)$.

The function $\omega_{\seq Q}(t)=\sup_k\log (t^k/Q_k)$ satisfies $\omega_{\seq Q}|_{[0,1]} =0$, since $Q_0=Q_1=1$. 
Furthermore, cf.\ \cite[Chapitre~I]{Mandelbrojt52}, $\omega_{\seq Q}$ is increasing and satisfies \eqref{om3} and \eqref{om4}.
The arguments in the proof of the implication $(4)\Rightarrow(5)$ 
in \cite[Lemma~12]{BMM07} show that also  $\eqref{om2}$ holds  
(in fact, $\omega_{\seq Q}(t)=o(t)$ as $t \to \infty$).
By \cite[Lemma~3.4]{JSS19}, $\omega_{\seq Q}$ is equivalent to a concave weight function. 
Hence $\omega_{\seq Q}$ is a weight function that is equivalent to a concave weight function.
By \cite[Lemma~4.1]{Komatsu73}, $\omega_{\seq Q}$ is quasianalytic, since $\seq Q$ is quasianalytic.
We have $\cE^{\{\omega_{\seq Q}\}}\nsubseteq\cE^{\{\fG\}}$, since    
$\cB^{\{\seq Q\}}(K)\subseteq\cB^{\{\omega_{\seq Q}\}} (K)$ for
any compact set $K \subseteq \R^n$.


\def\cprime{$'$}
\providecommand{\bysame}{\leavevmode\hbox to3em{\hrulefill}\thinspace}
\providecommand{\MR}{\relax\ifhmode\unskip\space\fi MR }
\providecommand{\MRhref}[2]{%
  \href{http://www.ams.org/mathscinet-getitem?mr=#1}{#2}
}
\providecommand{\href}[2]{#2}

\end{document}